\documentclass{amsart}

\input xy
\xyoption{all}

\usepackage{geometry}
\usepackage{amsmath}
\usepackage{amssymb}
\usepackage{cite}
\usepackage{amsthm}  
\usepackage{tikz-cd}
\usetikzlibrary{er,positioning}
\usepackage{subfiles}
\usepackage{thmtools}
\usepackage{thm-restate}
\usepackage[shortlabels]{enumitem}
\usepackage{cleveref}

\newtheorem{same}{This should never appear}[section]
\newtheorem{defin}[same]{Definition}

\newtheorem{remark}[same]{Remark}
\newtheorem{theorem}[same]{Theorem}
\newtheorem{example}[same]{Example}
\newtheorem{lemma}[same]{Lemma}
\newtheorem{fact}[same]{Fact}
\newtheorem{question}[same]{Question}
\newtheorem{cor}[same]{Corollary}
\newtheorem{prop}[same]{Proposition}

\newtheorem{hypothesis}[same]{Hypothesis}
\newtheorem{nota}[same]{Notation}

\newtheorem{probl}[same]{Problem}

\newtheorem{defin*}{Definition}
\newtheorem*{theorem*}{Theorem}
\newtheorem*{longlimtheorem}{Theorem \ref{largelimitsareisothm*}}
\newtheorem*{shortlimtheorem}{Theorem \ref{non-iso}}
\newtheorem*{completelimtheorem}{Theorem}

\makeatletter
\newcommand{\skipitems}[1]{%
  \addtocounter{\@enumctr}{#1}%
}

\newcommand{\rest}{\mathord{\upharpoonright}}
\newcommand{\id}{\textrm{id}}

\newcommand{\K}{\mathbf{K}}

\newcommand{\Kto}{\K^{\text{Tor}}}

\newcommand{\Km}{\K^{R\text{-Mod}}}
\newcommand{\LS}{\operatorname{LS}}

\newcommand{\calt}{\mathcal{T}}
\newcommand{\lesst}{\vartriangleleft}
\newcommand{\stab}{\operatorname{Stab}}
\newcommand{\splt}{\operatorname{split}}
\newcommand{\Kkappalims}{{\K_{(\lambda, \geq \kappa)}}}
\newcommand{\kkappalims}{{K_{(\lambda, \geq \kappa)}}}
\newcommand{\Kmupluslims}{{\K_{(\lambda, \geq \mu^+)}}}

\newcommand{\kappadnfu}{{\kappa(\dnf, \K_\lambda, \lea^u)}}
\newcommand{\chidnfu}{{\chi(\dnf, \K, \lea^u)}}
\newcommand{\chidnfum}{{\chi(\dnf, \K^{R\text{-Mod}}, \leq_{\K^{R\text{-Mod}}}^u)}}
\newcommand{\chidnfut}{{\chi(\dnf, \K^{\text{Tor}}, \leq_{\K^{\text{Tor}}}^u)}}
\newcommand{\shortkappadnfu}{{\kappa_\lambda^u(\dnf)}}

\newcommand{\leap}[1]{\le_{#1}}

\newcommand{\lea}{\leap{\K}}

\newcommand{\mon}{\mathfrak{C}}

\newcommand{\gtp}{\mathbf{gtp}}
\newcommand{\gS}{\mathbf{gS}}

\DeclareMathOperator{\cof}{cf}    

\newbox\noforkbox \newdimen\forklinewidth
\setbox0\hbox{$\textstyle\smile$}
\setbox1\hbox to \wd0{\hfil\vrule width \forklinewidth depth-2pt
 height 10pt \hfil}
\setbox\noforkbox\hbox{\lower 2pt\box1\lower 2pt\box0\relax}
\def\unionstick{\mathop{\copy\noforkbox}\limits}

\def\nonfork_#1{\unionstick_{\textstyle #1}}

\setbox0\hbox{$\textstyle\smile$}
\setbox1\hbox to \wd0{\hfil{\sl /\/}\hfil}
\setbox2\hbox to \wd0{\hfil\vrule height 10pt depth -2pt width
              \forklinewidth\hfil}
\newbox\doesforkbox
\setbox\doesforkbox\hbox{\lower 2pt\box1 \lower 2pt\box2\lower2pt\box0\relax}
\def\nunionstick{\mathop{\copy\doesforkbox}\limits}

\def\fork_#1{\nunionstick_{\textstyle #1}}

\newcommand{\dnf}{\unionstick}

\newcommand{\nf}{\unionstick}
\newcommand{\dnfb}[4]{#2 \overset{#4}{\underset{#1}{\overline{\nf}}} #3}

\title{On the spectrum of limit models}
\date{AMS 2020 Subject Classification:
Primary:  03C48. Secondary: 03C45.} 

\author{Jeremy Beard}
\email{jbeard@andrew.cmu.edu}
\urladdr{https://jeremypbeard.github.io}
\address{Department of Mathematical Sciences \\ Carnegie Mellon University \\ Pittsburgh, Pennsylvania, USA}

\author{Marcos Mazari-Armida}
\email{marcos\_mazari@baylor.edu}
\urladdr{https://sites.baylor.edu/marcos\_mazari/}
\address{Department of Mathematics \\ Baylor University \\ Waco, Texas, USA}
\thanks{The second author's research was partially supported by an NSF grant DMS-2348881, a Simons Foundation grant MPS-TSM-00007597 and an AMS-Simons Travel Grant 2022--2024.}

\begin{document}

\begin{abstract}

We study the spectrum of limit models assuming the existence of a nicely behaved independence notion. Under reasonable assumptions, we show that all `long'  limit models are isomorphic, and all `short' limit models are non-isomorphic. 

\begin{completelimtheorem}
    Let $\K$ be a $\aleph_0$-tame abstract elementary class stable in $\lambda \geq \LS(\K)$ with amalgamation, joint embedding and no maximal models. Let $\kappa < \lambda^+$ be regular. Suppose $\dnf$ is an independence relation on the models of size $\lambda$ that satisfies uniqueness, extension, non-forking amalgamation, universal continuity, and $(\geq \kappa)$-local character.

    Suppose $\delta_1, \delta_2 < \lambda^+$ with $\cof(\delta_1) < \cof(\delta_2)$. Then for any $N_1, N_2, M \in \K_\lambda$ where $N_l$ is a $(\lambda, \delta_l)$-limit model over $M$ for $l = 1, 2$, 
    \[N_1 \text{ is isomorphic to } N_2 \text{ over } M  \iff \cof(\delta_1) \geq \kappadnfu\]

\end{completelimtheorem}

Both implications in the conclusion have improvements. High cofinality limits are isomorphic without the $\aleph_0$-tameness assumption and assuming $\dnf$ is defined only on high cofinality limit models. Low cofinality limits are non-isomorphic without assuming non-forking amalgamation.

We show how our results can be used to study limit models in both abstract settings and in natural examples of abstract elementary classes.

\end{abstract}


\maketitle

\tableofcontents

\section{Introduction} 

Limit models (see Definition \ref{univ-limit-model-def}), originally introduced by Kolmann and Shelah as a surrogate for saturated models \cite{kosh}, have proved to be a key notion for extending the classification theory of first order model theory to abstract elementary classes (AECs). In particular, the theory developed around them has been used to prove various approximations of the main test question for AECs, \emph{Shelah's categoricity conjecture} \cite{sh87a}, \cite[6.13(3)]{sh704}, which proposes that if an AEC is categorical in some large enough cardinal, then it is categorical in all large enough cardinals \cite{sh394}, \cite{shelahaecbook}, \cite{bgvv16}, \cite{grvan06b}, \cite{grvan06c}, \cite{vasey18a}, \cite{vasey18}, \cite{shvas}.

The required framework to study limit models has been weakened over time, from assuming categoricity in some higher cardinal \cite{shvi}, to $\lambda$-superstability and $\lambda$-symmetry \cite{vand}, and now even to very nice AECs strictly stable in $\lambda$ \cite{bovan}. 

For now, suppose $\K$ is a nice AEC (i.e. has a monster model), and is stable in $\lambda \geq \LS(\K)$. In this setting, all $\lambda$-limit models exist (see Fact \ref{limitsexist*}). Since limit models play the role of saturated models to some extent, a key question that has been thoroughly studied is: 

\begin{question}\label{same-base-lim-question}
    Suppose $\delta_1, \delta_2 < \lambda^+$ are limit ordinals. Suppose $N_l$ is a $(\lambda, \delta_l)$-limit model over $M$ for $l = 1, 2$. Is $N_1$ isomorphic to $N_2$ over $M$?
\end{question}

Positive answers to Question \ref{same-base-lim-question} play an essential role in understanding superstability and stability in the context of AECs. Whether Question \ref{same-base-lim-question} answers positively for all $\delta_1, \delta_2$ seems to be a natural dividing line; in  complete first order theories, or more generally tame AECs, this is equivalent to superstability \cite{grva}. Question \ref{same-base-lim-question} has also helped us better understand independence notions, see e.g. Lemma \ref{disjointness-of-dnf}, \cite[III.1.21]{shelahaecbook}, \cite[2.8]{vasey18}.

For nice AECs, Question \ref{same-base-lim-question} always has a positive answer when $\cof(\delta_1) = \cof(\delta_2)$ by a back and forth argument (see Fact \ref{cofinalityiso*}). But in many settings, much more is true. In \cite[3.3.7]{shvi}, Shelah and Villaveces state that assuming a local character condition of $\lambda$-splitting and categoricity in some $\lambda' > \lambda$ with $\lambda' \geq \beth_{(2^{LS(\K)})^+}$ (and still assuming the existence of a monster model for simplicity), Question \ref{same-base-lim-question} has a positive answer for all limit ordinals $\delta_1, \delta_2 < \lambda^+$. VanDieren found major gaps in \cite{shvi} and over the course of several papers, \cite{van02}, \cite{van06}, \cite{van13}, \cite{grvavi}, \cite{van16a}, and \cite{vand}, VanDieren obtained a correct proof of \cite[3.3.7]{shvi} (including fixing another gap found by Tapani Hyttinen, see \cite{van13}) and extended the arguments to AECs with $\lambda$-superstability and $\lambda$-symmetry (both of which follow from categoricity in any $\lambda' > \lambda$ \cite[4.8]{vasey17b}).

In a strictly stable setting however, it is possible to have non-isomorphic limit models (even in the first order setting, see \cite[6.1]{grvavi}). This naturally raises the question of what the structure of different isomorphism types of limit models look like in a more general setting. 

In \cite{bovan}, Boney and VanDieren answer part of this question - they show that assuming $\K$ is stable in $\lambda$ (but possibly not $\lambda$-superstable), and that non-splitting satisfies universal continuity, symmetry, and local character above some $\kappa$, all the limit models of regular length at least $\kappa$ are isomorphic. So, `long limit models are isomorphic' in this setting.

On the other hand, little has been said about the behavior of `short' limit models. In \cite[6.1]{grvavi}, Grossberg, VanDieren and Villaveces noted that in the first order setting, when strictly stable in $\lambda$, the $(\lambda, \aleph_0)$-limit model is not isomorphic to the $(\lambda, \mu)$-limit model for any regular cardinal $\mu \geq \kappa(T)$.

In this paper, we find generalisations for both sides of this picture: criteria on an independence relation on long limit models that implies uniqueness of long limit models (Theorem \ref{largelimitsareisothm*}), and criteria on an AEC that ensure the short limits are all distinct (Theorem \ref{non-iso}). Both of these theorems are  local in that they only use information about the models of cardinality $\lambda$. 

\begin{longlimtheorem}
    Let $\K$ be an AEC stable in $\lambda \geq \LS(\K)$, with AP, JEP, and NMM in $\K_\lambda$. Let $\kappa < \lambda^+$ be a regular cardinal. Let $\dnf$ be an independence relation on the $(\lambda, \geq \kappa)$-limit models of $\K$ that satisfies uniqueness, extension, non-forking amalgamation, $(\geq \kappa)$-local character, and $\Kkappalims$-universal continuity* in $\K$.
    
    Let $\delta_1, \delta_2 < \lambda^+$ be limit ordinals where $\kappa \leq \cof(\delta_1), \cof(\delta_2)$.
    If $M, N_1, N_2 \in \K_\lambda$ where $N_l$ is a $(\lambda, \delta_l)$-limit over $M$ for $l=1, 2$, then there is an isomorphism from $N_1$ to $N_2$ fixing $M$.

    Moreover, if $N_1, N_2 \in \K_\lambda$ where $N_l$ is $(\lambda, \delta_l)$-limit for $l=1, 2$, then $N_1$ is isomorphic to $N_2$.
\end{longlimtheorem}

The proof of Theorem \ref{largelimitsareisothm*} relies on a notion called \emph{towers}. Towers are increasing sequences of models with independence recorded along each level of the tower. They were originally introduced by Shelah and Villaveces in \cite{shvi} in a form involving two systems of models and a list of singletons, and captured independence of the singletons with $\lambda$-splitting, which \cite{van02}, \cite{van06}, \cite{van13}, \cite{grvavi}, \cite{van16a}, \cite{vand}, and \cite{bovan} also follow. Later, Vasey found a simplified presentation in the $\lambda$-superstable $\lambda$-symmetric case using a modified notion of tower \cite{vasey18}, which we adapt in this paper. Vasey's presentation used only a single system of models with a system of singletons, and used $\lambda$-non-forking rather than $\lambda$-non-splitting, which satisfies stronger properties (in particular, uniqueness, extension, and transitivity on limit models - $\lambda$-non-splitting is only known to satisfy weaker forms of these, see Fact \ref{uniqueness-split}).

 Our argument is similar to that of  Vasey \cite{vasey18}, but  we assume our non-forking relation $\dnf$ is defined only on $(\lambda, \geq \kappa)$-limit models, and that $\dnf$ satisfies weaker forms of local character and continuity. Our proofs differ when these weakened forms are applied, and we must also take care to ensure all the models of our towers are $(\lambda, \geq \kappa)$-limit models. For example, since our notion of towers is not closed under unions of low cofinality chains, we adapt Vasey's tower extension lemma \cite[16.17]{vasey18} to allow us to find extensions of any chain of towers, rather than just a single tower (see Proposition \ref{towerextensionprop*}). Similar tweaks and workarounds appear in Lemma \ref{high-tower-unions-exist}, Lemma \ref{reducedextensionsexist*}, Proposition \ref{reducedimplieshighcontinuity*}, and Lemma \ref{highfullunionsarefull*}. Where the proofs differ less from \cite{vasey18} we direct the reader to the original proofs in \cite{vaseyn} and \cite{vasey18}, and to the previous version of this paper \cite{bemavone}, which included the details.

Theorem \ref{largelimitsareisothm*} is similar to \cite{bovan} as they show high cofinality limit models are isomorphic in a different context. They use $\lambda$-non-splitting and towers similar to \cite{shvi} as mentioned above, rather than an arbitrary independence relation with forking-like properties. Our assumptions on $\dnf$ are stronger than the known properties of $\lambda$-non-splitting in their context (our relation has full extension, uniqueness, and transitivity, rather than the weaker forms $\lambda$-non-splitting is known to satisfy). However, under their assumptions, another independence relation, $\lambda$-non-forking, can be defined, which satisfies all the assumptions of Theorem \ref{largelimitsareisothm*} provided it is  assumed that $\lambda$-non-forking satisfies uniqueness (\cite{leu2} attempts to prove uniqueness in the context of \cite{bovan} - see Remark \ref{leung-uniqueness-necessary}). Thus if uniqueness can be proved from the other assumptions, Theorem \ref{largelimitsareisothm*} would imply \cite{bovan}.

In a related setting, combining Theorem \ref{largelimitsareisothm*} with \cite[\textsection 4, \textsection 5]{vasey16b}, we show that in nice $\mu$-tame AECs where $\mu \geq \LS(\K)$ with $\lambda \geq \mu^{+}$ a stability cardinal, assuming universal continuity of $\mu$-non-splitting and symmetry of $(\geq \mu)$-non-forking in $\Kmupluslims$, all $(\lambda, \geq \mu^{+})$-limit models are isomorphic (see Corollary \ref{vasey-tame-saturated-dnf-cor}). The $\mu$-non-forking relation is only known to behave well over $\mu^+$-saturated models, so in particular over $(\lambda, \geq \mu^{+})$-limit models. This and Example \ref{dnf-examples-long-lim}(3) use the full strength of our assumption that $\dnf$ need only be defined over high cofinality limit models.

Now we move on to the `short' limit side of the picture.

\begin{shortlimtheorem} Let $\K$ be an $\aleph_0$-tame AEC stable in $\lambda \geq \LS(\K)$, with AP, JEP, and NMM in $\K_\lambda$. Let $\kappa < \lambda^+$ be a regular cardinal. Let $\dnf$ be an independence relation on $\K_\lambda$ that satisfies uniqueness, extension, universal continuity, and ($\geq \kappa$)-local character.

    If $\operatorname{cf}(\delta_1) < \kappadnfu$ and  $\operatorname{cf}(\delta_1) \neq \operatorname{cf}(\delta_2)$, then the $(\lambda, \delta_1)$-limit model is not isomorphic to the $(\lambda, \delta_2)$-limit model.
\end{shortlimtheorem}

The proof of Theorem \ref{non-iso} is surprisingly short. We show that for regular $\delta < \lambda^+$, if the $(\lambda, \delta)$-limit model $M$ is $\delta^+$-saturated (which a $(\lambda, \delta')$-limit model will be for all regular $\delta' > \delta$), then $\delta \geq \kappadnfu$ (see Lemma \ref{key-no}). This uses the argument of \cite[4.6]{sebastien-successive-categ}, which allows us to prove $\delta$-local character of $\lambda$-non-splitting in the above context, and that non-splitting is `close' to $\dnf$-non-forking (see Lemma \ref{split-fork}). In Theorem \ref{non-iso} we assume that $\K$ is \emph{$\aleph_0$-tame} - that is, types are equal if their restrictions to countable subsets are equal (see Definition \ref{tameness-def}). A problem left open is whether tameness can be avoided in Theorem \ref{non-iso}.

A key technical step to prove Theorem \ref{non-iso} is to determine the relationship between $\lambda$-non-splitting, $\lambda$-non-forking and $\dnf$-non-forking. Due to this, we spend Subsection \ref{dnf-is-like-dns} studying how they interact. Among the results we obtain is a canonicity result for $\lambda$-non-forking for long limit models (see Theorem \ref{nf-limits}).

In Section 5, we present results which combine the results of Section 3 and 4 in a natural setting. Using both Theorem \ref{largelimitsareisothm*} and Theorem \ref{non-iso}, we show that for distinct regular cardinals $\delta_1, \delta_2 \leq \lambda$, the $(\lambda, \delta_1)$-limit model and $(\lambda, \delta_2)$-limit model are isomorphic exactly when $\delta_1, \delta_2 \geq \kappadnfu$ (see Theorem \ref{full-picture}). Since $\kappadnfu$ eventually stabilises, we get that for all high enough $\lambda$ (above $\beth_{\beth_{2^{\LS(\K)}+}}$), the isomorphism spectra of the $\lambda$-limit models are all the same (see Theorem \ref{cor-2} and Theorem \ref{cor-2.5}). These results have the advantage that there are many examples of natural AECs that satisfy their hypotheses, and they can be used as black boxes when studying limit models in these cases; we demonstrate this in Section 6. 

More precisely in Section 6, we showcase how to use these results to study limit models on the AECs of: modules, with embeddings; torsion abelian groups, with pure embeddings; and models of a complete first-order theory, with elementary embeddings. In particular, we derive as an immediate corollary a slight weakening of the main theorem of \cite{maz24}: that in the AEC of $R$-modules with embeddings, $R$ is $(<\aleph_n)$-Noetherian but not $(< \aleph_{n-1})$-Noetherian if and only if for all stability cardinals $\lambda \geq \beth_{\beth_{(2^{\operatorname{card}(R) + \aleph_0})^+}}$, there are exactly $n+1$ non-isomorphic $\lambda$-limit models (this result is slightly weaker than \cite[3.17]{maz24} as the lower bound on $\lambda$ is higher than in the original result).

The paper has six sections. Section 2 presents necessary background, some basic results and some examples (though many of the examples are not necessary to understand the main results). Section 3  deals with high cofinality limit models, and ends with the application to AECs with $\mu$-tameness. Section 4 addresses low cofinality limit models. Section 5 presents `general' results that give the full spectrum of limit models, including the `black box' versions of the main theorem that can be applied for all large enough stable $\lambda$. Section 6 presents some applications of our results in natural AECs.

This paper was written while the first author was working on a Ph.D. thesis under the direction of Rami Grossberg at Carnegie Mellon University, and the first author would like to thank Professor Grossberg for his guidance and assistance in his research in general and in this work specifically. We would also like to thank Sebastien Vasey for sharing with us key ideas that play an important role in this paper (specifically in Subsection 3.4 and Subsection 4.2). We thank John Baldwin for comments that help improve the paper. We also thank the anonymous referee for detailed comments that helped improve the presentation of the paper.

\section{Preliminaries and basic results}

We assume some basic knowledge of abstract elementary classes (AECs), such as presented in \cite{baldwinbook09}.

\subsection{Basic notions} Typically we use $\K = (K, \lea)$ to denote an AEC, or sometimes an abstract class (an abstract class, or AC, is a class of models $K$ in a fixed language closed under isomorphisms with a partial ordering $\lea$ that respects isomorphisms). When we say $\K'$ is a sub-AC of $\K$, we mean $\K'$ is an AC, $K' \subseteq K$, and $\leq_{\K'} = \lea \upharpoonright (K')^2$. We use $M, N$ for models, $a, b$ for elements of models (always singletons), $|M|$ to denote the universe of model $M$, and $\|M\|$ for the cardinality of the universe of $M$. By $M \cong N$, we mean that $M$ and $N$ are isomorphic, and by $M \underset{M_0}{\cong} N$, that $M$ and $N$ are isomorphic over $M_0$; that is, there is an isomorphism from $M$ to $N$ that fixes the substructure $M_0$ of $M$ and $N$.

We use standard abuses of notation such as $b \in M$ as shorthand for $b \in |M|$. Given $M, N \in \K$ where $M \lea N$ and $b \in N$, $\gtp(b/M, N)$ denotes the (Galois) type of the singleton $b$ over $M$ in $N$. $\gS(M)$ is the set of types of singletons over $M$. Occasionally we use the more general notions of $\gtp(a/A, N)$ and $\gS(A; N)$, for $A \subseteq |N|$ where $N \in \K$, see for example \cite[2.16, 2.20]{vasey16c}. Given types $p \in \gS(M)$ and $q \in \gS(N)$ with $M \lea N$, we say $q$ extends $p$ or write $p \subseteq q$ if every realisation of $q$ realises $p$ (or equivalently, there exists a common realisation), defined similarly for $p \in \gS(A; N)$, $q \in \gS(B; N)$ where $A \subseteq B$.

We use $\alpha, \beta, \gamma, \delta$ for ordinals (typically $\gamma, \delta$ are limit), and $\lambda, \kappa, \mu$ for infinite cardinals (unless explicitly mentioned, cardinals will be infinite) - typically $\K$ will be stable in $\lambda$, and $\kappa$ will be a regular cardinal with $\kappa < \lambda^+$. We usually use $\kappa$ as a local character cardinal, as this is related to  Shelah's $\kappa_r(T)$ for first order stable theories (see subsection \ref{beyond-modules}), and it mixes well with Vasey's notation in \cite{vaseyt} (see Definition \ref{kappa-defs}).

We also use the following standard notions. Let $\K$ be an AC. We say $\K$ satisfies the \emph{amalgamation property} or \emph{AP} if every span of models $M_1, M_2$ over $M_0$ in $\K$ may be amalgamated into some $N \in \K$. We say $\K$ satisfies the \emph{joint embedding property} or \emph{JEP} if any pair of models $M_1, M_2 \in \K$ can be embedded into some common model $N \in \K$. We say $\K$ has \emph{no maximal models} or \emph{NMM} if for every $M \in \K$, there exists $N \in \K$ where $M \lea N$ and $M \neq N$. 

Given a cardinal $\lambda$, $\K_\lambda$ denotes the sub-AC with underlying class $K_\lambda = \{M \in K : \|M\| = \lambda\}$, and with $\leq_{\K_\lambda}$ the restriction of $\lea$ to $K_\lambda$. 

Tameness, identified by Grossberg and VanDieren in \cite{tamenessone}, has proved to be a vital notion in the classification theory of AECs. 

\begin{defin}\label{tameness-def}
    Let $\K$ be an AEC and $\theta$ be a cardinal. We say $\K$ is \emph{$(<\theta)$-tame} if for every $M \in \K$, and every $p, q \in \gS(M)$, if $p \upharpoonright A = q \upharpoonright A$ for all $A \subseteq |M|$ with $|A| < \theta$, then $p = q$.

    We say $\K$ is $\theta$-tame if $\K$ is $(<\theta^+)$-tame.
\end{defin}

Note that if  $LS(\K) < \theta$ then the definition of $(<\theta)$-tame is equivalent when we replace $A \subseteq M$ with $M' \lea N$. Also,  if $\K$ is $(<\theta_1)$-tame and $\theta_1 < \theta_2$ then $\K$ is $(<\theta_2)$-tame. It is worth emphasizing that several of our results do not assume tameness. 

\subsection{Limit Models}

We begin with a brief recounting of facts about limit models

\begin{defin}
Let $\K$ be an AEC and $\lambda \geq \LS(\K)$.
    \begin{enumerate}
        \item Let $M_0, M \in \K_\lambda$. We say $M$ is \emph{universal over $M_0$} if for every $N \in \K_\lambda$ with $M_0 \lea N$, there exists a $\K$-embedding $f:N \underset{M_0}{\rightarrow} M$. We sometimes write $M_0 \lea^u M$.

        \item Let $\alpha$ be an ordinal, and let $\langle M_i : i < \alpha \rangle$ be a $\lea$-increasing sequence of models in $\K$ (sometimes called a \emph{chain}). 
        \begin{enumerate}
            \item For $\gamma < \alpha$ limit, we say $\langle M_i : i < \alpha \rangle$ is \emph{continuous at $\gamma$} if $M_\gamma = \bigcup_{i < \gamma} M_i$.
        
            \item We say $\langle M_i : i < \alpha \rangle$ is \emph{continuous} if it is continuous at every limit $\gamma < \alpha$
        
            \item We say $\langle M_i : i <\alpha \rangle$ is \emph{universal} if for all $i < \alpha$ with $i + 1 <\alpha$, $M_i \lea^u M_{i+1}$

            \item We say $\langle M_i : i <\alpha \rangle$ is \emph{strongly universal} if for all $i \in [0, \alpha)$, $\bigcup_{r < i} M_r \lea^u M_i$.
        \end{enumerate}
    \end{enumerate}
\end{defin}

\begin{remark}\label{cts-univ-str-defs-generalise-remark}
    We stated these definitions in terms of sequences indexed by ordinals for legibility, but they may be used also when $\alpha$ is replaced by an arbitrary well-ordering $(I, <)$ and the sequence by $\langle M_i : i \in I \rangle$. Also, if $\alpha = \delta + 1$, we may write $\langle M_i : i \leq \delta \rangle$ in place of $\langle M_i : i < \delta + 1 \rangle$.
\end{remark}

\begin{defin}\label{univ-limit-model-def}
Let $\K$ be an AEC and $\lambda \geq \LS(\K)$.
    \begin{enumerate}
        
        \item Let $M, N \in \K_\lambda$ and $\delta < \lambda^+$ a limit ordinal. We say $N$ is a \emph{$(\lambda, \delta)$-limit over $M$} if there is a $\lea$-increasing universal continuous chain $\langle M_i : i \leq \delta \rangle$ such that $M_0 = M$ and $M_\delta = N$ (in particular, $N = \bigcup_{i < \delta} M_i$).
        
        We say the sequence $\langle M_i : i \leq \delta \rangle$ is \emph{witnessing} the limit, and that $\delta$ is the \emph{length} of the limit.

        \item $N \in \K_\lambda$ is a $(\lambda, \delta)$-limit model if it is a $(\lambda, \delta)$-limit model over some $M \in \K_\lambda$. 
        
        \item $N \in \K_\lambda$ is a $\lambda$-limit model (over $M$) if it is a $(\lambda, \delta)$-limit model (over $M$) for some limit ordinal $\delta < \lambda^+$. We sometimes omit $\lambda$ when it is clear from context.

        \item Let $\kappa< \lambda^+$ be regular. We say $N \in \K_\lambda$ is a $(\lambda, \geq \kappa)$-limit model if $N$ is a $(\lambda, \delta)$-limit model for some regular $\delta \in [\kappa, \lambda^+)$.

        \item $\kkappalims$ is the class of all $(\lambda, \geq \kappa)$-limit models in $\K$. $\Kkappalims$ is the AC of $\K$ restricted to $\kkappalims$, that is, $(\kkappalims, \lea\upharpoonright (\kkappalims)^2)$.
    \end{enumerate}
\end{defin}

Note in particular that if $N$ is a $\lambda$-limit over $M$, then $N$ is also universal over $M$. Also, the definition of $(\lambda, \delta)$-limit model is equivalent if we only assume the sequence $\langle M_i : i \leq \delta \rangle$ is continuous at $\delta$.

A lot is already known about when limit models exist and when they are isomorphic. The following are essentially {\cite[II.1.16.1(a)]{shelahaecbook}}, but we cite \cite{tamenessone} as this provides a proof.

\begin{fact}[{\cite[2.12]{tamenessone}}]\label{universalchriterionfact*}
    Let $\K$ be an AEC and $\lambda \geq \LS(\K)$. If $\langle M_i : i \leq \lambda \rangle$ is a $\lea$-increasing sequence in $\K_\lambda$ continuous at $\lambda$ such that $M_{i+1}$ realises all types over $M_i$ for every $i < \lambda$, then $M_\lambda$ is universal over $M_0$.
\end{fact}

A corollary of this is:

\begin{fact}[{\cite[2.9]{tamenessone}}]\label{limitsexist*}
    Let $\K$ be an AEC and $\lambda \geq \LS(\K)$ such that $\K$ is stable in $\lambda$ and  $\K_\lambda$ has AP and NMM. Then for every $M \in \K_\lambda$, there is $N \in \K_\lambda$ universal over $M$.

    Moreover, under the same assumptions, for every $M \in \K_\lambda$ and every limit $\delta < \lambda^+$ there is a $(\lambda, \delta)$-limit model over $M$.
\end{fact}

The following holds by a straightforward back and forth argument, first present in \cite{sh394}. 

\begin{fact}[{\cite[1.3.6]{shvi}}]\label{cofinalityiso*}
    Let $\K$ be an AEC and $\lambda \geq \LS(\K)$. Suppose $\K_\lambda$ has AP, and that $\delta_1, \delta_2 < \lambda^+$ such that $\cof(\delta_1) = \cof(\delta_2)$. Suppose $M, N_1, N_2 \in \K_\lambda$ where $N_l$ is a $(\lambda, \delta_l)$-limit over $M$ for $l=1, 2$. Then there is an isomorphism $f:N_1 \cong N_2$ fixing $M$.

    If in addition $\K_\lambda$ has JEP, then for any $N_1, N_2 \in \K_\lambda$ where $N_l$ is a $(\lambda, \delta_l)$-limit for $l=1, 2$, $N_1 \cong N_2$.
\end{fact}

So if $\K_\lambda$ is well behaved, all possible limit models exist, and by Fact \ref{cofinalityiso*} they are unique for any fixed cofinality of the limit's length. This means we can restrict to studying limits of infinite regular lengths.

\begin{defin}
    Let $\K$ be an AEC, $\mu$ an infinite cardinal. We say that a model $M \in \K$ is \emph{$\mu$-saturated} if for all $A \subseteq M$ with $|A| < \mu$, and $N \in \K$ with $M \lea N$ and $p \in \gS(A;N)$, $p$ is realised in $M$. We say $M$ is \emph{saturated} if $M$ is $\|M\|$-saturated.
\end{defin}

\begin{remark}
    If $\mu > \LS(\K)$, then $M$ is $\mu$-saturated if and only if for all $M_0 \lea M$ with $\|M_0\| < \mu$, and all $p \in \gS(M_0)$, $p$ is realised in $M$.
\end{remark}

The following is \cite[2.8(1)]{grva}, but since they do not include a proof we provide one.

\begin{fact}\label{long-limits-are-saturated}
Suppose $\K$ is an AEC with AP in $\K_\lambda$, $\lambda \geq \LS(\K)$, $\delta <\lambda^+$ is regular, and $M$ is a $(\lambda, \delta)$-limit model. Then $M$ is $\delta$-saturated.
\end{fact}

\begin{proof}
    Fix $\langle M_i : i < \delta \rangle$ witnessing that $M$ is a $(\lambda, \delta)$-limit model. Suppose $A \subseteq |M|$ with $|A| < \delta$, $N \in \K$ with $M \lea N$, and $p \in \gS(A;N)$. Say $p = \gtp(a/A, N)$. Take $i < \delta$ such that $A \subseteq M_i$. Then $p \subseteq \gtp(a/M_i, N)$. Since $M_i \lea^u M_{i+1} \lea M$, $\gtp(a/M_i, N)$ is realised in $M$, and hence $p$ is also realised in $M$.
\end{proof}

\subsection{Independence relations} Independence relations generalize first order non-forking to AECs. These have been thoroughly studied in recent years. 

\begin{defin}\label{dnf-def}
    Given an abstract class $\K$, a \emph{weak independence relation on $\K$} is a relation $\dnf$ on tuples $(M_0, a, M, N)$, where $M_0, M, N \in \K$, $a \in N$, and $M_0 \lea M \lea N$. We write $a \dnf_{M_0}^{N} M$ as a shorthand for $\dnf(M_0, a, M, N)$.
\end{defin}

\begin{defin}
 An \emph{independence relation $\dnf$} on an abstract class $\K$ is a weak  independence relation  that satisfies:
    \begin{enumerate}
        \item \emph{Invariance:} whenever $a \dnf_{M_0}^{N} M$ holds, and $f:N \cong N'$ is an isomorphism, we have that $f(a) \dnf_{f[M_0]}^{N'} {f[M]}$.
        \item \emph{Monotonicity:} whenever $a \dnf_{M_0}^{N} M$ holds and $M_0 \lea M_1 \lea M \lea N_0 \lea N \lea N'$ and $a \in N_0$, we have both $a \dnf_{M_0}^{N_0} M_1$ and $a \dnf_{M_0}^{N'} M_1$ (that is, we can shrink $M$, and we can shrink or grow $N$, so long as we preserve $M_0 \lea M \lea N$, $a \in N$).
        \item \emph{Base monotonicity:} whenever $a \dnf_{M_0}^{N} M$ holds and $M_0 \lea M_1 \lea M$, then $a \dnf_{M_1}^{N} M$.
    \end{enumerate}
\end{defin}

\begin{defin}
    Suppose $\dnf$ is an independence relation over an abstract class $\K$. Let $M_0 \lea M$ and $p \in \gS(M)$. We say $p \dnf$-does not fork over $M_0$ if there exist $N \in \K$ and $a \in N$ such that $p = \gtp(a/M, N)$ and $a \dnf_{M_0}^N M$.
\end{defin}

\begin{remark}
    By invariance and monotonicity, for every $p \in \gS(M)$ and $M_0 \lea M$, $p$ $\dnf$-does not fork over $M_0$ if and only if for \emph{every} $N \in \K$, $a \in N$ such that $p = \gtp(a/M, N)$, $a \dnf_{M_0}^N M$. That is, the choice of representatives of $p$ do not matter.
\end{remark}

Nicely behaved independence relations typically satisfy a number of the following properties. Our results on isomorphism types of limit models will involve assuming existence of an independence relation satisfying different combinations of these (and some additional ones we will define later).

\begin{defin}\label{independence-relation-properties}
    Given an independence relation $\dnf$ on an abstract class $\K$, $\K' \subseteq \K$ an abstract subclass, $\kappa$ a regular cardinal, and $\theta$ any infinite cardinal, we say $\dnf$ satisfies:
    \begin{enumerate}
        \item \emph{Uniqueness} if whenever $M \lea N$, and $q_1, q_2 \in \gS(N)$ satisfy that $q_1 \upharpoonright M = q_2 \upharpoonright M$ and $q_l$ $\dnf$-does not fork over $M$ for $l = 1, 2$, then $q_1 = q_2$.
        \item \emph{Extension} if whenever $M_0 \lea M \lea N$, and $p \in \gS(M)$ $\dnf$-does not fork over $M_0$, then there exists $q \in \gS(N)$ extending $p$ such that $q$ $\dnf$-does not fork over $M_0$.
        \item \emph{Transitivity} if whenever $M_0 \lea M \lea N$ and $p \in \gS(N)$ satisfies both that $p$ $\dnf$-does not fork over $M$ and $p \upharpoonright M$ $\dnf$-does not fork over $M_0$, then $p$ $\dnf$-does not fork over $M_0$.
        \item \emph{Existence} if whenever $M \in \K$ and $p \in \gS(M)$, $p$ $\dnf$-does not fork over $M$.
        \item \emph{$\kappa$-local character} if whenever $\langle M_i : i < \kappa \rangle$ is a $\lea^u$-increasing sequence, and $p \in \gS(M_\kappa)$ where $M_\kappa = \bigcup_{i<\kappa} M_i \in \K$, then there exists $i < \kappa$ such that $p$ $\dnf$-does not fork over $M_i$.
        \item \emph{$(\geq \kappa)$-local character} if $\dnf$ satisfies $\gamma$-local character for each regular $\gamma \geq \kappa$.
        \item \emph{strong $(<\kappa)$-local character} if for all $N \in K$ and $p \in \gS(N)$, there exists $M \lea N$ with $\|M\| < \kappa$ and $p$ $\dnf$-does not fork over $M$.
        \item \emph{strong $\kappa$-local character} if $\dnf$ satisfies strong $(<\kappa^+)$-local character.
        \item \emph{$\kappa$-universal continuity} if whenever $\langle M_i : i < \kappa \rangle$ is a $\lea^u$-increasing sequence and $p \in \gS(M_\kappa)$ where $M_\kappa = \bigcup_{i<\kappa} M_i \in \K$, then provided that $p \upharpoonright M_i$ $\dnf$-does not fork over $M_0$ for all $i < \kappa$, $p$ $\dnf$-does not fork over $M_0$.
        \item \emph{$(\geq \kappa)$-universal continuity} if $\dnf$ satisfies $\gamma$-universal continuity for each regular $\gamma \geq \kappa$.
        \item \emph{Universal continuity} if $\dnf$ has $(\geq \aleph_0)$-universal continuity.
        \item \emph{Non-forking amalgamation} if given $M_0, M_1, M_2 \in K$ and $a_1 \in M_1, a_2 \in M_2$, such that $M_0 \lea M_l$ for $l = 1, 2$, then there exists $N \in \K$ with $M_0 \lea N$ and $\K$-embeddings $f_l: M_l \underset{M_0}{\rightarrow} N$ such that $f_l(a_l) \dnf_{M_0}^N f_{3-l}[M_{3-l}]$ for $l = 1, 2$ (that is, you can amalgamate $M_1, M_2$ such that the images of the $a_l$ are independent of the `opposite' model $M_{3-l}$ over $M_0$).
        \item \emph{$(<\theta)$-witness property (for singletons)} if whenever $M \lea N$ and $p \in \gS(N)$, if for all $A \subseteq N$ where $|A| < \theta$, there exists $N_0 \lea N$ where $A \subseteq |N_0|$ and $M \lea N_0$ such that $p \upharpoonright N_0$ $\dnf$-does not fork over $M$, then $p$ $\dnf$-does not fork over $M$.
        \item \emph{$\theta$-witness property} if $\dnf$ has the $(< \theta^+)$-witness property.
    \end{enumerate}
\end{defin}

The following definition is inspired by the formulation of $\lambda$-symmetry in \cite[2.6]{vasey18}.

\begin{defin}\label{symmetry-def}
    Let $\K$ be an AC, and $\dnf$ an independence relation on $\K$. We say $\dnf$ has \emph{symmetry} if whenever $M \lea N$, and $a, b \in N$, then the following are equivalent:
        \begin{enumerate}
            \item There exist $M_b, N_b \in \K$ with $M \lea M_b \lea N_b$ and $N \lea N_b$ such that $b \in M_b$ and $\gtp(a/M_b, N_b)$ $\dnf$-does not fork over $M$
            \item There exist $M_a, N_a \in \K$ with $M \lea M_a \lea N_a$ and $N \lea N_a$ such that $a \in M_a$ and $\gtp(b/M_a, N_a)$ $\dnf$-does not fork over $M$
        \end{enumerate}

    If $\K'$ is a sub-AC of $\K$ and the restriction of $\dnf$ to $\K'$ satisfies symmetry, we say $\dnf$ has \emph{symmetry in $\K'$}.
\end{defin}

We summarise how several of these properties are related. The following is essentially \cite[Corollary III.4.4]{shbook}.

\begin{fact}
    Let $\K$ be an AC with an independence relation $\dnf$. If $\dnf$ satisfies uniqueness and extension, then $\dnf$ satisfies transitivity. 
\end{fact}

Extension has a weaker formulation, but assuming uniqueness it is equivalent to our version.

\begin{lemma}\label{extension-equivalence}
Let $\K$ be an AC with an independence relation $\dnf$  satisfying uniqueness. Then extension is equivalent to saying that whenever $M \lea N$ and $p \in \gS(M)$ $\dnf$-does not fork over $M$, then there is $q \in \gS(N)$ with $p \subseteq q$ and $q$ $\dnf$-does not fork over $M$.
\end{lemma}

\begin{proof}
    The version from Definition \ref{independence-relation-properties} implies the second version by setting $M_0 = M$. For the reverse implication, suppose $M_0 \lea M \lea N$ with $p \in \gS(M)$ such that $p$ $\dnf$-does not fork over $M_0$. Then $p \upharpoonright M_0$ $\dnf$-does not fork over $M_0$ by monotonicity. So there exists $q \in \gS(N)$ extending $p \upharpoonright M$ such that $q$ $\dnf$-does not fork over $M_0$. By monotonicity, $q \upharpoonright M$ $\dnf$-does not fork over $M_0$, so by uniqueness, $q \upharpoonright M = p$. Hence $p \subseteq q$ and $q$ $\dnf$-does not fork over $M_0$ as desired.
\end{proof}

\begin{lemma}\label{wp-implies-universal_continuity} Let $\K$ be an AC with an independence relation $\dnf$.
    If $\dnf$ satisfies the $(<\theta)$-witness property in some regular $\theta$, then $\dnf$ satisfies $(\geq \theta)$-universal continuity.
\end{lemma}

\begin{proof}
    Suppose $\langle M_i : i \leq \gamma \rangle$ are in $\K$ for some regular $\gamma \geq \theta$, where $M_\gamma = \bigcup_{i < \gamma}M_i$. Suppose $p \in \gS(M_\gamma)$, and $p \upharpoonright M_i$ $\dnf$-does not fork over $M_0$ for all $i <\gamma$. 
    
    We will verify the hypotheses of the $(<\theta)$-witnessing property for $p$ and $M_0$. Fix $A \subseteq M_\gamma$ with $|A| < \theta$. Since $\gamma$ is regular and $|A| < \gamma$, there exists $i < \gamma$ such that $A \subseteq M_i$. We know $M_0 \lea M_i$, $A \subseteq |M_i|$, and $p \upharpoonright M_i$ $\dnf$-does not fork over $M_0$. Therefore, by the $(<\theta)$-witnessing property, $p$ $\dnf$-does not fork over $M_0$ as desired. 
\end{proof}

\begin{remark}\label{strong-local-char-implies-universal}
Let $\K$ be an AC with an independence relation $\dnf$.
If $\dnf$ satisfies strong $(<\kappa)$-local character then $\dnf$ satisfies $(\geq \kappa)$-local character.
\end{remark}

In a nice enough independence relation, symmetry implies non-forking amalgamation.

\begin{fact}\label{symm-implies-nfap-fact}
    Let $\K$ be an AC with AP. Assume $\dnf$ is an independence relation on $\K$ satisfying extension, uniqueness, and symmetry. Then $\dnf$ satisfies non-forking amalgamation.
\end{fact}

\begin{proof}
    By the same method as \cite[5.2]{vasey18} or \cite[16.2]{vaseyn}, which are based on \cite[II.2.16]{shelahaecbook}, but replacing the use of $\lambda$-symmetry with symmetry of $\dnf$. The uses of the monster model can be avoided with care.
\end{proof}

Two independence relations of interest are \emph{$\lambda$-non-splitting} and \emph{$\lambda$-non-forking}. The following definition follows \cite[I.4.2]{van06}, but both can be traced back to Shelah \cite{shbook}.

\begin{defin}\label{splitting_def}
    Let $\K$ be an AC with AP, and $\lambda \geq \LS(\K)$. Let $M, N \in \K$ with $M \lea N$. We say that \emph{$p \in \gS(N)$ splits over $M$} if there exist $N_1, N_2 \in \K$ with $M \lea N_l \lea N$ for $l = 1, 2$, and an isomorphism $f : N_1 \underset{M}{\rightarrow} N_2$, such that $f(p \upharpoonright N_1) \neq p \upharpoonright N_2$. We say $p$ \emph{$\lambda$-splits over $M_0$} if we additionally require that $\|N_1\| = \|N_2\| = \lambda$.
    
    \emph{Non-splitting} is the independence relation $\dnf_{\splt}$ given by $a \underset{M_0}{\dnf_{\splt}^N} M$ if and only if $\gtp(a/M, N)$ does not split over $M_0$. Similarly define \emph{$\lambda$-non-splitting}, the relation $\dnf_{\lambda-\splt}$, from $\lambda$-splitting.
\end{defin}

The following definition follows \cite[4.2,3.8]{vasey16} (see also \cite[4.1]{leu2}).

\begin{defin}\label{forking_def}
    Let $M, N \in \K_{\lambda}$. We say that \emph{$p \in \gS(N)$ does not $\lambda$-fork over $M$} if and only if there is $M_0 \in \K_\lambda$ such that $p$ does not $\lambda$-split over $M_0$ and $M_0 \lea^u M$.
    
    \emph{$\lambda$-non-forking} is the independence relation $\dnf_{\lambda-f}$, defined as in Definition \ref{splitting_def} from $\lambda$-forking.
\end{defin}

\begin{defin}
    Let $M, N \in \K_{\geq\lambda}$. We say that \emph{$p \in \gS(N)$ does not $(\geq\lambda)$-fork over $M$} if and only if there is $M' \in \K_\lambda$ where $M' \lea M$ such that for all $N' \in \K_\lambda$ with $M' \lea N' \lea N$, we have that $p \upharpoonright N'$ does not $\lambda$-fork over $M'$.

    \emph{$(\geq\lambda)$-non-forking} is the independence relation $\dnf_{(\geq\lambda)-f}$, defined as in Definition \ref{splitting_def} from $(\geq\lambda)$-forking.
\end{defin}

\begin{remark}\label{symmetry-remark}
    All versions of non-splitting and non-forking satisfy invariance, monotonicity, and base monotonicity. Much more can be said under additional assumptions (see Example \ref{dnf-examples-vasey} and Example \ref{dnf-examples-leung}). Additionally, $(\geq \lambda)$-non-forking is the same as $\lambda$-non-forking when restricted to $\K_\lambda$ (use $M' = M$ and monotonicity).
\end{remark}

\begin{defin}\label{lambda-symmetry-def}
    Let $\K$ be a AEC stable in $\lambda \geq \LS(\K)$ with AP, JEP, and NMM in $\K_\lambda$. Let $\kappa< \lambda^+$ be a regular cardinal. We say $\K$ has \emph{$(\lambda, \geq\kappa)$-symmetry} if $\dnf_{\lambda-f}$ has symmetry in $\Kkappalims$.  We say $\K$ has \emph{$\lambda$-symmetry} if $\K$ has $(\lambda, \geq \aleph_0)$-symmetry.
\end{defin}

\begin{remark}
    In fact there are several alternative definitions of $\lambda$-symmetry, and most turn out to be equivalent in superstable AECs \cite[4.3]{vanvas}, \cite[2.6]{vasey18}. In the contexts we examine, our definition of $\lambda$-symmetry is weakest (when we should have $M \lea^u M_a$ or $M_a$ limit over $M$, universal or limit, you can always find a limit model over $M_a$ and apply symmetry to that). In the strictly stable context, more exacting forms of symmetry exist ($(\lambda, \delta)$-symmetry in \cite[5.2]{leu2}, \cite[2.8]{bovan}), but $(\lambda, \geq \kappa)$-symmetry is more natural in our setting.
\end{remark}

There are many well known classes $\K$ and relations $\dnf$ that satisfy some combination of the properties we just introduced. The following include our main examples of interest for this paper.

\begin{example}\label{dnf-examples-fo} 
    Suppose $T$ is a first order theory, stable in $\lambda \geq |T|$. Define $\dnf$ on $(\operatorname{Mod}(T), \preccurlyeq)$ by $a \dnf_{M_0}^N M $ if and only if $\mathbf{tp}(a/M, N)$ does not fork over $M_0$ (in the usual sense). This satisfies invariance, monotonicity, base monotonicity, uniqueness, extension, $(\geq \kappa)$-local character in some $\kappa <\lambda^+$, universal continuity, non-forking amalgamation, and the $(<\aleph_0)$-witness property (see Lemma \ref{first-order}).
\end{example}

\begin{example}\label{dnf-examples-lrv}
    In \cite[\textsection3]{lrv1}, the authors define a notion of a weakly stable independence relation from a categorical perspective on `amalgams' (from here on we will call this a \emph{weakly stable independence relation in the LRV sense}). This can be viewed as a relation on 4-tuples of models. Given such a relation $\dnf$ on an AEC, they show it can be extended to a relation $\dnfb{}{}{}{}$ that allows the intermediate models to be replaced by arbitrary subsets (i.e. the relation $\dnfb{M_0}{A}{B}{N}$ is defined for $M_0 \lea N$, $A, B \subseteq N$)\cite[8.2]{lrv1}. In \cite[\textsection8]{lrv1} they show such relations satisfy (broader versions of) invariance, monotonicity, base monotonicity, existence, extension, uniqueness, and transitivity. These relations also satisfy a form of symmetry, which says that $\dnfb{M_0}{A}{B}{N} \iff \dnfb{M_0}{B}{A}{N}$ (we will call this \emph{symmetry in the LRV sense} to distinguish it from Definition \ref{symmetry-def}).
        
    Further, if $\dnf$ is a stable independence relation in the LRV sense, rather than weakly stable, $\dnf$ satisfies the $(<\theta)$-witness property in the sense of \cite[8.7]{lrv1} (over models rather than singletons) in some cardinal $\theta$ (from here on we will call this the \emph{$(<\theta)$-witness property in the LRV sense}), and has strong $(<\kappa)$-local character for some cardinal $\kappa$. If we restrict to singletons and models, $\dnfb{}{}{}{}$ has many of the useful properties we listed (see Lemma \ref{lrv-dnf-property-lemma}).

\end{example}

\begin{lemma}\label{lrv-dnf-property-lemma}
    Suppose $\dnf$ is a weakly stable independence relation in the LRV sense. Then the restriction $\dnfb{}{}{}{}$ to singletons and models (that is, restrict to the case $\dnfb{M_0}{a}{M}{N}$) is an independence relation that satisfies invariance, monotonicity, base monotonicity,  extension, uniqueness, non-forking amalgamation, and transitivity.
            
    Moreover:
    \begin{enumerate} 
        \item if $\dnfb{}{}{}{}$ has strong $(<\kappa)$-local character, then $\dnfb{}{}{}{}$ has $(\geq \kappa)$-local character
        \item if $\dnfb{}{}{}{}$ satisfies the $(<\theta)$-witness property (for singletons), then the restriction does also, and if $\theta$ is regular, then $\dnfb{}{}{}{}$ also satisfies $(\geq \theta)$-universal continuity. 
    \end{enumerate}
\end{lemma}

\begin{proof}
    Invariance, monotonicity, base monotonicity, extension, uniqueness, transitivity, and the $(<\theta)$-witness property all follow immediately from their versions in the $\dnfb{M_0}{A}{B}{N}$ case.
    
    Next we address non-forking amalgamation. Suppose $M_0, M_1, M_2 \in K$ with $M_0 \lea M_l$ for $l = 1, 2$ and $a_l \in M_l$ for $l = 1, 2$. Let $i_{0, l}:M_0 \rightarrow M_l$ be the identity maps for $l = 1, 2$. By existence, there are $N \in \K$ and $f_l : M_l \rightarrow N$ such that $\dnf(i_{0, 1}, i_{0, 2}, f_1, f_2)$ - that is, $f_1 \upharpoonright M_0 = f_2 \upharpoonright M_0$ and $f_1(M_1) \dnf_{f_1(M_0)}^N f_2(M_2)$. Since $f_1 \upharpoonright M_0 = f_2 \upharpoonright M_0$ and using invariance, we may assume $f_l \upharpoonright M_0 = \id_{M_0}$ by composing $f_1, f_2$ both with an extension of $(f_1 \upharpoonright M_0)^{-1}$ to $N$. By symmetry of $\dnf$ in the LRV sense, we also have $f_2(M_2) \dnf_{M_0}^N f_1(M_1)$. As $f_l(a) \in f_l(M_l)$, we have $\dnfb{M_0}{f_l(a_l)}{f_{3-l}(M_{2-l})}{N}$ for $l = 1, 2$; that is, $\gtp(f_l(a_l)/f_{3-l}(M_l), N)$ $\dnf$-does not fork over $M_0$ for $l = 1, 2$ as desired.

    For the moreover part (1), $(\geq \kappa)$-local character follows from Remark \ref{strong-local-char-implies-universal}. For the moreover part (2), as before the $(<\theta)$-witness property for singletons holds for the restriction immediately. For $(\geq \theta)$-universal continuity, apply Fact \ref{wp-implies-universal_continuity}.
        \end{proof}

\begin{remark}
    While they look similar, in general it is not obvious whether the $(<\theta)$-witness property (for singletons) is equivalent to the $(<\theta)$-witness property in the sense of LRV. That said, if $\K$ has intersections (that is, for all $N \in \K$ and $A \subseteq |N|$, $\bigcap \{M \in \K : A \subseteq M \lea N\} \lea N$) and $\dnfb{}{}{}{}$ is a stable independence relation with the $(<\theta)$-witnessing property in the LRV sense, then $\dnfb{}{}{}{}$ has the witness property (for singletons).
\end{remark}

\begin{example}[{\cite[13.16]{vaseyn}}]\label{dnf-examples-vasey}
 If $\K$ is $\lambda$-superstable and has $\lambda$-symmetry, then $\lambda$-forking restricted to $\K_\lambda$ satisfies invariance, monotonicity, base monotonicity, uniqueness, extension, non-forking amalgamation, universal continuity, and $(\geq \aleph_0)$-local character.

\end{example}

\begin{example}[{\cite[\textsection 4]{leu2}}]\label{dnf-examples-leung}
        
        Suppose $\K$ is an AEC with AP, JEP, and NMM, and there is $\lambda \geq \LS(\K)$ such that
        \begin{enumerate}
            \item $\K$ is stable in $\lambda$
            \item $\K$ is $\lambda$-tame
            \item when restricted to  $\K_\lambda$, $\dnf_{\operatorname{\lambda-\splt}}$ satisfies universal continuity and $(\geq \kappa)$-local character for some regular $\kappa \leq \lambda$
            \item $\K$ has $(\lambda, \geq \kappa)$-symmetry.
        \end{enumerate}
        Let $\dnf$ be the restriction of $\dnf_{\lambda-f}$ to $\Kkappalims$. Assuming that $\dnf$ satisfies uniqueness, $\dnf$ satisfies invariance, monotonicity, base monotonicity, uniqueness, extension, non-forking amalgamation, universal continuity, and $(\geq \kappa)$-local character.
\end{example}

\begin{remark}\label{leung-uniqueness-necessary}
    Example \ref{dnf-examples-leung} above is stated in \cite{leu2} without the uniqueness assumption (uniqueness is proved from the other hypotheses). Unfortunately, we located an error in the proof of uniqueness \cite[4.5]{leu2}. In the penultimate paragraph of the proof, it is claimed that the map $f$ fixes $M_0$, but this is not true (in fact the image of $f$ is a proper substructure of $M_0$ by NMM and $f[M_{i+1}] = N^* \lea^u M_0$). We have not found a way to fix this issue. Note if $\kappa = \aleph_0$, then uniqueness holds, as in that case $\K$ is $\lambda$-superstable and we fall into Example \ref{dnf-examples-vasey}.
\end{remark}

\begin{example}[{\cite[\textsection 4, \textsection 5]{vasey16b}}]\label{dnf-examples-vasey-saturated}
    Let $\K$ be AEC with AP, NMM, stable in some $\mu \geq \LS(\K)$, and satisfying JEP in $\K_\lambda$ and $\mu$-tameness. Let $\dnf$ be $\dnf_{(\geq\mu)-f}$ restricted to models in $\K_{\geq \mu^+}^{\mu^+\operatorname{-sat}}$ (that is, the $\mu^+$-saturated models in $\K_{\geq \mu^+}$ ordered by $\lea$). Then $\dnf$ has many of the useful properties we listed (see Lemma \ref{dnf-lemma-vasey-saturated}).
\end{example}

The following notation is similar to that of  \cite{vaseyt}.
\begin{defin}\label{kappa-defs}

Assume $\K$ is an AEC stable in $\lambda$ and $\dnf$ is an independence relation on $\K_\lambda$ (or any AC $\K'$ with $\K_\lambda \subseteq \K' \subseteq \K$).
\begin{enumerate}
    \item $\underline{\kappa}(\dnf, \K_\lambda, \lea^{u}) = \{ \delta < \lambda^+ : \text{whenever } \langle M_i : i \leq \delta \rangle$ is a $\lea^u$-increasing continuous chain in $\K_\lambda$ and  $p \in \gS(M_\delta), \text{ then there is } i <\delta \text{ such that } p \dnf\text{-does not fork over } M_i \}$
    \item $\underline{\kappa}^{\operatorname{wk}}(\dnf, \K_\lambda, \lea^{u}) = \{ \delta < \lambda^+ : \text{whenever } \langle M_i : i \leq \delta \rangle $   is a $\lea^u$-increasing continuous chain in $\K_\lambda$ and  $p \in \gS(M_\delta), \text{ then there is } i <\delta \text{ such that } p\upharpoonright M_{i+1} \dnf\text{-does not fork over } M_i \}$
    \item $\kappa(\dnf, \K_\lambda, \lea^{u}) = \operatorname{min}\{\mu \leq \lambda: [\mu, \lambda^+) \cap \text{Reg} \subseteq \underline{\kappa}(\dnf, \K_\lambda, \lea^{u}) \}$ when it exists, else $\kappa(\dnf, \K_\lambda, \lea^{u}) = \infty$.
\end{enumerate}
\end{defin}

\begin{remark}
    If $\K$ is an AEC stable in $\lambda \geq \LS(\K)$, then for all limit $\delta < \lambda^+$, $\delta \in \underline{\kappa}(\dnf, \K_\lambda, \lea^{u})$ if and only if $\cof(\delta) \in \underline{\kappa}(\dnf, \K_\lambda, \lea^{u})$.
\end{remark}

\begin{nota}
We will sometimes use the shorthand $\shortkappadnfu$ to denote $\kappadnfu$ when the AEC $\K$ is unambiguous to avoid notational clutter in more complicated expressions.
\end{nota}

\begin{remark}\label{weak-loc-char-iff-strong}
    If $\K$ is an AEC and $\dnf$ is an independence relation on $\K_\lambda$ with extension, uniqueness, and universal continuity, then $\underline{\kappa}(\dnf, \K_\lambda, \lea^{u}) \cap Reg = \underline{\kappa}^{\operatorname{wk}}(\dnf, \K_\lambda, \lea^{u})\cap Reg$. The $\subseteq$ direction follows from the same method as \cite[11(1)]{bgvv16}, the $\supseteq$ direction follows from monotonicity. In particular, since ${\underline{\kappa}^{\operatorname{wk}}(\dnf, \K_\lambda, \lea^{u})}$ is an interval of the form $[\alpha, \lambda^+)$, in this case $\kappa(\dnf, \K_\lambda, \lea^{u}) = \operatorname{min}( {\underline{\kappa}(\dnf, \K_\lambda, \lea^{u})} \cap Reg)$.
\end{remark}

\begin{remark}
    Definition \ref{kappa-defs} is inspired by the notion of $\underline{\kappa}(\K_\lambda, <^u_{\K})$ from \cite[3.8]{vasey18}, and in fact for any AEC $\K$, $\underline{\kappa}(\K_\lambda, <^u_{\K}) = \underline{\kappa}(\dnf_{\splt}, \K_\lambda, \lea^u) \cup ([\lambda^+, \infty)\cap Reg)$. The regular cardinals greater or equal to $\lambda^+$ give no new information, since no $\lea^u$-increasing sequences of those lengths exist inside $\K_\lambda$ under NMM in $\K_\lambda$. They are included in $\underline{\kappa}(\K_\lambda, <^u_{\K})$ because the broader definition of $\underline{\kappa}$ in \cite[2.2]{vasey18} allows for classes with arbitrarily large models, but our definition is more intuitive when restricted to $\K_\lambda$.
\end{remark}

\section{Long limit models}\label{long-limits-section}


Our goal in this section is to show that, in a very general setting, all the high cofinality limit models are the same. We first present the result for convenience of the reader, then introduce the precise hypotheses below.

\begin{restatable}{theorem}{largelimitsareisothmii}\label{largelimitsareisothm*}
    Assume Hypothesis \ref{long-limit-dnf-hypotheses} holds for an AEC $\K$, $\lambda \geq \LS(\K)$, and $\kappa < \lambda^+$ regular. Let $\delta_1, \delta_2 < \lambda^+$ be limit ordinals where $\kappa \leq \cof(\delta_1), \cof(\delta_2)$.
    If $M, N_1, N_2 \in \K_\lambda$ where $N_l$ is a $(\lambda, \delta_l)$-limit over $M$ for $l=1, 2$, then there is an isomorphism from $N_1$ to $N_2$ fixing $M$.

    Moreover, if $N_1, N_2 \in \K_\lambda$ where $N_l$ is $(\lambda, \delta_l)$-limit for $l=1, 2$, then $N_1$ is isomorphic to $N_2$.
\end{restatable}

First we state the natural surrogate for universal continuity when an independence relation is only defined on a sub-AC of an AEC - specifically $\Kkappalims$.

\begin{defin}\label{univ-cty-star-def}
    Let $\K$ be an AEC, $\K'$ a sub-AC of $\K$, and $\dnf$ an independence relation on $\K'$.\footnote{Post-publication footnote: In addition one needs to assume that \emph{$\K'$ preserves types from $\K$}, i.e.,
    \begin{enumerate}
		\item for all $M \in \K'$, $\varphi_M : \gS_{\K'}(M) \rightarrow \gS_{\K}(M)$ given by $\varphi_M(\gtp_{\K'}(a/M, N)) = \gtp_{\K}(a/M, N)$ for $M \leq_{\K'} N$ and $a \in N$ is a well-defined bijection, and
		\item when $M \leq_{\K'} N$, $p \in \gS_{\K'}(M)$, and $q \in \gS_{\K'}(N)$, then $p \subseteq q$ if and only if $\varphi_M(p) \subseteq \varphi_N(q)$.
	\end{enumerate}
    
    Essentially this says we can `forget' about whether types over models in $\K'$ are computed in $\K$ or $\K'$. 
    
       The assumption that \emph{$\K'$ preserves types from $\K$} is not in the published version of the paper. Fortunately, this  does not affect any of the results of the paper, since preservation of types holds in all the settings we consider in the paper. In particular, in the context of Definition \ref{lims-univ-cty-star-def} where $\Kkappalims \subseteq \K' \subseteq \K_\lambda$, as well as when $\K' = \K_{\geq \mu^+}^{\mu^+-\operatorname{sat}}$ in Subsection \ref{tame-long-lim-subsection}.}

    We say that $\dnf$ has \emph{universal continuity* in $\K$} if and only if whenever $\delta$ is an ordinal and $\langle M_i : i \lea \delta \rangle$ is a $\lea^u$-increasing sequence in $\K'$ with $\bigcup_{i<\delta} M_i \lea M$ for some $M \in \K'$, and $\langle p_i \in \gS_{\K'}(M_i) : i < \delta\rangle$ is an increasing sequence of types where $p_i$ $\dnf$-does not fork over $M_0$ for all $i < \delta$, then there is a unique $p_\delta \in \gS_{\K}(\bigcup_{i<\delta}M_i)$ such that $p_i \subseteq p_\delta$ for all $i < \delta$.

    We may omit \emph{in $\K$} when $\K$ is clear from context.
\end{defin}

\begin{remark} In Definition \ref{univ-cty-star-def}, intuitively, $p_\delta$ is the unique $\dnf$ non-forking type extending $p_0$ over $\bigcup_{i<\delta} M_i$ - the relation $\dnf$ is not necessarily defined on this model, but if it could be extended, $p_\delta$ would be the only choice for the non-forking extension. \end{remark}

\begin{defin}\label{lims-univ-cty-star-def}
    Let $\K$ be an AEC stable in $\lambda \geq \LS(\K)$, and $\kappa < \lambda^+$ be infinite and regular. Let $\dnf$ be an independence relation on some sub-AC $\K'$ of $\K_\lambda$ where $\Kkappalims \subseteq \K'$ (normally $\K' = \K$ or $\K' = \Kkappalims$). 

    We say that $\dnf$ has \emph{$\Kkappalims$-universal continuity* in $\K$} if $\dnf$ restricted to $\Kkappalims$ has universal continuity* (in $\K$).
\end{defin}

\begin{nota}
    We will often omit the `in $\K$' part of the definition as in most cases $\K$ is clear from context.
\end{nota}

\begin{remark}
    The reason we use $\Kkappalims$-universal continuity* rather than `standard' universal continuity is to accommodate Example \ref{dnf-examples-leung} and the setup of Subsection \ref{tame-long-lim-subsection}, where the relation behaves well on $\Kkappalims$ but not on all models (or even all limit models). In a sense, this is the closest to continuity we can get when $\dnf$ is only defined on $\Kkappalims$.  We formalise this to some degree in Lemma \ref{dnf-high-continuity} and Lemma \ref{cont-iff-weak-cont}.
\end{remark}

Now we specify the conditions our non-forking relation needs to satisfy to apply Theorem \ref{largelimitsareisothm*}.

\begin{restatable}{hypothesis}{longlimithypothesis}\label{long-limit-dnf-hypotheses}
    Let $\K$ be an AEC stable in $\lambda \geq \LS(\K)$, with AP, JEP, and NMM in $\K_\lambda$. Let $\kappa < \lambda^+$ be a regular cardinal. Let $\dnf$ be an independence relation on $\Kkappalims$ that satisfies uniqueness, extension, non-forking amalgamation, $(\geq \kappa)$-local character, and $\Kkappalims$-universal continuity* in $\K$. 
\end{restatable}

\begin{remark}\label{long-lim-can-restrict}
    If we made the same assumptions on a relation $\dnf$ defined on all of $\K_\lambda$ (or any sub-AC $\K'$ with $\Kkappalims \subseteq \K' \subseteq \K_\lambda$), the restriction to $\Kkappalims$ satisfies Hypothesis \ref{long-limit-dnf-hypotheses}. This is immediate for each property besides non-forking amalgamation - in that case, just note that the `largest' model $N$ can be replaced by a $(\lambda, \geq \kappa)$-limit model extending the original model.
\end{remark}

\begin{remark}
    For $M, N \in \Kkappalims$, $N$ is universal over $M$ in $\Kkappalims$ if and only if $N$ is universal over $M$ in $\K$, so for the properties involving $\leq_{\Kkappalims}^u$ ($\Kkappalims$-universal continuity*, local character) hold with the usual $\lea^u$. Hence we can use $\lea^u$ instead unambiguously when dealing with ACs $\K'$ with $\Kkappalims \subseteq \K' \subseteq \K$.
\end{remark}

Before proving Theorem \ref{largelimitsareisothm*}, we explore the assumptions in Hypothesis \ref{long-limit-dnf-hypotheses} and consider some examples that satisfy it.

\begin{lemma}\label{existence-on-high-limits}
    Suppose $\K$ is an AEC, with sub-AC $\K'$ where $\Kkappalims \subseteq \K'$. Suppose $\dnf$ is an independence relation on $\K'$ satisfying $(\geq\kappa)$-local character. If $M \in \Kkappalims$ and $p \in \gS(M)$, then $p$ $\dnf$-does not fork over $M$.

    In particular, assuming Hypothesis \ref{long-limit-dnf-hypotheses}, $\dnf$ satisfies existence.
\end{lemma}

\begin{proof}
    Let $\langle M_i : i \leq \delta \rangle$ witness that $M$ is a $(\lambda, \geq \kappa)$-limit. Then by $(\geq \kappa)$-local character, $p$ $\dnf$-does not fork over $M_i$ for some $i < \delta$. By base monotonicity, $p$ $\dnf$-does not fork over $M$.
\end{proof}

The following lemmas show the relationship between full universal continuity and our replacement, $\Kkappalims$-universal continuity*; in particular, universal continuity and universal continuity* are equivalent for nice $\dnf$ when $\K_{(\lambda, \geq \aleph_0)} \subseteq \K' \subseteq \K$ (see Lemma \ref{cont-iff-weak-cont}).

\begin{lemma}\label{dnf-high-continuity}
 Suppose $\K$ is an AEC with amalgamation in $\K_\lambda$, and $\dnf$ is an independence relation on a sub-AC $\K'$ of $\K_\lambda$ where $\Kkappalims \subseteq \K'$ satisfying extension, uniqueness, and $\Kkappalims$-universal continuity* in $\K$. Then $\dnf$ has $(\geq \kappa)$-universal continuity.
\end{lemma}

\begin{proof}
    Suppose $\langle M_i : i < \delta \rangle$ is a $\lea^u$-increasing chain in $\K'$, and $p \in \gS(\bigcup_{i<\delta}M_i)$ is such that $p \upharpoonright M_i$ $\dnf$-does not fork over $M_0$. We must show $p$ $\dnf$-does not fork over $M_0$.
    
    Note $\bigcup_{i<\delta} M_i \in \Kkappalims \subseteq \K'$. So by extension, there exists $q \in \gS(\bigcup_{i<\delta} M_i)$ such that $q$ $\dnf$-does not fork over $M_0$ and $q \upharpoonright M_0 = p \upharpoonright M_0$. By uniqueness, $p \upharpoonright M_i = q \upharpoonright M_i$ for $i < \delta$. For $i < \delta$, take $N_i \in \Kkappalims$ such that $M_i \lea^u N_i \lea^u M_{i+1}$ (this is possible as $\langle M_i : i < \delta \rangle$ is $\lea^u$-increasing). We now have $\langle N_i : i < \delta \rangle$ $\lea^u$-increasing where $p \upharpoonright N_i = q \upharpoonright N_i$ for all $i < \delta$ and $\bigcup_{i<\delta}N_i = \bigcup_{i<\delta}M_i$. So $p = q$ by $\Kkappalims$-universal continuity* in $\K$. Therefore $p$ $\dnf$-does not fork over $M_0$ as desired.
\end{proof}

\begin{lemma}\label{cont-iff-weak-cont}
    Suppose $\K$ is an AEC with amalgamation in $\K_\lambda$, and $\dnf$ is an independence relation on a sub-AC $\K'$ of $\K_\lambda$ where $\K_{(\lambda, \geq \aleph_0)} \subseteq \K'$ satisfying extension and uniqueness.
    
    Then $\dnf$ satisfies universal continuity if and only if $\dnf$ satisfies $\K_{(\lambda, \geq \aleph_0)}$-universal continuity* in $\K$.
\end{lemma}

\begin{proof}
    First, assume $\dnf$ satisfies universal continuity. Suppose $\delta < \lambda^+$, $\langle M_i : i < \delta \rangle$ a $\lea^u$-increasing sequence in $\K_{(\lambda, \geq \aleph_0)}$, and $\langle p_i \in \gS(M_i) : i < \delta\rangle$ is an increasing sequence of types where $p_i$ $\dnf$-does not fork over $M_0$. Let $M_\delta = \bigcup_{i<\delta}M_i$. The result is trivial if $\delta$ is 0 or successor, so assume $\delta$ is limit. Since $M_\delta \in \K_{(\lambda, \geq \aleph_0)}$, $M \in \K'$, and we can take any $p \in \gS(M_\delta)$ extending $p_0$ which $\dnf$-does not fork over $M_0$ by extension. By uniqueness, since $p \upharpoonright M_i$ and $p_i$ are both non-forking extensions of $p_0$, $p \upharpoonright M_i = p_i$. And $p$ is the unique extension of the $p_i$'s: if also $q \in \gS(M_\delta)$ extends $p_i$ for all $i < \delta$, universal continuity gives that $q$ $\dnf$-does not fork over $M_0$. $p \upharpoonright M_0 = p_0 = q \upharpoonright M_0$, so by uniqueness, $p = q$. Therefore $p$ is the unique extension as desired.

    The converse follows from Lemma \ref{dnf-high-continuity} with $\kappa = \aleph_0$.
\end{proof}

Our motivating examples satisfying Hypothesis \ref{long-limit-dnf-hypotheses} are the following:

\begin{example}\label{dnf-examples-long-lim}\
    \begin{enumerate}
        \item Suppose $\dnf$ is a stable independence relation in the LRV sense with universal continuity on an AEC $\K$, and $\lambda \geq \LS(\K)$ is a stability cardinal where $\K_\lambda$ has JEP and NMM. Then the restriction of $\dnfb{}{}{}{}$ to singletons and $(\lambda, \geq\kappa)$-limit models satisfies Hypothesis \ref{long-limit-dnf-hypotheses} with $\kappa = \kappa(\dnf, \K_\lambda, \lea^u)$, by Lemma \ref{lrv-dnf-property-lemma}. Note that all the non-forking properties transfer down to this restriction by Lemma \ref{long-lim-can-restrict}, and $\Kkappalims$-universal continuity* in $\K$ follows from universal continuity of $\dnf$ and Lemma \ref{cont-iff-weak-cont}.
        
        \item Example \ref{dnf-examples-vasey} with $\kappa = \aleph_0$.
        \item Example \ref{dnf-examples-leung} with $\kappa = \kappa$.

        \item Example \ref{dnf-examples-vasey-saturated} assuming $\dnf_{(\geq \lambda)-f}$ has symmetry in $\Kmupluslims$. The background and proofs that this setup satisfy Hypothesis \ref{long-limit-dnf-hypotheses} are in Subsection \ref{tame-long-lim-subsection}.
    \end{enumerate}
\end{example}

\begin{remark}
    In the setting of Example \ref{dnf-examples-long-lim}(2), the main result of this section (Theorem \ref{largelimitsareisothm*}) is already known (proved originally in \cite{van16a}, see also \cite[15.8]{vaseyn}). This section is a generalisation of the method used in \cite{vaseyn}.
\end{remark}

\begin{remark}
    Of our examples, only Example \ref{dnf-examples-long-lim}(3) and  Example \ref{dnf-examples-long-lim}(4) need $\Kkappalims$-universal continuity* in $\K$ rather than regular universal continuity of non-forking. In the other cases, $\dnf$ is defined on all limit models, so we don't have to worry about whether $\dnf$ is well defined when taking unions of towers, and we could more closely mimic the approach of \cite[16.17]{vaseyn}. In this sense,  Example \ref{dnf-examples-long-lim}(3) and  Example \ref{dnf-examples-long-lim}(4) (see Corollary \ref{vasey-tame-saturated-dnf-cor}) are our main examples that use the full strength of Theorem \ref{largelimitsareisothm*}.
\end{remark}


\subsection{Towers}

We assume Hypothesis \ref{long-limit-dnf-hypotheses} throughout this subsection. Note since AP, JEP, and NMM hold in $\K_\lambda$ under Hypothesis \ref{long-limit-dnf-hypotheses}, the Facts \ref{limitsexist*} and \ref{cofinalityiso*} apply, and $\dnf$ satisfies existence by Lemma \ref{existence-on-high-limits}.

In this subsection we start working towards a proof of Theorem \ref{largelimitsareisothm*}. We begin by defining a notion of \emph{towers}.

Towers were introduced by Villaveces and Shelah in \cite{shvi} in their attempt to prove uniqueness of limit models under certain assumptions involving categoricity. Their towers were composed of two increasing chains of models and a list of singletons - the singletons' types over the big models did not $\lambda$-split over the smaller ones. This is also the version of tower used in  \cite{van02}, \cite{van06}, \cite{van13}, \cite{grvavi}, \cite{van16a}, \cite{vand}, and \cite{bovan}. In \cite{vaseyn} and \cite{vasey18}, Vasey simplified the presentation of the argument using a different form of tower. First, he made use of $\lambda$-non-forking, which has stronger properties than $\lambda$-non-splitting on limit models. Also, he took out the models the types do not fork over, and instead captured the $\lambda$-non-forking of the singletons in the tower ordering. Our approach is heavily based off \cite{vasey18}, but with looser assumptions. The main differences are at the points that local character and universal continuity appear in Vasey's proof. In particular, in our setting unions of $\lesst$-increasing towers of length $<\kappa$ may not be towers, so we adapt Vasey's tower extension lemma \cite[16.17]{vasey18} to allow us to find towers that extend a chain of towers, rather than just a single tower (see Proposition \ref{towerextensionprop*}). We will  generalise statements and proofs that involve our weakened assumptions. Some proofs which are very similar to \cite{vasey18} are omitted in the interest of space (Lemma \ref{high-tower-unions-def}, Lemma \ref{disjointness-of-dnf}, Lemma \ref{reducedunionsarereduced*}, and Lemma \ref{highfullunionsarefull*}). However, full proofs can be found in \cite{bemavone}.

\begin{nota}
    Let $I$ be well ordered by $\leq_I$.
    
    \begin{enumerate}
        \item We use $I^-$ to denote 
        \begin{itemize}
            \item $I$ if $I$ has $0$ or limit order type (i.e. it has no final element)
            \item $I \setminus \{i\}$ if $I$ has successor order type where $i$ is the final element of $I$.
        \end{itemize}
        
        \item We use $i +_I 1$ to denote the successor of $i \in I$ in the ordering $\leq_I$ of $I$, if it exists. When unambiguous we may write $i + 1$.

        \item Given $r, s \in I$, we use $[r, s)_I$ to denote the interval of all $i \in I$ with $r \leq_I i <_I s$.
    \end{enumerate}
\end{nota}

That is, $I^-$ removes the final element if it exists. When clear from context, we omit the subscript and write $<$ and $\leq$ for $<_I$ and $\leq_I$.

Our version of tower is almost the same as \cite[5.4]{vasey18}, differing only in that we require all models to be $(\lambda, \geq \kappa)$-limits rather than just limit models (so they are compatible with $\dnf$).

\begin{defin}
   A \emph{tower} is a sequence $\calt = \langle M_i : i \in I \rangle^\wedge \langle a_i : i \in I^- \rangle$ where 
   \begin{enumerate}
       \item $I$ is a well ordered set with $\operatorname{otp}(I) < \lambda^+$
       \item $\langle M_i : i \in I\rangle$ is a $\lea$-increasing sequence of models in $\Kkappalims$
       \item for all $i \in I^-$, $a_i \in |M_{i+1}| \setminus |M_i|$.
   \end{enumerate}

   Given such $I, \calt$:
    \begin{enumerate}
        \item We call $I$ the \emph{index} of the tower $\calt$
        \item If $I_0 \subseteq I$, then define $\calt \upharpoonright I_0 = \langle M_i : i \in I \in I_0 \rangle ^\wedge \langle a_i : i \in (I_0)^- \rangle$
        \item Given $i \in I$ limit, we say $\calt$ is \emph{continuous at $i$} if $M_i = \bigcup_{r<i} M_r$
        \item $\calt$ is \emph{universal} if $\langle M_i : i \in I \rangle$ is universal; that is, for all $i \in I^-$, $M_i \lea^u M_{i+1}$
        \item $\calt$ is \emph{strongly universal} if $\langle M_i : i \in I \rangle$ is strongly universal; that is, for all non-initial $i \in I$, $\bigcup_{r < i} M_r \lea^u M_i$.
    \end{enumerate}    
\end{defin}

\begin{remark}\label{towersexist*}
    It is straightforward to see that given $M \in \Kkappalims$, there is a strongly universal tower $\calt = \langle M_i : i \in \alpha \rangle^\wedge \langle a_i : i \in \alpha^- \rangle$ where $M = M_0$ for any $\alpha < \lambda^+$ (or indeed any well ordering $I$ with $\operatorname{otp}(I) < \lambda^+$): $M_0$ is given; if you have $M_i$ take $\gtp(a_i/M_i, M_{i+1}')$ to be any non-algebraic type (which is possible by NMM) and take $M_{i+1}$ to be any $(\lambda, \kappa)$-limit model over $M_{i+1}'$; at limit $i$ take $M_i$ to be a $(\lambda, \kappa)$-limit over $\bigcup_{k<i} M_k$. Since our models are all in $\Kkappalims$ rather than simply limit models, it is not obvious if continuous towers exist (that is, towers continuous at every limit $i \in I$). However we will be able to guarantee continuity at limits with high cofinality (e.g. take unions when $\cof(i) \geq \kappa$ in the above construction, or use Lemma \ref{reducedextensionsexist*} and Proposition \ref{reducedimplieshighcontinuity*}).
\end{remark}

\begin{defin}[{\cite[5.7]{vasey18}}]\label{towerorderingdef}
    Let $\calt = \langle M_i : i \in I \rangle^\wedge \langle a_i : i \in I^- \rangle$, $\calt' = \langle M_i' : i \in I' \rangle^\wedge \langle a_i' : i \in (I')^- \rangle$ be towers. We define the \emph{tower ordering} $\lesst$ by $\calt \lesst \calt'$ if and only if 
    \begin{enumerate}
        \item $I \subseteq I'$
        \item for all $i \in I$, $M_i'$ is universal over $M_i$
        \item for all $i \in I^-$, $a_i' = a_i$
        \item for all $i \in I^-$, we have $\gtp(a_i/ M_i', M_{i+_{I'}1}')$ $\dnf$-does not fork over $M_i$.
    \end{enumerate}
\end{defin}

\begin{remark}
    $\lesst$ is a strict partial order - for transitivity of $\lesst$, (1) and (3) are clear, (2) follows from transitivity of $\lea^u$, and (4) follows from transitivity of $\dnf$-non-forking.
\end{remark}

In our setting, arbitrary unions of towers may no longer be towers - for example, the models of a union of $\omega$ many towers will be $(\lambda, \omega)$-limit models, so this union may not be a tower if $\kappa \geq \aleph_1$. However, we can take unions of high cofinality, and these interact well with the tower ordering.

\begin{defin}\label{high-tower-unions-def}
    Suppose $\gamma < \lambda^+$ is a limit ordinal where $\cof(\gamma) \geq \kappa$ and $\langle \calt^j : j < \gamma \rangle$ is a $\lesst$-increasing sequence of towers where $\calt^j = \langle M_i^j : i \in I^j \rangle ^\wedge \langle a_i^j : i \in (I^j)^-\rangle$ for $j < \gamma$, and $\bigcup_{j<\gamma} I^j$ is a well ordering. Define $\bigcup_{j<\gamma}\calt^j = \langle M^\gamma_i : i \in I^\gamma \rangle ^\wedge \langle a^\gamma_i : i \in (I^\gamma)^-\rangle$ where
    \begin{enumerate}
        \item $I^\gamma = \bigcup_{j<\gamma}I^j$
        \item for all $i \in I^\gamma$, $M^\gamma_i = \bigcup \{M^j_i : j < \gamma \text{ such that } i \in I^j \}$
        \item for all $i \in (I^\gamma)^-$, $a^\gamma_i$ is any $a^j_i$ for which $i \in (I^j)^-$ (note the choice does not matter by the definition of the tower ordering).
    \end{enumerate}
\end{defin}

\begin{lemma}\label{high-tower-unions-exist}
    Suppose $\gamma < \lambda^+$ is a limit ordinal where $\cof(\gamma) \geq \kappa$ and $\langle \calt^j : j < \gamma \rangle$ is a $\lesst$-increasing sequence of towers where $\calt^j$ is indexed by $I^j$ for $j < \gamma$, and $\bigcup_{j<\gamma} I^j$ is a well ordering. Then $\bigcup_{j<\gamma} \calt^j$ is a tower, and for all $k < \gamma$, $\calt^k \lesst \bigcup_{j<\gamma} \calt^j$.
\end{lemma}

\begin{proof}

    The method is the same as \cite[5.13]{vasey18}, besides noting that $\cof(\gamma) \geq \kappa$ ensures $\bigcup_{j<\gamma} \calt^j$ consists of $(\lambda, \geq \kappa)$-limit models. See \cite[3.22]{bemavone} for the details. \end{proof}

To prove Theorem \ref{largelimitsareisothm*}, we will construct a $(\delta_1+1)$-long $\lesst$-increasing sequence of towers, continuous at $\delta_1$, where the towers' indexes contain a copy of $(\delta_2 + 1)$, and the final tower's models indexed by this copy will be $\lea^u$-increasing and continuous at `$\delta_2$' (in the copy of $(\delta_2 + 1)$). This will guarantee the `$\delta_2$'th model in the final tower will be both a $(\lambda, \delta_1)$-limit model and $(\lambda, \delta_2)$-limit model.

\def\uppersemi at (#1,#2){\def\Radius{0.75}
    
  \draw
    (\Radius + #1, #2) arc(0:180:\Radius);}

\def\sidesemi at (#1,#2){\def\Radius{0.75}
    
  \draw
    (#1, #2 - \Radius) arc(-90:90:\Radius);}

\def\fullcircle at (#1,#2){\def\Radius{0.75}
    
  \draw
    (\Radius + #1, #2) arc(0:360:\Radius);}

\def\upperright at (#1,#2){\def\Radius{0.75}
    
  \draw
    (\Radius + #1, #2) arc(0:90:\Radius);}
    
\def\upperleft at (#1,#2){\def\Radius{0.75}
    
  \draw
    (#1, \Radius + #2) arc(90:180:\Radius);}

\def\threequarter at (#1,#2){\def\Radius{0.75}
    
  \draw
    (#1, #2 - \Radius) arc(-90:180:\Radius);}

\begin{figure}[!ht]
\centering
{%
\begin{tikzpicture}
\tikzstyle{every node}=[font=\small]


\sidesemi at (13.5, 8.25)
\upperright at (13.5,9)
\upperright at (13.5,9.75)
\upperright at (13.5,10.5)
\upperright at (13.5,11.25)
\upperright at (13.5,12)
\upperright at (13.5,14.5)

\node [font=\small] at (13.6,8.25) {$M_0^{\delta_1}$};
\node [font=\small] at (13.6,9.3) {$M_1^{\delta_1}$};
\node [font=\small] at (13.6,10.05) {$M_2^{\delta_1}$};
\node [font=\small] at (13.6,10.8) {$M_3^{\delta_1}$};
\node [font=\small] at (13.6,11.55) {$M_{3.5}^{\delta_1}$};
\node [font=\small] at (13.6,12.3) {$M_4^{\delta_1}$};
\node [font=\small] at (13.6,14.8) {$M_{\delta_2}^{\delta_1}$};
\node [font=\small] at (13.6,13.6) {$\vdots$};

\node [font=\small] at (13.75,15.75) {$\mathcal{T}^{\delta_1}$};


\node [font=\small] at (12.5,8.25) {$\dots$};
\node [font=\small] at (12.5,9.3) {$\dots$};
\node [font=\small] at (12.5,10.05) {$\dots$};
\node [font=\small] at (12.5,10.8) {$\dots$};
\node [font=\small] at (12.5,11.55) {$\dots$};
\node [font=\small] at (12.5,12.3) {$\dots$};
\node [font=\small, rotate=90] at (12.3,13.6) {$\ddots$};
\node [font=\small] at (12.5,14.8) {$\dots$};

\sidesemi at (10.75,8.25)
\upperright at (10.75,9)
\upperright at (10.75,9.75)
\upperright at (10.75,10.5)
\upperright at (10.75,11.25)
\upperright at (10.75,12)
\upperright at (10.75, 14.5)

\node [font=\small] at (11,8.25) {$M_0^4$};
\node [font=\small] at (11,9.3) {$M_1^4$};
\node [font=\small] at (11,10.05) {$M_2^4$};
\node [font=\small] at (11,10.8) {$M_3^4$};
\node [font=\small] at (11.1,11.55) {$M_{3.5}^4$};
\node [font=\small] at (11,12.3) {$M_4^4$};
\node [font=\small] at (11.05,14.8) {$M_{\delta_2}^4$};
\node [font=\small] at (11.1,13.6) {$\vdots$};

\node [font=\small] at (11,15.75) {$\mathcal{T}^4$};

\sidesemi at (10,8.25)
\upperright at (10,9)
\upperright at (10,9.75)
\upperright at (10,10.5)
\upperright at (10,11.25)
\upperright at (10,12)
\upperright at (10, 14.5)

\node [font=\small] at (10.25,8.25) {$M_0^3$};
\node [font=\small] at (10.25,9.3) {$M_1^3$};
\node [font=\small] at (10.25,10.05) {$M_2^3$};
\node [font=\small] at (10.25,10.8) {$M_3^3$};
\node [font=\small] at (10.35,11.55) {$M_{3.5}^3$};
\node [font=\small] at (10.25,12.3) {$M_4^3$};
\node [font=\small] at (10.3,14.8) {$M_{\delta_2}^3$};
\node [font=\small] at (10.35,13.6) {$\vdots$};

\node [font=\small] at (10.25,15.75) {$\mathcal{T}^3$};

\sidesemi at (9.25,8.25)
\upperright at (9.25,9)
\upperright at (9.25,9.75)
\upperright at (9.25,10.5)
\upperright at (9.25,12)
\upperright at (9.25,14.5)

\node [font=\small] at (9.5,8.25) {$M_0^2$};
\node [font=\small] at (9.5,9.3) {$M_1^2$};
\node [font=\small] at (9.5,10.05) {$M_2^2$};
\node [font=\small] at (9.5,10.8) {$M_3^2$};
\node [font=\small] at (9.5,12.3) {$M_4^2$};
\node [font=\small] at (9.55,14.8) {$M_{\delta_2}^2$};
\node [font=\small] at (9.6,13.6) {$\vdots$};

\node [font=\small] at (9.5,15.75) {$\mathcal{T}^2$};

\sidesemi at (8.5,8.25)
\upperright at (8.5,9)
\upperright at (8.5,9.75)
\upperright at (8.5,10.5)
\upperright at (8.5,12)
\upperright at (8.5,14.5)

\node [font=\small] at (8.75,8.25) {$M_0^1$};
\node [font=\small] at (8.75,9.3) {$M_1^1$};
\node [font=\small] at (8.75,10.05) {$M_2^1$};
\node [font=\small] at (8.75,10.8) {$M_3^1$};
\node [font=\small] at (8.75,12.3) {$M_4^1$};
\node [font=\small] at (8.8,14.8) {$M_{\delta_2}^1$};
\node [font=\small] at (8.85,13.6) {$\vdots$};

\node [font=\small] at (8.75,15.75) {$\mathcal{T}^1$};

\fullcircle at (7.75,8.25)
\uppersemi at (7.75,9)
\uppersemi at (7.75,9.75)
\uppersemi at (7.75,10.5)
\uppersemi at (7.75,12)
\upperleft at (7.75, 11.25)
\uppersemi at (7.75,14.5)

\node [font=\small] at (7.75,15.75) {$\mathcal{T}^0$};
\node [font=\small] at (7.75,8.25) {$M$};
\node [font=\small] at (7.75,9.3) {$M_1^0$};
\node [font=\small] at (7.75,10.05) {$M_2^0$};
\node [font=\small] at (7.75,10.8) {$M_3^0$};
\node [font=\small] at (7.75,12.3) {$M_4^0$};
\node [font=\small] at (7.75,14.8) {$M_{\delta_2}^0$};
\node [font=\small] at (7.75,13.6) {$\vdots$};

\node at (7.2, 9.2)    {$\bullet$};
\node at (7.2, 9.95)   {$\bullet$};
\node at (7.2, 10.7)   {$\bullet$};
\node at (7.2, 11.45)   {$\bullet$};
\node at (10.15, 12.6) {$\bullet$}; 
\node at (7.2, 12.95)   {$\bullet$};

\node at (6.8, 9.2)   {$a_0$};
\node at (6.8, 9.95)   {$a_1$};
\node at (6.8, 10.7)   {$a_2$};
\node at (6.8, 11.45)   {$a_3$};
\node at (10.3, 12.9) {$a_{3.5}$};
\node at (6.8, 12.95)   {$a_4$};

\draw (7.75,9) -- (13.5,9);
\draw (7.75,7.5) -- (13.5,7.5);
\draw (7.75,9.75) -- (13.5,9.75);
\draw (7.75,10.5) -- (13.5,10.5);
\draw (7.75,11.25) -- (13.5,11.25);
\draw (7.75,12   ) -- (13.5,12   );
\draw (7.75,12.75) -- (13.5,12.75);
\draw (7.75,15.25) -- (13.5,15.25);

\draw (7    ,8.25) -- (7    ,14.5);
\draw (8.5  ,8.25) -- (8.5  ,14.5);
\draw (9.25 ,8.25) -- (9.25 ,14.5);
\draw (10   ,8.25) -- (10   ,14.5);
\draw (10.75,8.25) -- (10.75,14.5);
\draw (11.5 ,8.25) -- (11.5 ,14.5);
\draw (14.25,8.25) -- (14.25,14.5);

\end{tikzpicture}
}%
\end{figure}

The diagram captures the intuition of this construction - the $j$th column represents the tower $\calt^j$, and the $i$th row consists of the models of the $i$th levels of each tower. The top row and final column will be continuous and universal, and hence witness that the largest model $M_{\delta_2}^{\delta_1}$ is both a $(\lambda, \delta_1)$ and $(\lambda, \delta_2)$-limit model respectively. As we go further along the chain of towers, new levels will be inserted between the `main' levels of the towers (see where row 3.5 is introduced in $\calt^3$). This will become important when we examine \emph{full towers}, and is why we index towers by arbitrary well orderings $I$ instead of ordinals.

First we will show that strict extensions of towers exist. Towards that goal, we use the following fact, which is based on \cite[13.16(8)]{vaseyn}. We will only need it for Proposition \ref{towerextensionprop*}, Proposition \ref{towerextensionswithb*}, and Proposition \ref{finish_tower_extension}, to verify the types of the $a_i$'s in the extending towers are still non-algebraic.

\begin{lemma}[Disjointness]\label{disjointness-of-dnf}
    If $M, N \in \Kkappalims$, $p \in \gS(N)$ $\dnf$-does not fork over $M$, and $p \upharpoonright M$ is non-algebraic, then $p$ is non-algebraic.
\end{lemma}

\begin{proof}

    The method is essentially \cite[13.16(8)]{vaseyn}. See \cite[3.23]{bemavone} for the details.\end{proof}

The following three results are based on \cite[16.17]{vaseyn} and allow us to extend towers in various ways. The first is similar to \cite[16.17(1)]{vaseyn} but differs in that we take an extension of a whole $\lesst$-increasing chain of towers rather than a single tower. This will be necessary as we cannot take arbitrary unions of towers under our assumptions, unlike \cite{vaseyn}.

\begin{prop}\label{towerextensionprop*}
    Suppose $1 \leq \alpha < \lambda^+$ and $\langle\calt^j : j < \alpha \rangle$ is a $\lesst$-increasing sequence of towers, where $\calt^j = \langle M_i : i \in I^j \rangle ^\wedge \langle a_i : i \in (I^j)^- \rangle$. Suppose $I = \bigcup_{j<\alpha} I^j$ is well-ordered. Then there exists a strongly universal tower $\calt'$ indexed by $I$ such that $\calt^j \lesst \calt'$ for each $j < \alpha$.
\end{prop}

\begin{proof}
    For $i \in I$, use $M^\alpha_i$ to denote $\bigcup \{M^j_i : j < \alpha, i \in I^j\}$, and let $i_0$ be minimal in $I$. Note $M^\alpha_i$ may not be in $\Kkappalims$ if $\alpha$ is a limit, but if $\alpha$ is not a limit, say $\alpha = \beta + 1$, then $M_i^\alpha = M_i^\beta$. We define recursively on $I$ a $\lea$-increasing sequence of models $\langle N_i : i \in I \rangle$ in $\Kkappalims$ and a $\subseteq$-increasing sequence of $\K$-embeddings $\langle f_i: i \in I \rangle$ such that 
    \begin{enumerate}
        \item for all $i \in I$, $f_i : M_i^\alpha \rightarrow N_i$
        \item $f_{i_0} = \id_{M_{i_0}^\alpha}$
        \item for all $i \in I$, $N_i$ is universal over $f_i[M_i^\alpha]$
        \item for all $i \in I \setminus \{i_0\}$, $N_i$ is universal over $\bigcup_{j<i}N_j$
        \item for all $j<\alpha$ and $i \in (I^j)^-$, $\gtp(f_{i+1}(a_i)/N_i, N_{i+1})$ $\dnf$-does not fork over $f_i[M^j_i]$.
        \end{enumerate}

    \textbf{This is enough:} Let $f = \bigcup_{i\in I}f_i : \bigcup_{i \in I}M_i^\alpha\rightarrow \bigcup_{i\in I}N_i$, and extend this to an isomorphism $g:M' \cong \bigcup_{i \in I }N_i$. Let $M_i' = g^{-1}[N_i]$ for each $i < \alpha$. Let $\calt' = \langle M_i' : i \in I \rangle ^\wedge \langle a_i : i \in I^- \rangle$. 

    By (5), for all $j < \alpha$ and $i \in (I^j)^-$, $\gtp({a_i}/M_i', M_{i+1}')$ $\dnf$-does not fork over $M^j_i$. For each $i \in I^-$ there is $j < \alpha$ where $i, i +_I 1 \in I^j$ and $a_i \notin M^j_i$, so by Lemma \ref{disjointness-of-dnf} we have that $a_i \in |M_{i+1}'| \setminus |M_i'|$. So $\calt'$ is a tower, and by (4) it is a strong limit tower. By (3), for all $j < \alpha$, for every $i \in I^j$, $M_i'$ is universal over $M^j_i$. So $\calt^j \lesst \calt'$ as desired.

    \textbf{This is possible:} For \underline{$i=i_0$}, we can take $N_{i_0}$ $(\lambda, \kappa)$-limit over $M_{i_0}^\alpha$ (hence universal over $M_{i_0}^\alpha$ also), and $f_{i_0} = \id_{M^\alpha_{i_0}}$ as specified by (2).

    For the \underline{successor step}, suppose we have $f_{i'}:M_{i'}^\alpha \rightarrow N_{i'}$ defined for $i' \leq i \in I$. Extend $f_i$ to an embedding $\bar{f}_i : M^\alpha_{i+1} \rightarrow \bar{N}_i$ where $N_i \leq \bar{N}_i$ (this is possible by AP).

    Let $j_i$ be the least $j < \alpha$ such that $i \in (I^j)^-$. For $j \in [j_i, \alpha]$, let $q_j = \gtp(\bar{f}_i(a_i)/f_i[M^j_i], \bar{f}_{i}[M_{i+_{I^j}1}^j])$. The tower ordering tells us that $q_j$ $\dnf$-does not fork over $f_i[M_i^{j_i}]$ for $j \in [j_i, \alpha)$, so in particular these types all extend $q_{j_i}$ and increase with $j$ by uniqueness.

    By extension, there is $p \in \gS(N_i)$ extending $q_{j_i}$ that does not fork over $f_i[M^{j_i}_i]$. By uniqueness, $p \upharpoonright f_i[M_i^j] = q_j$ for all $j \in [j_i, \alpha)$. If $\alpha$ is limit, by $\Kkappalims$-universal continuity*, $p \upharpoonright f_i[M_i^\alpha] = q_\alpha$ (note that $\Kkappalims$-universal continuity* does not impose restrictions on the cofinality of $\alpha$). If $\alpha$ is not limit, then $\alpha = \beta + 1$ for some $\beta$, and we have $M_i^\alpha = \bigcup_{j<\alpha}M_i^j = M_i^\beta$, and $p \upharpoonright f_i[M_i^\alpha] = q_{\beta} = q_\alpha$ by uniqueness. In either case, $p \upharpoonright M_i^\alpha = q_\alpha$.

    Say $p = \gtp(b/N_i, \hat{N}_i)$. Using $p \upharpoonright f_i[M_i^\alpha] = q_\alpha$ and AP, there is an embedding $g_i:\bar{N}_i \rightarrow N_{i+1}'$ where $\hat{N}_i \lea N_{i+1}'$ and $g_i(\bar{f}_i(a_i)) = b$. Let $N_{i+1}$ be any $(\lambda, \kappa)$-limit over $N_{i+1}'$. Set $f_{i+1} = g_i \circ \bar{f}_i:M_{i+1}^\alpha \rightarrow N_{i+1}$.

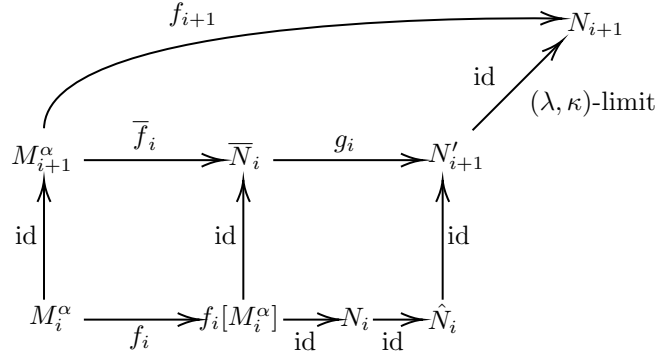
\begin{center}

\tikzset{every picture/.style={line width=0.75pt}} 

\begin{tikzpicture}[x=0.75pt,y=0.75pt,yscale=-1,xscale=1]

\draw    (80,201) -- (80,143) ;
\draw [shift={(80,141)}, rotate = 90] [color={rgb, 255:red, 0; green, 0; blue, 0 }  ][line width=0.75]    (10.93,-3.29) .. controls (6.95,-1.4) and (3.31,-0.3) .. (0,0) .. controls (3.31,0.3) and (6.95,1.4) .. (10.93,3.29)   ;
\draw    (100,211) -- (156,211) ;
\draw [shift={(158,211)}, rotate = 180] [color={rgb, 255:red, 0; green, 0; blue, 0 }  ][line width=0.75]    (10.93,-3.29) .. controls (6.95,-1.4) and (3.31,-0.3) .. (0,0) .. controls (3.31,0.3) and (6.95,1.4) .. (10.93,3.29)   ;
\draw    (180,201) -- (180,143) ;
\draw [shift={(180,141)}, rotate = 90] [color={rgb, 255:red, 0; green, 0; blue, 0 }  ][line width=0.75]    (10.93,-3.29) .. controls (6.95,-1.4) and (3.31,-0.3) .. (0,0) .. controls (3.31,0.3) and (6.95,1.4) .. (10.93,3.29)   ;
\draw    (100,131) -- (168,131) ;
\draw [shift={(170,131)}, rotate = 180] [color={rgb, 255:red, 0; green, 0; blue, 0 }  ][line width=0.75]    (10.93,-3.29) .. controls (6.95,-1.4) and (3.31,-0.3) .. (0,0) .. controls (3.31,0.3) and (6.95,1.4) .. (10.93,3.29)   ;
\draw    (245,211) -- (268,211) ;
\draw [shift={(270,211)}, rotate = 180] [color={rgb, 255:red, 0; green, 0; blue, 0 }  ][line width=0.75]    (10.93,-3.29) .. controls (6.95,-1.4) and (3.31,-0.3) .. (0,0) .. controls (3.31,0.3) and (6.95,1.4) .. (10.93,3.29)   ;
\draw    (200,211) -- (226,211) ;
\draw [shift={(228,211)}, rotate = 180] [color={rgb, 255:red, 0; green, 0; blue, 0 }  ][line width=0.75]    (10.93,-3.29) .. controls (6.95,-1.4) and (3.31,-0.3) .. (0,0) .. controls (3.31,0.3) and (6.95,1.4) .. (10.93,3.29)   ;
\draw    (280,201) -- (280,143) ;
\draw [shift={(280,141)}, rotate = 90] [color={rgb, 255:red, 0; green, 0; blue, 0 }  ][line width=0.75]    (10.93,-3.29) .. controls (6.95,-1.4) and (3.31,-0.3) .. (0,0) .. controls (3.31,0.3) and (6.95,1.4) .. (10.93,3.29)   ;
\draw    (195,131) -- (268,131) ;
\draw [shift={(270,131)}, rotate = 180] [color={rgb, 255:red, 0; green, 0; blue, 0 }  ][line width=0.75]    (10.93,-3.29) .. controls (6.95,-1.4) and (3.31,-0.3) .. (0,0) .. controls (3.31,0.3) and (6.95,1.4) .. (10.93,3.29)   ;
\draw    (295,116) -- (338.59,72.41) ;
\draw [shift={(340,71)}, rotate = 135] [color={rgb, 255:red, 0; green, 0; blue, 0 }  ][line width=0.75]    (10.93,-3.29) .. controls (6.95,-1.4) and (3.31,-0.3) .. (0,0) .. controls (3.31,0.3) and (6.95,1.4) .. (10.93,3.29)   ;
\draw    (80,115) .. controls (81.65,62.25) and (278.99,59.06) .. (338.25,59.97) ;
\draw [shift={(340,60)}, rotate = 181.01] [color={rgb, 255:red, 0; green, 0; blue, 0 }  ][line width=0.75]    (10.93,-3.29) .. controls (6.95,-1.4) and (3.31,-0.3) .. (0,0) .. controls (3.31,0.3) and (6.95,1.4) .. (10.93,3.29)   ;

\draw (71,202.4) node [anchor=north west][inner sep=0.75pt]    {$M_{i}^{\alpha }$};
\draw (62,122.4) node [anchor=north west][inner sep=0.75pt]    {$M_{i+1}^{\alpha }$};
\draw (156,202.4) node [anchor=north west][inner sep=0.75pt]    {$f_i[M_i^\alpha]$};
\draw (171,120.4) node [anchor=north west][inner sep=0.75pt]    {$\overline{N}_{i}$};
\draw (227,203.4) node [anchor=north west][inner sep=0.75pt]    {$N_{i}$};
\draw (271,201.4) node [anchor=north west][inner sep=0.75pt]    {$\hat{N}_{i}$};
\draw (271,120.4) node [anchor=north west][inner sep=0.75pt]    {$N'_{i+1}$};
\draw (341,55.4) node [anchor=north west][inner sep=0.75pt]    {$N_{i+1}$};
\draw (123,109.4) node [anchor=north west][inner sep=0.75pt]    {$\overline{f}_{i}$};
\draw (121,212.4) node [anchor=north west][inner sep=0.75pt]    {$f_{i}$};
\draw (64,162.4) node [anchor=north west][inner sep=0.75pt]    {$\id$};
\draw (224,117.4) node [anchor=north west][inner sep=0.75pt]    {$g_i$};
\draw (204,215.4) node [anchor=north west][inner sep=0.75pt]    {$\id$};
\draw (248,215.4) node [anchor=north west][inner sep=0.75pt]    {$\id$};
\draw (281,162.4) node [anchor=north west][inner sep=0.75pt]    {$\id$};
\draw (294,82.4) node [anchor=north west][inner sep=0.75pt]    {$\id$};
\draw (141,50.4) node [anchor=north west][inner sep=0.75pt]    {$f_{i+1}$};
\draw (322,94.4) node [anchor=north west][inner sep=0.75pt]    {$( \lambda ,\kappa )$-limit};
\draw (164,162.4) node [anchor=north west][inner sep=0.75pt]    {$\id$};

\end{tikzpicture}

\end{center}

    Note, $f_{i+1}$ extends $f_i$, and (5) holds: since $g_i(\bar{f}_i(a_i)) = b$, we have $\gtp(f_{i+1}(a_i)/N_i, N_{i+1}) = \gtp(g_i(\bar{f}_i(a_i))/N_i, N_{i+1}) = \gtp(b/N_i, \hat{N}_i) = p$, which $\dnf$-does not fork over $f_i[M_i^{j_i}]$. Hence $\gtp(f_{i+1}(a_i)/N_i, N_{i+1})$ $\dnf$-does not fork over $f_i[M_i^{j}]$ for all $j \in [j_i, \alpha)$ by base monotonicity. So $f_{i+1}$ fulfills the inductive hypothesis.

    At \underline{limit steps} the proof is simpler as we do not have to worry about non-forking of the $a_i$'s. Given $i \in I$ limit and $f_{i'}$ for all $i' < i \in I$ satisfying the inductive hypothesis, let $\hat{f}_i = \bigcup_{i' < i} f_{i'} : \bigcup_{i' < i}M_{i'}^\alpha \rightarrow \bigcup_{i' < i}N_{i'}$. Extend this to an embedding $f_i:M_i^\alpha \rightarrow N^0_i$ where $\bigcup_{i' < i} N_{i'} \lea N^0_i$, using AP. Now find $N_i$ a $(\lambda, \kappa)$-limit over $N^0_i$. Then $f_i:M_i^\alpha \rightarrow N_i$ satisfies the induction hypothesis. This completes the recursion, and the proof. \end{proof}

The following is the analogue in our context of \cite[16.17(2)]{vaseyn}. It is the only place we will use non-forking amalgamation of $\dnf$. The proof is very similar to the last result - in fact, in \cite{vaseyn} they are collected into a single proof. However we keep the results separate to avoid overcomplicating the statement - unlike Proposition \ref{towerextensionprop*}, here we only need to extend a single tower. 

We assume the index is an ordinal, but note it is also true (with some relabeling) for a general well ordered set $I$, as are all the following results where the index is an ordinal - we only use ordinals for notational convenience.

\begin{prop}\label{towerextensionswithb*}
    Suppose $\calt = \langle M_i : i < \beta \rangle ^\wedge \langle a_i : i \in \beta^- \rangle$ is a tower. Suppose in addition $p \in \gS(M_0)$. Then there exists a strongly universal tower $\calt' = \langle M_i' : i< \beta \rangle ^\wedge \langle a_i : i \in \beta^- \rangle$ and $b \in M_0'$ such that $\calt \lesst \calt'$, $\gtp(b/M_0, M_0') = p$, and for all $i < \beta$, $\gtp(b/M_i, M_i')$ $\dnf$-does not fork over $M_0$.
\end{prop}

\begin{proof}
    Let $N \in \K_\lambda$ and $c \in N$ be such that $p = \gtp(c/M_0, N)$. As before, we recursively define a $\lea$-increasing sequence of models $\langle N_i : i < \beta \rangle$ and a $\subseteq$-increasing sequence of $\K$-embeddings $\langle f_i: i < \beta \rangle$ such that 
    \begin{enumerate}
        \item for all $i < \beta$, $f_i : M_i \rightarrow N_i$
        \item $f_0 = \id_{M_0}$ (in particular, $M_0 \lea N_0$), $N \lea N_0$, and $c \in N_0$
        \item for all $i < \beta$, $N_i$ is universal over $f_i[M_i]$
        \item for all $i \in [1, \beta)$, $N_i$ is universal over $\bigcup_{r<i}N_r$
        \item for all $i \in \beta^-$, $\gtp(f_{i+1}(a_i)/N_i, N_{i+1})$ $\dnf$-does not fork over $f_i[M_i]$
        \item for all $i < \beta$, $\gtp(c/{f_i[M_i]}, N_i)$ $\dnf$-does not fork over $M_0$.
    \end{enumerate}

    \textbf{This is enough:} As before, extending $\bigcup_{i<\beta}f_i$ to an isomorphism $g:M' \cong \bigcup_{i<\beta} N_i$, setting $M_i' = g^{-1}[N_i]$ for $i < \beta$, and taking $\calt' = \langle M_i' : i < \beta \rangle ^\wedge \langle a_i : i \in \beta^- \rangle$, we get $\calt'$ is a tower and $\calt \lesst \calt'$ (just as in the last proof, with $\alpha = 1$). Further if we set $b = g^{-1}(c)$, we have $b \in M_0'$, $p = \gtp(b/M_0, M_0')$, and by condition (6), $\gtp(b/M_i, M_i')$ $\dnf$-does not fork over $M_0$ for all $i < \beta$ as desired.

    \textbf{This is possible:} For \underline{$i=0$}, we can take $N_0$ a $(\lambda, \kappa)$-limit over $N$ (hence universal over $M_0$ also), and $f_0 = \id_{M_0}$ as specified.

    For the \underline{successor step}, suppose we have defined $N_i$ and $f_i$ satisfying conditions (1)-(6), and must define $f_{i+1}$ and $N_{i+1}$. By non-forking amalgamation of $\dnf$, there exists $N_{i+1}' \in \Kkappalims$ and $f_{i+1}:M_{i+1} \rightarrow N_{i+1}'$ such that $N_i \lea N_{i+1}'$, $\gtp(f_{i+1}(a_i)/ N_i, N_{i+1}')$ $\dnf$-does not fork over $f_i[M_i]$, and $\gtp(c/ f_{i+1}[M_{i+1}], N_{i+1})$ $\dnf$.
    
    Take $N_{i+1}$ a $(\lambda, \kappa)$-limit over $N_{i+1}'$, hence also over both $f_{i+1}[M_{i+1}]$ and $N_i$.

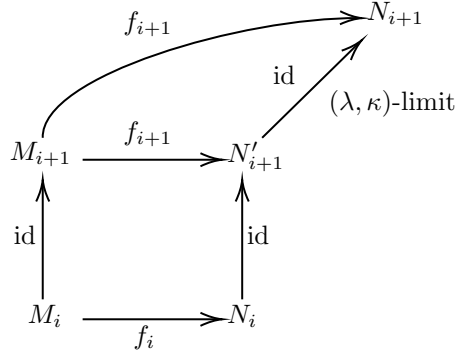
\begin{center}

\tikzset{every picture/.style={line width=0.75pt}} 

\begin{tikzpicture}[x=0.75pt,y=0.75pt,yscale=-1,xscale=1]

\draw    (80,201) -- (80,143) ;
\draw [shift={(80,141)}, rotate = 90] [color={rgb, 255:red, 0; green, 0; blue, 0 }  ][line width=0.75]    (10.93,-3.29) .. controls (6.95,-1.4) and (3.31,-0.3) .. (0,0) .. controls (3.31,0.3) and (6.95,1.4) .. (10.93,3.29)   ;
\draw    (100,211) -- (168,211) ;
\draw [shift={(170,211)}, rotate = 180] [color={rgb, 255:red, 0; green, 0; blue, 0 }  ][line width=0.75]    (10.93,-3.29) .. controls (6.95,-1.4) and (3.31,-0.3) .. (0,0) .. controls (3.31,0.3) and (6.95,1.4) .. (10.93,3.29)   ;
\draw    (180,201) -- (180,143) ;
\draw [shift={(180,141)}, rotate = 90] [color={rgb, 255:red, 0; green, 0; blue, 0 }  ][line width=0.75]    (10.93,-3.29) .. controls (6.95,-1.4) and (3.31,-0.3) .. (0,0) .. controls (3.31,0.3) and (6.95,1.4) .. (10.93,3.29)   ;
\draw    (100,131) -- (168,131) ;
\draw [shift={(170,131)}, rotate = 180] [color={rgb, 255:red, 0; green, 0; blue, 0 }  ][line width=0.75]    (10.93,-3.29) .. controls (6.95,-1.4) and (3.31,-0.3) .. (0,0) .. controls (3.31,0.3) and (6.95,1.4) .. (10.93,3.29)   ;
\draw    (190,120) -- (238.57,72.4) ;
\draw [shift={(240,71)}, rotate = 135.58] [color={rgb, 255:red, 0; green, 0; blue, 0 }  ][line width=0.75]    (10.93,-3.29) .. controls (6.95,-1.4) and (3.31,-0.3) .. (0,0) .. controls (3.31,0.3) and (6.95,1.4) .. (10.93,3.29)   ;
\draw    (80,120) .. controls (79.01,88.81) and (180.93,59.59) .. (238.28,59.98) ;
\draw [shift={(240,60)}, rotate = 181.01] [color={rgb, 255:red, 0; green, 0; blue, 0 }  ][line width=0.75]    (10.93,-3.29) .. controls (6.95,-1.4) and (3.31,-0.3) .. (0,0) .. controls (3.31,0.3) and (6.95,1.4) .. (10.93,3.29)   ;

\draw (71,202.4) node [anchor=north west][inner sep=0.75pt]    {$M_{i}$};
\draw (62,120.4) node [anchor=north west][inner sep=0.75pt]    {$M_{i+1}$};
\draw (171,202.4) node [anchor=north west][inner sep=0.75pt]    {$N_{i}$};
\draw (171,120.4) node [anchor=north west][inner sep=0.75pt]    {$N'_{i+1}$};
\draw (241,50.4) node [anchor=north west][inner sep=0.75pt]    {$N_{i+1}$};
\draw (64,162.4) node [anchor=north west][inner sep=0.75pt]    {$\id$};
\draw (124,212.4) node [anchor=north west][inner sep=0.75pt]    {$f_{i}$};
\draw (181,162.4) node [anchor=north west][inner sep=0.75pt]    {$\id$};
\draw (194,82.4) node [anchor=north west][inner sep=0.75pt]    {$\id$};
\draw (119,110.4) node [anchor=north west][inner sep=0.75pt]    {$f_{i+1}$};
\draw (119,55.4) node [anchor=north west][inner sep=0.75pt]    {$f_{i+1}$};
\draw (222,94.4) node [anchor=north west][inner sep=0.75pt]    {$( \lambda ,\kappa )$-limit};

\end{tikzpicture}

\end{center}

    By monotonicity, $\gtp(f_{i+1}(a_i)/ N_i, N_{i+1})$ $\dnf$-does not fork over $f_i[M_i]$, and $\gtp(c/ f_{i+1}[M_{i+1}], N_{i+1})$ $\dnf$-does not fork over $f_i[M_i]$. Since $\gtp(c/ f_i[M_i], N_{i+1})$ $\dnf$-does not fork over $M_0$ by the induction hypothesis and monotonicity, $\gtp(c/f_{i+1}[M_{i+1}], N_{i+1})$ $\dnf$-does not fork over $M_0$ by transitivity. At this point we have shown $N_{i+1}$ and $f_{i+1}:M_{i+1} \rightarrow N_{i+1}$ satisfy the relevant conditions (1)-(6).

    For \underline{limit ordinals} $i$, we have $N_r$ and $f_r$ for all $r < i$ and must construct $N_i, f_i$. Let $M_i^0 = \bigcup_{r<i} M_r$, $N_i^0 = \bigcup_{r<i} N_r$ and $f_i^0 = \bigcup_{r<i} f_r : M_i^0 \rightarrow N_i^0$. Let $g_i:M_i \rightarrow N_i'$ be any extension of $f_i^0$ with domain $M_i$ and where $N_i^0 \lea N_i'$ (this exists by AP). Take any $q \in \gS(g_i[M_i])$ extending $p$ which $\dnf$-does not fork over $M_0$. Since $\gtp(c/g_i[M_r], N_i^0)$ $\dnf$-does not fork over $M_0$ for all $r < i$, by uniqueness $q \upharpoonright g_i[M_r] = \gtp(c/g_i[M_r], N_i^0)$ for each $r<i$. By $\Kkappalims$-universal continuity*, $q \upharpoonright g_i[M_i^0] = \gtp(c/g_i[M_i^0], N_i^0)$. Let $\hat{N}_i \in \K_\lambda$  and $d \in \hat{N}_i$ be such that $N_i' \lea \hat{N}_i$ and $q = \gtp(d/g_i[M_i], \hat{N}_i)$. Then by type equality there is some $\bar{N_i} \in \K_\lambda$ where $N_i' \lea \bar{N}_i$ and some $h_i : \hat{N}_i \rightarrow \bar{N}_i$ such that $h_i$ is the identity on $g_i[M_i^0]$ and $h_i(d) = c$. Let $N_i$ be a $(\lambda, \kappa)$-limit over $\bar{N}_i$.

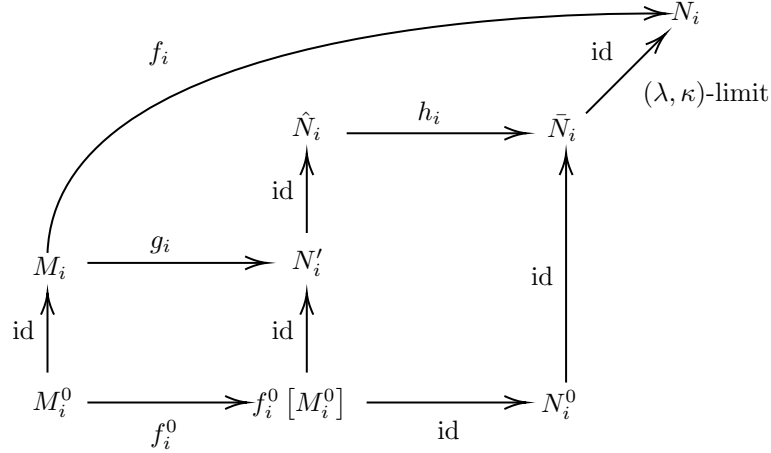
\begin{center}

\tikzset{every picture/.style={line width=0.75pt}} 

\begin{tikzpicture}[x=0.75pt,y=0.75pt,yscale=-1,xscale=1]

\draw    (140,265) -- (218,265) ;
\draw [shift={(220,265)}, rotate = 180] [color={rgb, 255:red, 0; green, 0; blue, 0 }  ][line width=0.75]    (10.93,-3.29) .. controls (6.95,-1.4) and (3.31,-0.3) .. (0,0) .. controls (3.31,0.3) and (6.95,1.4) .. (10.93,3.29)   ;
\draw    (250,250) -- (250,212) ;
\draw [shift={(250,210)}, rotate = 90] [color={rgb, 255:red, 0; green, 0; blue, 0 }  ][line width=0.75]    (10.93,-3.29) .. controls (6.95,-1.4) and (3.31,-0.3) .. (0,0) .. controls (3.31,0.3) and (6.95,1.4) .. (10.93,3.29)   ;
\draw    (120,250) -- (120,212) ;
\draw [shift={(120,210)}, rotate = 90] [color={rgb, 255:red, 0; green, 0; blue, 0 }  ][line width=0.75]    (10.93,-3.29) .. controls (6.95,-1.4) and (3.31,-0.3) .. (0,0) .. controls (3.31,0.3) and (6.95,1.4) .. (10.93,3.29)   ;
\draw    (140,195) -- (228,195) ;
\draw [shift={(230,195)}, rotate = 180] [color={rgb, 255:red, 0; green, 0; blue, 0 }  ][line width=0.75]    (10.93,-3.29) .. controls (6.95,-1.4) and (3.31,-0.3) .. (0,0) .. controls (3.31,0.3) and (6.95,1.4) .. (10.93,3.29)   ;
\draw    (250,180) -- (250,142) ;
\draw [shift={(250,140)}, rotate = 90] [color={rgb, 255:red, 0; green, 0; blue, 0 }  ][line width=0.75]    (10.93,-3.29) .. controls (6.95,-1.4) and (3.31,-0.3) .. (0,0) .. controls (3.31,0.3) and (6.95,1.4) .. (10.93,3.29)   ;
\draw    (270,130) -- (358,130) ;
\draw [shift={(360,130)}, rotate = 180] [color={rgb, 255:red, 0; green, 0; blue, 0 }  ][line width=0.75]    (10.93,-3.29) .. controls (6.95,-1.4) and (3.31,-0.3) .. (0,0) .. controls (3.31,0.3) and (6.95,1.4) .. (10.93,3.29)   ;
\draw    (390,120) -- (428.59,81.41) ;
\draw [shift={(430,80)}, rotate = 135] [color={rgb, 255:red, 0; green, 0; blue, 0 }  ][line width=0.75]    (10.93,-3.29) .. controls (6.95,-1.4) and (3.31,-0.3) .. (0,0) .. controls (3.31,0.3) and (6.95,1.4) .. (10.93,3.29)   ;
\draw    (280,265) -- (358,265) ;
\draw [shift={(360,265)}, rotate = 180] [color={rgb, 255:red, 0; green, 0; blue, 0 }  ][line width=0.75]    (10.93,-3.29) .. controls (6.95,-1.4) and (3.31,-0.3) .. (0,0) .. controls (3.31,0.3) and (6.95,1.4) .. (10.93,3.29)   ;
\draw    (380,255) -- (380,142) ;
\draw [shift={(380,140)}, rotate = 90] [color={rgb, 255:red, 0; green, 0; blue, 0 }  ][line width=0.75]    (10.93,-3.29) .. controls (6.95,-1.4) and (3.31,-0.3) .. (0,0) .. controls (3.31,0.3) and (6.95,1.4) .. (10.93,3.29)   ;
\draw    (120,190) .. controls (124.31,82.87) and (343.46,69.14) .. (428.72,69.99) ;
\draw [shift={(430,70)}, rotate = 180.68] [color={rgb, 255:red, 0; green, 0; blue, 0 }  ][line width=0.75]    (10.93,-3.29) .. controls (6.95,-1.4) and (3.31,-0.3) .. (0,0) .. controls (3.31,0.3) and (6.95,1.4) .. (10.93,3.29)   ;

\draw (111,256.4) node [anchor=north west][inner sep=0.75pt]    {$M_{i}^{0}$};
\draw (221,256.4) node [anchor=north west][inner sep=0.75pt]    {$f_{i}^{0}\left[ M_{i}^{0}\right]$};
\draw (241,186.4) node [anchor=north west][inner sep=0.75pt]    {$N'_{i}$};
\draw (111,190.4) node [anchor=north west][inner sep=0.75pt]    {$M_{i}$};
\draw (241,117.4) node [anchor=north west][inner sep=0.75pt]    {$\hat{N}_{i}$};
\draw (366,257.4) node [anchor=north west][inner sep=0.75pt]    {$N_{i}^{0}$};
\draw (369,119.4) node [anchor=north west][inner sep=0.75pt]    {$\bar{N}_{i}$};
\draw (431,62.4) node [anchor=north west][inner sep=0.75pt]    {$N_{i}$};
\draw (101,222.4) node [anchor=north west][inner sep=0.75pt]    {$\id$};
\draw (231,222.4) node [anchor=north west][inner sep=0.75pt]    {$\id$};
\draw (231,152.4) node [anchor=north west][inner sep=0.75pt]    {$\id$};
\draw (361,195.4) node [anchor=north west][inner sep=0.75pt]    {$\id$};
\draw (391,82.4) node [anchor=north west][inner sep=0.75pt]    {$\id$};
\draw (417,102.4) node [anchor=north west][inner sep=0.75pt]    {$( \lambda ,\kappa )$-limit};
\draw (170,272.4) node [anchor=north west][inner sep=0.75pt]    {$f_{i}^{0}$};
\draw (170,180.4) node [anchor=north west][inner sep=0.75pt]    {$g_{i}$};
\draw (314,272.4) node [anchor=north west][inner sep=0.75pt]    {$\id$};
\draw (304,112.4) node [anchor=north west][inner sep=0.75pt]    {$h_{i}$};
\draw (169,82.4) node [anchor=north west][inner sep=0.75pt]    {$f_{i}$};

\end{tikzpicture}

\end{center}

    Then setting $f_i = h_i \circ g_i : M_i \rightarrow N_i$, we have that 
    \begin{align*}
        \gtp(c/f_i[M_i], N_i) = \gtp(h_i(d)/h_i[g_i[M_i]], N_i) &= \gtp(h_i(d)/h_i[g_i[M_i]], h_i[\hat{N_i}]) \\
        &= h_i(\gtp(d/ g_i[M_i], \hat{N_i})) = h_i(q)
    \end{align*}
    which $\dnf$-does not fork over $h_i[M_0] = M_0$, satisfying (6). Further, $f_i$ extends $f_{i'}$ for all $i' < i$ as $g_i$ extends $f_i^0$ and $h_i$ fixes $f_i^0[M_i^0]$.  As $N_i$ is $(\lambda, \kappa)$-limit over $\bar{N}_i$, which contains both $N_i^0$ and $f_i[M_i]$, (3) and (4) are satisfied. So all the relevant conditions hold up to $f_i$. This completes the recursion, and the proof.
\end{proof}

\begin{remark}
    In \cite{vasey18}, it is also assumed $\calt$ is universal and continuous for Proposition \ref{towerextensionswithb*}. It appears this assumption is unnecessary. In fact for our proof of Theorem \ref{largelimitsareisothm*} we will need to apply it to possibly non-continuous towers, as the construction will only guarantee continuity of towers at ordinals of high cofinality.
\end{remark}

We need one final version of tower extension - the ability to complete a partial extension on an initial segment of the tower.

\begin{prop}\label{finish_tower_extension}
    Suppose $\calt = \langle M_i : i < \beta \rangle ^\wedge \langle a_i : i \in \beta^- \rangle$ is a tower. Suppose there is $\gamma < \beta$ and $\calt^* = \langle M_i^* : i < \gamma \rangle ^\wedge \langle a_i : i \in \gamma^- \rangle$ indexed by $\gamma$ such that $\calt \upharpoonright \gamma \lesst \calt^*$. Then there exists a tower $\calt' = \langle M_i' : i< \beta \rangle ^\wedge \langle a_i : i \in \beta^- \rangle$ and some $N$, $g : \bigcup_{i < \beta} M_i \underset{\bigcup_{i < \gamma} M_i}{\longrightarrow} N$ such that $\calt \lesst \calt'$ and $g[\calt'] \upharpoonright \gamma = \calt^*$ (where $g[\calt'] = \langle g[M_i'] : i< \beta \rangle ^\wedge \langle a_i : i \in \beta^- \rangle$).
\end{prop}

\begin{proof}
    Without loss of generality, $\gamma \geq 1$ (the $\gamma = 0$ case is a just Proposition \ref{towerextensionprop*} with $\alpha = 1$). We again follow the blueprint of Proposition \ref{towerextensionprop*} and Proposition \ref{towerextensionswithb*}.
    
    Similar to before, by recursion construct a $\lea$-increasing sequence of models $\langle N_i : i < \beta \rangle$ and a $\subseteq$-increasing sequence of $\K$-embeddings $\langle f_i: i < \beta \rangle$ such that 
    \begin{enumerate}
        \item for all $i < \beta$, $f_i : M_i \rightarrow N_i$
        \item for $i < \gamma$, $N_i = M_i^*$ and $f_i = \id_{M_i}$
        \item for all $i < \beta$, $N_i$ is universal over $f_i[M_i]$
        \item for all $i \in \beta^-$, $\gtp(f_{i+1}(a_i)/N_i, N_{i+1})$ $\dnf$-does not fork over $M_i$.
    \end{enumerate}

    Rather than just fixing $f_0$ to be the identity on $M_0$, we now ensure that all the $f_i$ for $i < \gamma$ are the identity, so $\bigcup_{i<\beta} f_i$ will map $\calt \upharpoonright \gamma$ into $\calt^*$. This determines the first $\gamma$ steps of the construction. From there, proceed as in Proposition \ref{towerextensionprop*} (with $\alpha = 1$) for the remaining successor and limit $i < \beta$. After the construction is complete, as in Proposition \ref{towerextensionswithb*} take $f = \bigcup_{i < \beta} f_i : \bigcup_{i<\beta} M_i \rightarrow \bigcup_{i < \beta} N_i$, extend to $g : M' \cong \bigcup_{i < \beta} N_i$, and take $M_i' = f^{-1}[N_i]$ for $i < \beta$. As before take  $\calt' = \langle M_i' : i < \beta \rangle ^\wedge \langle a_i : i < \beta^- \rangle$. We have again that $\calt \lesst \calt'$, and additionally $g(a_i) = a_i$ for $i \in \gamma^-$ and $g[M_i'] = M_i^*$ for $i < \gamma$, meaning $g[\calt'] \upharpoonright \gamma = \calt^*$ as claimed.
\end{proof}

\subsection{Reduced towers and full towers}
We assume Hypothesis \ref{long-limit-dnf-hypotheses} throughout this subsection. In the previous subsection, we have shown we can build chains of towers of any length. We now focus on showing how we may guarantee that the final tower in the chain will be `continuous at $\delta_2$', and will be `universal'. Note these properties will not necessarily be preserved by tower unions. So we will introduce stronger conditions that will be preserved by tower unions. First we address continuity with reduced towers.

\begin{defin}\label{reduced-def}
    Let $\calt = \langle M_i : i \in I \rangle ^\wedge \langle a_i : i \in I^- \rangle$ be a tower. We say $\calt$ is \emph{reduced} if for every tower $\calt' = \langle M_i' : i \in I' \rangle ^\wedge \langle a_i : i \in (I')^- \rangle$ such that $\calt \lesst \calt'$, for all $r <_I s \in I$, we have $M_r' \cap M_s = M_r$.
\end{defin}

\begin{remark}
    The definition is equivalent if we require that $\calt'$ is also indexed by $I$, since the condition is true for $\calt'$ if and only if it holds for $\calt' \upharpoonright I$. 
\end{remark}

Reduced towers behave well with high cofinality unions:

\begin{lemma}\label{reducedunionsarereduced*}
    Suppose $\cof(\delta) \geq \kappa$. If $\langle\calt^j : j < \delta \rangle$ is a $\lesst$-increasing chain of reduced towers with $\calt^j$ indexed by $I^j$, and $\bigcup_{j<\delta}I^j$ is a well ordering. Then $\bigcup_{j<\delta} \calt^j$ is reduced.
\end{lemma}

\begin{proof}
    The method is essentially the same as \cite[5.13]{vasey18}. See \cite[3.30]{bemavone} for the details. \end{proof}

\begin{lemma}\label{reducedextensionsexist*}
    If $\calt = \langle M_i : i \in I \rangle ^\wedge \langle a_i : i \in I^- \rangle$ is a tower, there exists a reduced tower $\calt' = \langle M_i : i \in I \rangle ^\wedge \langle a_i : i \in I^- \rangle$ such that $\calt \lesst \calt'$.
\end{lemma}

\begin{proof}
    Suppose no such $\calt'$ exists for contradiction. Then construct by recursion a $\lesst$-increasing sequence of towers $\langle \calt^j : j < \lambda^+ \rangle$ where $\calt^j = \langle M_i^j : i \in I \rangle ^\wedge \langle a_i^j : i \in I^- \rangle$ for all $j < \lambda^+$ such that 
    \begin{enumerate}
        \item $\calt \lesst \calt^0$
        \item for all $j < \lambda^+$, there are $r < s \in I$ such that $M^{j+1}_r \cap M^j_s \neq M^j_r$
        \item for all $j < \lambda^+$ such that $\cof(j) \geq \kappa$, $\calt^j = \bigcup_{j' < j} \calt^{j'}$.
    \end{enumerate}

    $\calt^0$ can be found by Proposition \ref{towerextensionprop*}, and $\calt^j$ for limits $j$ where $\cof(j) \geq \kappa$ are determined. For other limits, take any $\calt^j$ where $\calt^{j'} \lesst \calt^j$ for all $j' < j$, which exists by Proposition \ref{towerextensionprop*}. Finally, for successors, if $\calt^j$ is given, since $\calt \lesst \calt^j$, $\calt^j$ is not reduced. So there is $\calt^{j+1} = \langle M_i^{j+1} : i \in I \rangle ^\wedge \langle a_i^{j+1} : i \in I^- \rangle$ such that for some $r < s \in I$, $M^{j+1}_r \cap M^j_s \neq M^j_r$. This completes the construction.

    We want this sequence of towers to be truly continuous for the argument we will use, but small unions of towers may not result in a tower (we could lose that $M^j_i \in \Kkappalims$). Nevertheless, define $N^j_i$ to be
    \begin{itemize}
        \item $M^j_i$ if $j$ is not a limit
        \item $\bigcup_{j'<j} M^{j'}_i$ if $j$ is a limit (note this means $N^j_i = M^j_i$ if $\cof(j) \geq \kappa$).
    \end{itemize}
    Again, $\langle N^j_i : i \in I\rangle$ cannot necessarily be used to form a tower, but are close enough to the $M^j_i$'s for the following argument to work (since $N^j_i \lea M^j_i \lea N^{j+1}_i$ for all $j < \delta$, and $N^j_i = M^j_i$ whenever $\cof(j) \geq \kappa$).

    For notational convenience, let $N^{\lambda^+}_i = \bigcup_{j < \lambda^+} N^j_i$ for $i \in I$ and $N^j_I = \bigcup_{i \in I} N^j_i$ for $j < \lambda^+$. Note $N^{\lambda^+}_i = \bigcup_{j < \lambda^+} M^j_i$ also as $M^j_i \lea M^{j+1}_i = N^{j+1}_i$ for all $i \in I$, $j < \lambda^+$, and for $\cof(j) \geq \kappa$, $N_I^j = \bigcup_{i \in I} M^j_i$.

    Now define $C_i = \{j \in \lambda^+ : N^{\lambda^+}_i \cap N^j_I = N^j_i\}$ for all $i \in I$. It is straightforward to see these are closed in $\lambda^+$ (because we use the $N^j_i$'s, which are continuous in $j$ unlike the $M^j_i$'s), and they are also unbounded in $\lambda^+$: given $j < \lambda^+$, we can construct an increasing continuous sequence $\langle j_n : n \leq \omega \rangle$ such that $j_0 = j$ and $N^{\lambda^+}_i \cap N^{j_n}_I \subseteq N^{j_{n+1}}_i$ for all $n < \omega$; then $j_\omega \in C_i$. So as $|I| < \lambda^+$, $\bigcap_{i \in I} C_i$ is closed and unbounded in $\lambda^+$. Since $S = \{j < \lambda^+ : \cof(j) = \kappa\}$ is stationary, $\left(\bigcap_{i \in I} C_i\right) \cap S$ is non-empty.

    But if $j \in \left(\bigcap_{i \in I} C_i\right) \cap S$, for all $i \in I$ we have $N^{\lambda^+}_i \cap N^j_I = N^j_i$. Note that as $\cof(j) = \kappa$, $N^j_i = M^j_i$ for all $i \in I$. Hence $\left(\bigcup_{k < \lambda^+} M^k_i\right) \cap \left(\bigcup_{r \in I} M^j_r\right) = M^j_i$ for all $i \in I$. So in particular for all $r < s \in I$, we have $M^{j+1}_r \cap M^j_s = M^j_r$. This contradicts (2) of the construction.
\end{proof}

Now we move towards showing reduced towers are continuous at all $i \in I$ with cofinality $\geq \kappa$. We begin with a lemma, which is our analogue of \cite[5.18]{vasey18}.

\begin{lemma}\label{reducedthencontinuousweirdextensionlemma*}
    Suppose $I$ is a well ordering, where $\operatorname{otp}(I) = \delta + 1$ for some limit ordinal $\delta < \lambda^+$  such that $\cof(\delta) \geq \kappa$. Let $i_0, i_\delta$ be the initial and final elements of $I$ respectively. Let $\calt = \langle M_i : i \in I \rangle ^\wedge \langle a_i : i \in I^- \rangle$ be a tower. Suppose there is $b \in M_{i_\delta}$ such that $\gtp(b/ M_i, M_{i_\delta})$ $\dnf$-does not fork over $M_{i_0}$ for all $i <_I i_\delta$. Then there exists $\calt' = \langle M_i' : i \in I \rangle ^\wedge \langle a_i : i \in I^- \rangle$ such that $\calt \lesst \calt'$ and $b \in M_{i_0}'$.
\end{lemma}


\begin{proof}
    First note that by relabeling, without loss of generality, we may assume $I = \delta + 1$, $i_0 = 0$, and $i_\delta = \delta$ respectively.

    So assume $I = \delta + 1$, $i_0 = 0$, $i_\delta = \delta$. By applying Proposition \ref{towerextensionswithb*} to $\calt\upharpoonright \delta$ and $p = \gtp(b/M_0, M_\delta)$, there is a tower $\calt^* = \langle M_i^* : i <\delta \rangle ^\wedge \langle a_i : i < \delta \rangle$ and $b^* \in M_0^*$, such that $\calt \upharpoonright \delta \lesst \calt^*$, $\gtp(b/M_0, M_\delta) = \gtp(b^*/M_0, M_0^*)$, and $\gtp(b^*/M_i, M_i^*)$ $\dnf$-does not fork over $M_0$ for each $i < \delta$. Since $\gtp(b/M_i, M_\delta)$ also $\dnf$-does not fork over $M_0$ and both of these types extend $p \in \gS(M_0)$, we have $\gtp(b/M_i, M_\delta) = \gtp(b^*/M_i, M_i^*)$ by uniqueness for each $i < \delta$. Since these types $\dnf$-do not fork over $M_0$, by $\Kkappalims$-universal continuity*, $\gtp(b/\bigcup_{i<\delta}M_i, M_\delta) = \gtp(b^*/\bigcup_{i<\delta}M_i, \bigcup_{i<\delta}M_i^*)$.

    By type equality, there is some $M_\delta^\circ \in \K_\lambda$ where $\bigcup_{i<\delta} M_i^* \lea M_\delta^\circ$ and some $f:M_\delta \rightarrow M_\delta^\circ$ fixing $\bigcup_{i<\delta} M_i$ such that $f(b) = b^*$. Let $M_\delta^*$ be a $(\lambda, \kappa)$-limit over $M_\delta^\circ$. So $f:M_\delta \rightarrow M_\delta^*$, and there is some $g:M_\delta' \cong M_\delta^*$ an isomorphism extending $f$. Then if we let $M_i' = g^{-1}[M_i^*]$ for all $i < \delta$, $\calt' = \langle M_i' : i \leq \delta \rangle ^\wedge \langle a_i : i < \delta \rangle$ is a tower, where $\calt \lesst \calt'$, since $\calt \upharpoonright \delta \lesst \calt^*$ and $M_\delta'$ is universal over $M_\delta$. Note also that $b = g^{-1}(b^*) \in g^{-1}[M_0^*] = M_0'$, as desired.
\end{proof}

The following three results make up our generalisation of \cite[5.19]{vasey18}. We fill in some details and split the proof into separate lemmas for clarity.

\begin{lemma}\label{reducedthencontinuousinitialseglemma*}
    If $\calt = \langle M_i : i < \alpha\rangle ^\wedge \langle a_i : i \in \alpha^- \rangle$ is reduced and $\beta < \alpha$, then $\calt \upharpoonright \beta$ is also reduced.
\end{lemma}

\begin{proof}
    Suppose $\calt \upharpoonright \beta \lesst \calt^*$. Using Proposition \ref{finish_tower_extension}, there exists $\calt' = \langle M_i' : i < \alpha \rangle ^\wedge \langle a_i : i \in \alpha^- \rangle$ and $N, g : \bigcup_{i < \alpha} M_i' \underset{\bigcup_{i < \beta} M_i}{\longrightarrow} N$ such that $g[\calt'] \upharpoonright \beta = \calt$. Since $\calt$ is reduced, we must have for any $r < s < \beta$ that $M_{s} \cap M_r' = M_r$. In particular, for $r < s < \beta$, we have $M_{s} \cap M_r' = g[M_{s} \cap M_r'] = g[M_r] = M_r$. Therefore, $\calt \upharpoonright \beta$ is reduced as claimed.
\end{proof}

\begin{lemma}\label{reduced-implies-top-reduced}
    Suppose $\calt = \langle M_i : i < \alpha \rangle^\wedge \langle a_i : i \in \alpha^- \rangle$ is a reduced tower. Then for all $\beta < \alpha$, $\calt \upharpoonright [\beta, \alpha)$ is reduced.
\end{lemma}

\begin{proof}
    Suppose $\calt \upharpoonright [\beta, \alpha) \lesst \calt'$, where $\calt' = \langle M_i' : i \in [\beta, \alpha) \rangle^\wedge \langle a_i : i \in [\beta, \alpha)^- \rangle$. By Proposition \ref{towerextensionprop*} (with $\alpha = 1$), there exists a tower $\calt^* = \langle M_i^* : i \leq \beta \rangle^\wedge \langle a_i : i < \beta \rangle$ where $\calt \upharpoonright (\beta + 1) \lesst \calt^*$.

    Since $M_\beta \lea^u M_\beta'$, there exists a $\K$-embedding $f : M^*_\beta \underset{M_\beta}{\rightarrow} M_\beta'$. For $i < \beta$, set $M_i' = f[M_i^*]$. Note that $f$ fixes $a_i$ for all $i < \beta$. Define a new tower $\calt'' = \langle M_i' : i < \alpha \rangle^\wedge \langle a_i : i \in \alpha^- \rangle$.

    We claim that $\calt \lesst \calt''$. The conditions (1)-(3) of the definition of $\lesst$ (Definition \ref{towerorderingdef}) lift immediately from $\calt \upharpoonright (\beta + 1) \lesst \calt^*$ and $\calt \upharpoonright [\beta, \alpha) \lesst \calt'$. So it remains to show that $\gtp(a_i/M_i', M_{i+1}')$ $\dnf$-does not fork over $M_i$ for all $i \in \alpha^-$.

    For $i < \beta$, note $\gtp(a_i/M_i', M_{i+1}') = \gtp(a_i/M_i', f[M_{i+1}^*])$ (since $M_{i+1}' = f[M_{i+1}]$ for $i+1 < \beta$, and if $i+1 = \beta$ it follows from monotonicity of $\dnf$ as $f[M_{i+1}^*] \lea M_{i+1}'$). So we have
    \[\gtp(a_i/M_i', M_{i+1}') = \gtp(a_i/M_i', f[M_{i+1}^*]) = \gtp(f(a_i)/f[M_i^*], f[M_{i+1}^*]) = f(\gtp(a_i/M_i^*, M_{i+1}^*)).\] 
    Since $\gtp(a_i/M_i^*, M_{i+1}^*)$ $\dnf$-does not fork over $M_i$ by $\calt\upharpoonright (\beta + 1) \lesst \calt^*$, $\gtp(a_i/M_i', M_{i+1}')$ $\dnf$-does not fork over $f[M_i] = M_i$ by invariance. 
    
    For $i \in [\beta, \alpha)^-$, $\gtp(a_i/M_i', M_{i+1}')$ $\dnf$-does not fork over $M_i$ as $\calt \upharpoonright [\beta, \alpha) \lesst \calt'$. So we have shown condition (4) of Definition \ref{towerorderingdef} holds for all $i \in \alpha^-$, and $\calt \lesst \calt''$ as desired.

    So, since $\calt$ is reduced, for all $r, s < \alpha$ with $r \leq s$, we have
    $M_s \cap M_r' = M_r$. This holds in particular for $r, s \in [\beta, \alpha)$, so $\calt \upharpoonright [\beta, \alpha)$ is reduced as desired.
\end{proof}

\begin{prop}\label{reducedimplieshighcontinuity*}
    Suppose $\calt = \langle M_i : i \in I \rangle ^\wedge \langle a_i : i \in I^- \rangle$ is a reduced tower, and $\delta \in I$ is limit in $I$ with cofinality $\cof(\delta) \geq \kappa$. Then $\calt$ is continuous at $\delta$; that is, $M_\delta = \bigcup_{i<\delta}M_i$.
\end{prop}

\begin{proof}
    
    Suppose for contradiction this is false. Let $\alpha$ be minimal such that there exists a well ordered set $I$ with $\operatorname{otp}(I) = \alpha$, $\delta \in I$ with $\cof(\delta) \geq \kappa$, and a reduced tower $\calt = \langle M_i : i \in I \rangle ^\wedge \langle a_i : i \in I^- \rangle$ that is not continuous at $\delta$. As in the proof of Lemma \ref{reducedthencontinuousweirdextensionlemma*}, with some relabeling, we may assume $I = \alpha$.

    Note by Lemma \ref{reducedthencontinuousinitialseglemma*}, $\calt \upharpoonright (\delta+1)$ is also reduced and not continuous at $\delta$. So by minimality, $\alpha = \delta + 1$.

    Since $M_\delta \neq \bigcup_{i<\delta}M_i$, there is some $b \in |M_\delta| \setminus \bigcup_{i<\delta} |M_i|$. Using Proposition \ref{towerextensionprop*}, there exists $\calt^* = \langle M_i^* : i \leq \delta \rangle ^\wedge \langle a_i : i < \delta \rangle$ strongly universal such that $\calt \lesst \calt^*$.
    
    As $\kappa \leq \cof(\delta)$ and $\calt^*$ is universal, by $(\geq \kappa)$-local character there is $\beta < \delta$ such that $\gtp(b/\bigcup_{i<\delta}M_i^*, M_\delta^*)$ $\dnf$-does not fork over $M_\beta^*$. In particular, by monotonicity, $\gtp(b/M_i^*, M_{i+1}^*)$ $\dnf$-does not fork over $M_\beta^*$ for $i \in [\beta, \delta)$. Let $\calt^{**} = \calt^* \upharpoonright [\beta, \delta]$. By Lemma \ref{reducedthencontinuousweirdextensionlemma*} (using that $\cof(\operatorname{otp}([\beta, \delta))) = \cof(\delta) \geq \kappa)$, there exists $\calt' = \langle M_i' : i \in [\beta, \delta] \rangle ^\wedge \langle a_i : i \in [\beta, \delta)\rangle$ such that $\calt^{**} \lesst \calt'$ and $b \in M'_\beta$. Note $\calt \upharpoonright [\beta, \delta] \lesst \calt^{**} \lesst \calt'$. By Lemma \ref{reduced-implies-top-reduced}, $\calt \upharpoonright [\beta, \delta]$ is reduced, so $M_\beta' \cap M_\delta = M_\beta$. But $b \in M_\beta' \cap M_\delta$ and $b \notin M_\beta$, a contradiction. \end{proof}

Now we will recall the notion of full towers used in \cite{vasey18}, which will let us guarantee a tower contains universal chains.

\begin{defin}[{\cite[5.20]{vasey18}}] \label{full-def}
    Let $\calt = \langle M_i : i \in I \rangle ^\wedge \langle a_i : i \in I^-\rangle$ be a tower, and $I_0 \subseteq I$. We say $\calt$ is \emph{$I_0$-full} if for every $i \in I_0^-$ and every $p \in \gS^{na}(M_i)$, there is $k \in [i, i+_{I_0}1)_I$ such that $\gtp(a_k/M_k, M_{k+_I 1})$ is the $\dnf$-non-forking extension of $p$.
\end{defin}

The following remark motivates this definition. It describes how we will show our $\delta_1$th tower in the proof of Theorem \ref{largelimitsareisothm*} contains a universal sequence witnessing a $(\lambda, \delta_2)$-limit model.

\begin{remark}\label{fulltowerscontainuniversalchains*}
    If $\calt$ is $I_0$-full, then in particular, by Fact \ref{universalchriterionfact*}, $M_{i+_{I_0} \lambda}$ is universal over $M_i$ for all $i \in I_0$. Taking this further, if $\delta< \lambda^+$ is a limit ordinal and $\calt$ is continuous at $i +
    _{I_0} \lambda \cdot \delta$, then $M_{i+_{I_0} \lambda \cdot \delta}$ is a $(\lambda, \delta)$-limit model over $M_i$.
\end{remark}

The type extensions in Definition \ref{full-def} are non-forking to make fullness work with high cofinality unions of towers - see the following lemma, which extends {\cite[5.24]{vasey18}}.


\begin{lemma}\label{highfullunionsarefull*}
    Suppose $\delta < \lambda^+$ is a limit ordinal and $\kappa \leq \cof(\delta)$. Suppose $\langle \calt^j : j \leq \delta \rangle$ is a $\lesst$-increasing sequence of towers where $\calt^\delta = \bigcup_{j<\delta} \calt^j$ (in particular, $\bigcup_{j<\delta} I^j$ is a well ordering). Say $\calt^j = \langle M_i^j : i \in I^j\rangle ^\wedge \langle a_i : i \in (I^j)^- \rangle$ for $j \leq \delta$. Suppose $I_0 \subseteq I^0$ and $\calt^j$ is $I_0$-full for all $j < \delta$. Then $\calt^\delta$ is $I_0$-full.
\end{lemma}

\begin{proof}

    The method is essentially \cite[5.24]{vasey18}, using $\kappa \leq \cof(\delta)$ when applying local character. \end{proof}

\begin{nota}
    Given two well ordered sets $I$ and $J$, $I \times J$ will denote the usual lexicographic ordering; that is, $(i, j) < (i', j')$ if and only if either $i <_I i'$, or $i = i'$ and $j <_J j'$. We will use the notation $<_{\operatorname{lex}}$ when the ordering is ambiguous.
\end{nota}

The following is a slight improvement of {\cite[5.28]{vasey18}}, where a proof is not given - we include one for completeness, and relax the conditions on the limit ordinals, though the method appears to be the same.

\begin{lemma}\label{fullcompletionsexist*}
    Let $I$ be a well-ordering, and $\alpha, \gamma < \lambda^+$ be limit ordinals with $\alpha < \gamma$ and $\cof(\gamma) = \lambda$. If $\calt$ is a strongly universal tower indexed by $I \times \alpha$, then there is an $I \times \{0\}$-full tower $\calt'$ indexed by $I \times \gamma$ such that $\calt' \upharpoonright (I \times \alpha) =  \calt$.
\end{lemma}

\begin{proof}
    Say $\calt = \langle M_{(i, k)} : {(i, k)} \in I \times \alpha \rangle ^\wedge \langle a_{(i, k)} : {(i, k)} \in I \times \alpha\rangle$. We define $M_{(i, k)}$ and $a_{(i, k)}$ for $(i, k) \in I \times [\alpha, \gamma)$ recursively by the following procedure.

    Fix $i \in I$. Let $\langle p_k : k \in [\alpha, \gamma)\rangle$ be an enumeration of $\gS^{na}(M_{(i, 0)})$, possibly with repetitions. Note this is possible by stability in $\lambda$ and $\cof(\gamma) = \lambda$.

    Define $M'_{(i, k)}$ and $a'_{(i. k)}$ for $k < \gamma$ by induction on $k$ as follows:
    \begin{enumerate}
        \item If $k < \alpha$, $M'_{(i, k)} = M_{(i, k)}$ and $a'_{(i. k)} = a_{(i. k)}$
        \item If $k \geq \alpha$ is limit, then $M'_{(i, k)}$ is any $(\lambda, \kappa)$-limit model over $\bigcup_{l < k} M'_{(i, l)}$
        \item If $k = l + 1$, then let $\gtp(a'_{(i, l)}/M'_{(i, l)}, \hat{M}_{(i, l+1)})$ be a type extending $p_l$ which $\dnf$-does not fork over $M_{(i, 0)}$ (this is possible by existence and extension). Let $M'_{(i, l+1)}$ be a $(\lambda, \kappa)$-limit model over $\hat{M}'_{(i, l+1)}$. This determines $a'_{(i, l)}$ and $M'_{(i, k)}$.
    \end{enumerate}

    As $\calt$ is strongly universal, $M_{(i+1, 0)}$ is universal over $\bigcup_{k < \alpha} M_{(i, k)}$. We have $\bigcup_{k < \alpha} M_{(i, k)} = \bigcup_{k \in \alpha} M'_{(i, k)} \lea \bigcup_{k \in \gamma} M'_{(i, k)}$. So there exists $f:\bigcup_{k \in \gamma} M'_{(i, k)} \rightarrow M_{(i+1, 0)}$ fixing $\bigcup_{k \in \alpha} M_{(i, k)}$. For $k \in [\alpha, \gamma)$, take $M_{(i, k)} = f[M'_{(i, k)}]$ and $a_{(i, k)} = f(a'_{(i, k)})$.

    Invariance maintains the non-forking properties, so $\calt' = \langle M_{(i, k)} : {(i, k)} \in I \times \gamma \rangle ^\wedge \langle a_{(i, k)} : {(i, k)} \in I \times \gamma\rangle$ is a tower and $\calt' \upharpoonright (I \times \alpha) = \calt$. Furthermore, $\calt'$ is $I \times \{0\}$-full by condition (3) from the construction.
\end{proof}

\subsection{The main result}

Finally we will restate and prove Theorem \ref{largelimitsareisothm*}. The argument is similar to Vasey's proof of {\cite[2.7]{vasey18}}.

\largelimitsareisothmii*

\begin{proof}[Proof of Theorem \ref{largelimitsareisothm*}]
    By Fact \ref{cofinalityiso*}, it is enough to show that whenever $\delta_1, \delta_2 < \lambda^+$ are regular and $\kappa \leq \delta_1, \delta_2$, and $M \in \K_\lambda$, there exists a model which is both $(\lambda, \delta_1)$-limit model over $M$ and a $(\lambda, \delta_2)$-limit model over $M$. Note also it is enough to prove it for $M \in \Kkappalims$, as each $M$ has a $(\lambda, \kappa)$-limit $M'$ over $M$, and a $(\lambda, \delta_l)$-limit model over $M'$ will also be $(\lambda, \delta_l)$-limit over $M$. Fix such $\delta_1, \delta_2 \geq \kappa$ and $M \in \Kkappalims$.

    We will build a $\lesst$-increasing sequence of towers $\langle \calt^j : j \leq \delta_1 \rangle$ and a $\leq$-increasing continuous sequence of limit ordinals $\langle \alpha_j : j \leq \delta_1 \rangle \subseteq \lambda^+$ such that 
    \begin{enumerate}
        \item for all $j < \delta_1$, $\calt^j = \langle M_i^j : i \in (\delta_2 + 1) \times \lambda \times \alpha_j \rangle ^\wedge \langle a_i : i \in (\delta_2 + 1) \times \lambda \times \alpha_j\rangle$
        \item $M = M^0_{(0, 0, 0)}$ (so all $M^j_i$ contain $M$)
        \item for all $j < \delta_1$, $\calt^{2j + 2}$ is a reduced tower
        \item for all $j < \delta_1$, $\calt^{2j+1}$ is a $((\delta_2+1) \times \lambda \times \{0\})$-full tower
        \item $\calt^{\delta_1} = \bigcup_{j<\delta_1} \calt^j$ (which is valid as $\cof(\delta_1) \geq \kappa$).
    \end{enumerate}

    \textbf{This is possible:} We proceed by recursion. Let $\alpha_0 = \omega$. By Remark \ref{towersexist*}, there is a tower $\calt^0$ starting at $M$ of length $(\delta_2+1) \times \lambda \times \alpha_0$.

    \underline{For successors}, given $\calt^{2j}$ and $\alpha_{2j}$, we do the next two steps. By Proposition \ref{towerextensionprop*}, there is $\calt^{2j}_*$ a strongly universal tower indexed by $(\delta_2 + 1) \times \lambda \times \alpha_{2j}$ such that $\calt^{2j} \lesst \calt^{2j}_*$. Let $\alpha_{2j+1} < \lambda^+$ be a limit ordinal greater than $\alpha_{2j}$ where $\cof(\alpha_{2j+1}) = \lambda$ (note this exists since such ordinals form an unbounded set in $\lambda^+$ by regularity). Then by Lemma \ref{fullcompletionsexist*} there exists a $(\delta_2+1) \times \lambda \times \{0\}$-full tower $\calt^{2j+1}$ such that $\calt^{2j+1} \upharpoonright (\delta_2+1) \times \lambda \times \alpha_{2j} = \calt^{2j}_*$. In particular, $\calt^{2j} \lesst \calt^{2j+1}$. Let $\alpha_{2j+2} = \alpha_{2j+1}$. By Lemma \ref{reducedextensionsexist*}, there exists a reduced tower $\calt^{2j+2}$ indexed by $(\delta_2 + 1) \times \lambda \times \alpha_{2j+2}$ such that $\calt^{2j+1} \lesst \calt^{2j+2}$. 

    Finally, if $j<\delta_1$ is \underline{limit}, let $\alpha_j = \bigcup_{j' < j} \alpha_{j'}$ and take $\calt^j$ given by Proposition \ref{towerextensionprop*} such that for all $j' < j$, $\calt^{j'} \lesst \calt^j$. 
    
    $\calt^{\delta_1}$ is given by (5). This completes the construction.

    \textbf{This is enough:} Consider the final tower $\calt^{\delta_1}$. Since $\langle \calt^{2j+2} : j < \delta_1\rangle$ is a $\lesst$-increasing sequence of reduced towers, $\calt^{\delta_1}$ is reduced by Proposition \ref{reducedunionsarereduced*}. Hence it is continuous at $(\delta_2, 0, 0)$ (which has cofinality $\delta_2 \geq \kappa$ in $(\delta_2 + 1) \times \lambda \times \alpha_{\delta_1}$) by Proposition \ref{reducedimplieshighcontinuity*} - that is, $M^{\delta_1}_{(\delta_2, 0, 0)} = \bigcup_{k <_{\operatorname{lex}} (\delta_2, 0, 0)} M^{\delta_1}_k = \bigcup_{i < \delta_2} M^{\delta_1}_{(i, 0, 0)}$. Since $\langle \calt^{2j+1} : j \leq \delta_1 \rangle$ is a $\lesst$-increasing sequence of $(\delta_2+1) \times \lambda \times \{0\}$-full towers, $\calt^{\delta_1}$ is $(\delta_2+1) \times \lambda \times \{0\}$-full by Lemma \ref{highfullunionsarefull*}. In particular for all $i < \delta_2+1$ and $i' < \lambda$, $M^{\delta_1}_{(i, i' + 1, 0)}$ realises all types over $M^{\delta_1}_{(i, i', 0)}$, so by Fact \ref{universalchriterionfact*}, $M^{\delta_1}_{(i+1, 0, 0)}$ is universal over $M^{\delta_1}_{(i, 0, 0)}$ for all $i < \delta_2$. So $\langle M^{\delta_1}_{(i, 0, 0)} : i < \delta_2 \rangle$ is a $\lea^u$-increasing chain with union $M^{\delta_1}_{(\delta_2, 0, 0)}$. Therefore $M^{\delta_1}_{(\delta_2, 0, 0)}$ is a $(\lambda, \delta_2)$-limit over $M^{\delta_1}_{(0, 0, 0)}$, and in particular (as $M = M^0_{(0, 0, 0)} \lea M^{\delta_1}_{(0, 0, 0)})$), a $(\lambda, \delta_2)$-limit over $M$.

    On the other hand, $M^{\delta_1}_{(\delta_2, 0, 0)} = \bigcup_{j < \delta_1} M^j_{(\delta_2, 0, 0)}$. $\langle M^j_{(\delta_2, 0, 0)} : j < \delta_1 \rangle$ is a $\lea$-increasing universal chain by definition of the tower ordering, so $M^{\delta_1}_{(\delta_2, 0, 0)}$ is $(\lambda, \delta_1)$-limit over $M^0_{(\delta_2, 0, 0)}$, and hence as before over $M$.
    
    For the `moreover' part, by the above there are some $(\lambda, \delta_l)$-limit model $N_l$ for $l = 1, 2$ which are isomorphic (over some $M$). Then Fact \ref{cofinalityiso*} implies all such limit models (over any models) are isomorphic.
\end{proof}

\subsection{Tame AECs}\label{tame-long-lim-subsection} In this subsection we show how Theorem \ref{largelimitsareisothm*} can be used to show that in tame AECs with enough symmetry, there is a threshold above which all limit models at high enough stability cardinals are isomorphic (see Corollary \ref{vasey-tame-saturated-dnf-cor}). A key difference between the results of this subsection and the rest of the paper is that we do not assume the existence of a nice stable independence relation, but instead show that such a relation exists.

The following result follows  {\cite[\textsection 4, \textsection 5]{vasey16b}} closely. However we have to do additional work to guarantee an approximation of `full' universal continuity of non-forking: \cite{vasey16b} does not assume $\mu$-non-splitting satisfies universal continuity, and only guarantees $(\geq \kappa)$-universal continuity where $\kappa$ is the local character cardinal. We need universal continuity* in $\K_{\geq \mu^+}$, a stronger condition. For this, we assume universal continuity of $\mu$-non-splitting. The universal continuity* arguments are adapted from \cite[4.4, 4.11]{leu2}.

\begin{lemma}\label{dnf-lemma-vasey-saturated}
    Let $\K$ be an AEC with AP, NMM, stable in $\mu \geq \LS(\K)$, and $\mu$-tame.  Assume $\mu$-non-splitting satisfies universal continuity. Let $\dnf$ be $\dnf_{(\geq\mu)-f}$ restricted to models in $\K_{\geq \mu^+}^{\mu^+\operatorname{-sat}}$ (that is, the $\mu^+$-saturated models in $\K_{\geq \mu^+}$ ordered by $\lea$).

    Then $\dnf$ satisfies invariance, monotonicity, base monotonicity, extension, uniqueness, universal continuity* in $\K_{\geq \mu^+}$, and $(\geq \mu^+)$-local character.
\end{lemma}

\begin{proof}
    Invariance, monotonicity, and base monotonicity are clear from the definition. Extension, and uniqueness follow from \cite[5.9]{vasey16b} and \cite[5.3]{vasey16b} respectively with $\lesst = \lea^u$. For extension, note the proof goes through even with algebraic types if you assume $p$ $(\geq \mu)$-does not fork over $M$; this gives the weak form of extension from Lemma \ref{extension-equivalence}, and then we can apply Lemma \ref{extension-equivalence}). Note that $\mu^+$-saturated models are $\mu^+$-model homogeneous under our assumptions as $\mu^+ > \LS(\K)$ (see \cite[2.11]{vasey16b}), so these results may be applied.
    
    By \cite[3.11, 4.11]{vasey16b} we have a stronger form of $(\geq \kappa)$-local character for some minimal $\kappa \leq \mu^{+}$. To be precise, $\kappa$ is the least regular cardinal such that for any increasing sequence $\langle M_i : i < \kappa \rangle$ where $M_i \in \K_{\geq \mu^+}^{\mu^+\operatorname{-sat}}$ for all $i < \kappa$, and for all $p \in S(\bigcup_{i < \kappa} M_i)$, there exists $i < \kappa$ such that $p$ $(\geq \mu)$-does not fork over $M_i$; and this property holds for all regular $\kappa' \geq \kappa$. This is stronger than $(\geq \kappa)$-local character of $\dnf$ as it is possible that $\bigcup_{i < \kappa}M_i$ is no longer $\mu^+$-saturated.

    Now we prove universal continuity* of $\dnf$ in $\K_{\geq \mu^+}$. Let $\langle M_i : i < \delta \rangle$ be an increasing sequence of $\mu^+$-saturated models. Let $M_\delta = \bigcup_{i < \delta} M_i$. Suppose we have $p_i \in \gS(M_i)$ increasing such that $p_i$ does not $(\geq\mu)$-fork over $M_0$ for each $i<\delta$. Take a $\mu^+$-saturated model $M^* \in \K$ with $M_\delta \lea^u M^*$. By extension there is $p^* \in \gS(M^*)$ such that $p_0 \subseteq p^*$ and $p$ does not $(\geq \mu)$-fork over $M_0$. By uniqueness, $p^* \upharpoonright M_i = p_i$. So $p_i \subseteq p^*$ for all $i < \delta$. Let $p = p^* \upharpoonright M_\delta$. We must show $p$ is the unique extension of the $p_i$'s.

    So suppose we have $q \in \gS(M_\delta)$ with $q \upharpoonright M_i = p_i$ for all $i < \delta$. We must show $p = q$. We go by cases.
    
    \textbf{Case 1:} assume $\delta \geq \kappa$. By \cite[4.12]{vasey16b}, $q$ does not $(\geq\mu)$-fork over $M_0$. This is true also of $p$, so by \cite[4.8]{vasey16b} there exist $M_0^p, M_0^q \in \K_\mu$ where $M_0^p, M_0^q \lea M_0$, $p$ does not $\mu$-split over $M_0^p$, and $q$ does not $\mu$-split over $M_0^q$ (note this only requires that $M_0$ is $\mu^+$-saturated, and not necessarily $M_\delta$). Taking $M_0^0 \in \K_\mu$ with $M_0^p, M_0^q \lea M_0^0 \lea M_0$, we have $p, q$ do not $\mu$-split over $M_0^0$ by monotonicity. Then take $M_0^1 \lea M_1$ in $\K_\mu$, universal over $M_0^0$ (this is possible as $M_0^0 \lea M_0 \lea^u M_1$ and $\K$ is $\mu$-stable). For every $N \in \K_\mu$ with $M_0^1 \lea N \lea M_\delta$, we have that $p \upharpoonright M_0^1 = q \upharpoonright M_0^1$, so by weak uniqueness of $\mu$-non-splitting (Fact \ref{uniqueness-split}, using that $M_0^0 \lea^u M_0^1 \lea N$), we have $p \upharpoonright N = q \upharpoonright N$. This holds for all such $N$, so by $\mu$-tameness, $p = q$ as desired.

    \textbf{Case 2:} assume $\delta < \kappa$. Then $\delta \leq \mu$ in particular. By \cite[4.8]{vasey16b} for each $i < \delta$ there is $N_i \in \K_\mu$ such that $N_i \lea M_0$ and $p_i$ does not $\mu$-split over $N_i$. Take $N \in \K_\mu$ such that $N_i \lea N \lea M_0$ for each $i < \delta$. Take $N' \in \K_\mu$ such that $N \lea^u N'$ and $N' \lea M_1$.
    
    Suppose that $N^* \in \K_\mu$ with $N' \lea N^* \lea M_\delta$. Take some $\lea^u$-increasing sequence of models $\langle M_i' : i \in \delta \rangle$ in $\K_\mu$ where $N' \lea M_i' \lea M_{i+1}$ and $|M_i| \cap |N^*| \subseteq |M_i'|$ for all $i < \delta$ (this is possible as $M_i \lea^u M_{i+1}$ and $\K$ is $\mu$-stable). Let $M'_\delta = \bigcup_{i<\delta} M_i'$. Note that $N^* \lea M'_\delta$. We have from monotonicity that $p \upharpoonright M_i' = q \upharpoonright M_i' = p_{i+1} \upharpoonright M_i'$ does not $\mu$-split over $N$. So by universal continuity of $\mu$-non-splitting, both $p \upharpoonright M'_\delta$ and $q \upharpoonright M'_\delta$ do not $\mu$-split over $N$. As $N \lea^u N' \lea M_\delta'$ and $p \upharpoonright N' = q \upharpoonright N'$, by weak uniqueness of $\mu$-non-splitting, $p \upharpoonright M'_\delta = q \upharpoonright M'_\delta$, and therefore $p \upharpoonright N^* = q \upharpoonright N^*$. This holds for all such $N^*$, so by $\mu$-tameness, $p = q$ as desired.
\end{proof}

\begin{remark}
    The above proof goes through if we take $\kappa = \mu^+$, rather than the minimal $\kappa$.
\end{remark}

\begin{lemma}\label{vasey-saturated-dnf-criteria-hold} Let $\K$ be an AEC, stable in $\mu \geq \LS(\K)$ and stable also in $\lambda \geq \mu^+$, where $\K$ has AP, NMM, $\mu$-tameness, and $\K_\lambda$ has JEP. Assume universal continuity of $\mu$-non-splitting.

   Let $\dnf$ be $\dnf_{(\geq\mu)-f}$ restricted to $\Kmupluslims$. Suppose $\dnf$ satisfies non-forking amalgamation. Then $\dnf$ satisfies Hypothesis \ref{long-limit-dnf-hypotheses} with $\kappa = \mu^{+}$ for any stability cardinal $\lambda \geq \mu^{+}$.
\end{lemma} 

\begin{proof}
    This is immediate from Lemma \ref{dnf-lemma-vasey-saturated} and the fact that $\Kmupluslims \subseteq \K_{\geq \mu^+}^{\mu^+\operatorname{-sat}}$. Note that $\Kmupluslims$-universal continuity* in $\K$ follows from the fact that $\dnf$ satisfies universal continuity* in $\K_{\geq \mu^+}$.
\end{proof}

We phrase the following in terms of symmetry of $(\geq \mu)$-non-forking in $\Kmupluslims$, which is a parallel assumption to $\lambda$-symmetry (see Definition \ref{lambda-symmetry-def} and Remark \ref{symmetry-remark}). In this sense, the following theorem may be regarded as assuming $\K$ has nice properties, but no existing independence relation.

\begin{cor}\label{vasey-tame-saturated-dnf-cor}
    Let $\K$ be an AEC, stable in $\mu \geq \LS(\K)$ and stable also in $\lambda \geq \mu^+$, where $\K$ has AP, NMM, $\mu$-tameness, and $\K_\lambda$ has JEP. Assume universal continuity of $\mu$-non-splitting. Suppose also that $(\geq \mu)$-non-forking has symmetry in $\Kmupluslims$ (or just that $(\geq \mu)$-non-forking restricted to models in $\Kmupluslims$ satisfies non-forking amalgamation).

    Let $\delta_1, \delta_2 < \lambda^+$ be limit ordinals where $\mu^{+} \leq \cof(\delta_1), \cof(\delta_2)$.
    If $M, N_1, N_2 \in \K_\lambda$ where $N_l$ is a $(\lambda, \delta_l)$-limit over $M$ for $l=1, 2$, then there is an isomorphism from $N_1$ to $N_2$ fixing $M$.

    Moreover, if $N_1, N_2 \in \K_\lambda$ where $N_l$ is $(\lambda, \delta_l)$-limit for $l=1, 2$, then $N_1$ is isomorphic to $N_2$.
\end{cor}

\begin{proof}
    Follows directly from Lemma \ref{vasey-saturated-dnf-criteria-hold}, Fact \ref{symm-implies-nfap-fact}, and Theorem \ref{largelimitsareisothm*}.
\end{proof}

\section{Short limit models} \label{short-limits-are-distinct-section}

Our goal in this section is to show that, in a very general setting, all the low cofinality limit models are non-isomorphic. As before, we present the result before describing the hypotheses.

\begin{restatable}{theorem}{smalllimitsarenoniso}\label{non-iso}
    Assume Hypothesis \ref{short-limit-dnf-hypotheses} holds for an AEC $\K$ and $\lambda \geq \LS(\K)$. Suppose $\K$ is $\aleph_0$-tame. If $\operatorname{cf}(\delta_1)   < \kappadnfu$ and  $\operatorname{cf}(\delta_1) \neq \operatorname{cf}(\delta_2)$  , then the $(\lambda, \delta_1)$-limit model is not isomorphic to the $(\lambda, \delta_2)$-limit model.
\end{restatable}

We assume the following hypothesis throughout this section.

\begin{restatable}{hypothesis}{shortlimithypothesis}\label{short-limit-dnf-hypotheses}
Let $\K$ be an AEC stable in $\lambda \geq \LS(\K)$, with AP, JEP, and NMM in $\K_\lambda$. Let $\kappa < \lambda^+$ be a regular cardial. Let $\dnf$ be an independence relation on $\K_\lambda$ that satisfies invariance, monotonicity, base monotonicity, uniqueness, extension, ($\geq \kappa$)-local character, and universal continuity.
\end{restatable}

\begin{remark}
    By Remark \ref{weak-loc-char-iff-strong}, the minimal such $\kappa$ is $\kappadnfu$ from Definition \ref{kappa-defs}. In particular $\kappadnfu \leq \lambda < \infty$.
\end{remark}

\subsection{Relating $\dnf$ to splitting and $\lambda$-forking} \label{dnf-is-like-dns} This short subsection studies how several independence relations interact with each other in a strictly stable set-up (assuming Hypothesis \ref{short-limit-dnf-hypotheses}). It finishes with a canonicity result of independence relations for long limit models.

The following is essentially \cite[4.2]{bgkv} (see also \cite[14.1]{vaseyn}). The method is due to Shelah \cite[Lemma III.1.9*]{shbook}.

\begin{fact}\label{dnf_implies_dns} Let $M \lea N$ in $\K_\lambda$ and $p \in \gS(N)$. 
If $p$ $\dnf$-does not fork over $M$, then $p$ does not $\lambda$-split over $M$.
\end{fact}

\begin{proof}
    Assume $p$ $\dnf$-does not fork over $M$. Suppose $N_1, N_2 \in \K_\lambda$ with $M \leq N_l \leq N$ for $l=1,2$. Suppose further that $f : N_1 \underset{M}{\cong} N_2$. By monotonicity and invariance, $f(p \upharpoonright N_1) \in \gS(N_2)$ $\dnf$-does not fork over $M$. By monotonicity $p \upharpoonright N_2$ $\dnf$-does not fork over $M$. Since $f(p \upharpoonright N_1) \upharpoonright M = (p \upharpoonright N_2) \upharpoonright M$, we have by uniqueness of $\dnf$ that $f(p \upharpoonright N_1) = p \upharpoonright N_2$. Therefore $p$ does not $\lambda$-split over $M$.
\end{proof}

For our next result, we use the following fact.

\begin{fact}[{\cite[I.4.12]{van06}}]\label{uniqueness-split} (Weak uniqueness of splitting) Let $M_0 \lea^u M_1 \lea M_2$ all in $\K_\lambda$. If  $p, q \in \gS(M_2)$,  $p\rest M_1 = q\rest M_1$, and  $p, q$ do not $\lambda$-split over $M_0$, then $p =q$.
\end{fact}

The following argument has some similarities with the second half of \cite[14.1]{vaseyn}, and with \cite[4.18]{leu2}.

\begin{lemma}\label{splitting-to-forking-2}
    Suppose $L_1 \lea L_2 \lea M$  are all in $\K_\lambda$, $L_2$ is a $(\lambda, \geq \shortkappadnfu)-$limit model over $L_1$, and $p \in \gS(M)$. If $p$ does not $\lambda$-split over $L_1$, then $p$ $\dnf$-does not fork over $L_2$. 
\end{lemma}
\begin{proof}
Observe that $p \upharpoonright L_2$ $\dnf$-does not fork over $L_2$ by Lemma \ref{existence-on-high-limits}. There exists $q \in \gS(M)$ such that $q$ $\dnf$-does not fork over $L_2$ and $q\rest L_2 = p\rest L_2$ by extension of $\dnf$. Let $\langle M_i :i < \theta \rangle$ witness that $L_2$ is a $(\lambda, \geq \shortkappadnfu)$-limit model over $L_1$. By  ($\geq \shortkappadnfu$)-local character, there is an $i < \theta$ such that $q \rest L_2$ $\dnf$-does not fork over $M_i$. Then $q$ $\dnf$-does not fork over $M_i$ by transitivity of $\dnf$. Therefore, $q$ does not $\lambda$-split over $M_i$ by Fact \ref{dnf_implies_dns}. 

As $L_1 \lea M_i$, $p$ does not $\lambda$-split over $M_i$ by base monotonicity of $\lambda$-non-splitting. We showed earlier that $q$ does not $\lambda$-split over $M_i$, and we know that $p\rest L_2 = q\rest L_2$, and $L_2$ is universal over $M_i$. Therefore, $p=q$ by weak uniqueness of splitting. As $q$ $\dnf$-does not fork over $L_2$, we have that $p$ $\dnf$-does not fork over $L_2$ as desired.
\end{proof}

The following result is a partial converse of Fact \ref{dnf_implies_dns}.

\begin{lemma}\label{split-fork} Suppose $L_1 \lea^u L_2 \lea M$ are all in $\K_\lambda$, and $p \in \gS(M)$. If $p$ does not $\lambda$-split over $L_1$, then $p$ $\dnf$-does not fork over $L_2$.

\end{lemma}
\begin{proof}
    Let $\langle M_i : i \leq \shortkappadnfu \rangle$ be a $\lea^u$-increasing continuous sequence with $M_0 = L_1$. By universality of $L_2$ over $L_1$, we may assume $M_i \lea L_2$ for all $i \leq \shortkappadnfu$. Then $p$ $\dnf$-does not fork over $M_\shortkappadnfu$ by Lemma \ref{splitting-to-forking-2}. Hence $p$ $\dnf$-does not fork over $L_2$ by base monotonicity.
\end{proof}

\begin{cor}\label{equal-kappa}
 $\underline{\kappa}(\dnf, \K_\lambda, \lea^u) = \underline{\kappa}(\dnf_{\splt}, \K_\lambda, \lea^u)$. In particular, ${\kappa(\dnf_{\splt}, \K_\lambda, \lea^u)} = \shortkappadnfu$.
\end{cor}
\begin{proof}
   $ \underline{\kappa}(\dnf, \K_\lambda, \lea^u) \subseteq \underline{\kappa}(\dnf_{\splt}, \K_\lambda, \lea^u)$ by Fact \ref{dnf_implies_dns} (this inclusion also appears in \cite[3.9(2)]{vaseyt}) and     $\underline{\kappa}(\dnf_{\splt}, \K_\lambda, \lea^u) \subseteq \underline{\kappa}(\dnf, \K_\lambda, \lea^u)$ by Lemma \ref{split-fork}.
\end{proof}

\begin{cor}\label{split-fork2}
Suppose $L \lea M$ are in $\K_\lambda$ and $p \in \gS(M)$. If $p$ does not $\lambda$-fork over $L$, then $p$ $\dnf$-does not fork over $L$.
\end{cor}
\begin{proof}
    Suppose $L \lea M$ in $\K_\lambda$, and $p \in \gS(M)$ does not $\lambda$-fork over $L$. Then by the definition of $\lambda$-forking (Definition \ref{forking_def}) there is $L_1 \in \K_\lambda$ such that $L_1 \lea^u L$ and $p$ does not $\lambda$-split over $L_1$. Hence $p$ $\dnf$-does not fork over $L$ by Lemma \ref{split-fork}.
\end{proof}

The following result is the latest of a family of results  dealing with canonicity of independence relations \cite[5.19]{bgkv} \cite[9.6]{vasey16}, \cite[2.5]{vasey18}.  A key difference between our result and those just mentioned is that a priori we do not know if $\lambda$-non-forking has uniqueness.

\begin{theorem}\label{nf-limits}
   Suppose $L \lea M$ are in $\K_\lambda$, $L$ is a $(\lambda, \geq \shortkappadnfu)$-limit model, and $p \in \gS(M)$. Then $p$ does not $\lambda$-fork over $L$ if and only if $p$ $\dnf$-does not fork over $L$.
\end{theorem}

\begin{proof}
    The forward implication follows from the Corollary \ref{split-fork2}. For the backward implication, suppose $p$ $\dnf$-does not fork over $L$. Let $\langle M_i : i \leq \theta \rangle$ witness that $L$ is a $(\lambda, \geq \shortkappadnfu)$-limit model. Then by $(\geq \shortkappadnfu)$-local character there is $i < \theta$ such that $p \rest L$ $\dnf$-does not fork over $M_i$. Then $p$ $\dnf$-does not fork over $M_i$ by transitivity of $\dnf$.  Hence $p$ does not $\lambda$-split over $M_i$ by  Fact \ref{dnf_implies_dns}.  As $M_i \lea^u L$, it follows that $p$ does not $\lambda$-fork over $L$.
\end{proof}

\begin{remark}
Theorem \ref{nf-limits} provides a correct proof of \cite[4.18]{leu2} assuming Hypothesis \ref{short-limit-dnf-hypotheses} (see Remark \ref{leung-uniqueness-necessary}).
\end{remark}

\subsection{Non-isomorphism results} This subsection has the main result of the section. Recall that we are still assuming  Hypothesis \ref{short-limit-dnf-hypotheses}.

 Part of the argument of the next Lemma is a special case of \cite[4.6]{sebastien-successive-categ}, when $\chi = \delta^+$ and $\theta = \delta$. We include that part of the argument for completeness. 

\begin{lemma}\label{s-dnf}
Suppose that $\K$ is $\delta$-tame for some regular $\delta < \lambda^+$. Suppose $\langle M_i : i \leq \delta \rangle$ is a $\lea^u$-increasing continuous chain in $\K_\lambda$.  If $M_\delta$ is $\delta^+$-saturated and $p \in \gS(M_\delta)$, then there is an $i < \delta$ such that $p$ $\dnf$-does not fork over $M_i$. 
\end{lemma}
\begin{proof}
Observe that it is enough to show that $p$ does not $\lambda$-split over $M_i$ for some $i < \delta$ since then one has that $p$ $\dnf$-does not fork over $M_{i+1}$ for some $i < \delta$ by Lemma \ref{split-fork}.

Suppose for a contradiction that $p$ $\lambda$-splits over $M_i$ for every $i < \delta$. Then for each $i < \delta$, there exist $N_i^1, N_i^2 \in \K_\lambda$ and $f_i : N_i^1 \underset{M_i}{\cong} N_i^2$ such that $M_i \lea N_i^l \lea M_\delta$ for $l = 1, 2$ and $f_i(p \upharpoonright N_i^1) \neq p \upharpoonright N_i^2$.

    For each $i < \delta$, by $\delta$-tameness there exists $A_i \subseteq N_i^2$ such that $|A_i| \leq \delta$ and 
    \begin{equation}\tag{$\dagger$}\label{Ai-splits}
    f_i(p \upharpoonright N_i^1) \upharpoonright A_i \neq (p \upharpoonright N_i^2) \upharpoonright A_i
    \end{equation}

    For each $i < \delta$, let $B_i = f_i^{-1}(A_i) \cup A_i$. Let $B = \bigcup_{i < \delta} B_i$. Note that $|B| \leq \delta$ as $|B_i| \leq \delta$ for all $i < \delta$.

    Since $M_\delta$ is $\delta^+$-saturated, there exists $b \in M$ such that $\gtp(b/B, M_\delta) = p \upharpoonright B$. As $\langle M_i : i \leq \delta \rangle$ is continuous, there is some $i < \delta$ such that $b \in M_i$. Let $g: M_\delta \cong N$ such that $f_i \subseteq g$ and $N_i^2 \lea N$. 

 On one hand. 
    \begin{align*}
        (p \upharpoonright N_i^2) \upharpoonright A_i = p \upharpoonright A_i = \gtp(b/A_i, M_{\delta}) = \gtp(b/A_i, N_i^2) = \gtp(b/A_i, N) 
        \end{align*} 
        
        where the last two equalities follow from $b \in M_i \leq N_i^2 \lea M_\delta$ and $N_i^2 \lea N$. 

On the other hand, 
\begin{align*}        
        f_i(p \upharpoonright N_i^1) \upharpoonright A_i =  f_i(p \upharpoonright f^{-1}_i(A_i)) =  g(\gtp(b/f^{-1}_i(A_i), M_{\delta})) =  \gtp(g(b)/A_i, N) =  \gtp(b/A_i, N)
    \end{align*}
where the last equality follows from the fact that $b \in M_i$ and $g \upharpoonright M_i = f_i \upharpoonright M_i = \id_{M_i}$.

Hence $(p \upharpoonright N_i^2) \upharpoonright A_i = f_i(p \upharpoonright N_i^1) \upharpoonright A_i$. This contradicts equation (\ref{Ai-splits}).
\end{proof}

\begin{lemma}\label{key-no}
  Suppose $\mu_1 < \mu_2  < \lambda^+$ are infinite regular cardinals and that $\K$ is $\mu_1$-tame. If $M_1$ is a $(\lambda, \mu_1)$-limit model, $M_2$ is a $(\lambda, \mu_2)$-limit model, and $M_1$ is isomorphic to $M_2$, then $\shortkappadnfu \leq \mu_1$. 
\end{lemma}
\begin{proof}
It is enough to show that $ \mu_1 \in \underline{\kappa}(\dnf, \K_\lambda, \lea^{u})$ by the minimality of $\shortkappadnfu$. Let $\langle N_i : i \leq \mu_1 \rangle$ be a $\lea^u$-increasing continuous chain in $\K_\lambda$ and  $p \in \gS(N_{\mu_1})$. 

As $M_1$ is isomorphic to $M_2$ and $M_2$ is a $(\lambda, \mu_2)$-limit model, it follows from Fact \ref{long-limits-are-saturated} that $M_1$ is $\mu_2$-saturated. So in particular $M_1$ is $(\mu_1)^+$-saturated. Since $M_1$ is a $(\lambda, \mu_1)$-limit model, $M_1$ is isomorphic to $N_{\mu_1}$ by Fact \ref{cofinalityiso*}. Hence $N_{\mu_1}$ is $(\mu_1)^+$-saturated. Then there is $i < \mu_1$ such that $p$ $\dnf\text{-does not fork over } N_i$ by Lemma \ref{s-dnf} as $\K$ is $\mu_1$-tame. Therefore $\mu_1 \in \underline{\kappa}(\dnf, \K_\lambda, \lea^{u})$. 
\end{proof}

\begin{remark}
    We only used $\mu_1$-tameness in Lemma \ref{key-no}. However we can only use the lemma to understand the full spectrum of limit models if we assume $\aleph_0$-tameness.
\end{remark}

The main result of this section follows from the previous Lemma and the fact that if $\K$ is $\aleph_0$-tame then it is $\mu$-tame for every infinite cardinal $\mu$.  We again emphasise the hypotheses as this is a key result,

\smalllimitsarenoniso*

Tameness plays a key role in the main result of this section. A key natural question is the following.

\begin{probl}\label{p-nontame}
Is Theorem \ref{non-iso} true without the tameness assumption?
\end{probl}
\section{General results}

The objective of this section is to  combine the results of Section 3 and 4 in a natural set up.  We hope that these results  can be used as black boxes when studying limit models in natural abstract elementary classes. We showcase how this can be achieved in the next section.


\subsection{Main results} The following result puts together the main results of the previous two sections. As such the result is a local result on limit models.

\begin{restatable}{theorem}{completelimitmodelpicture}\label{full-picture}
    Let $\K$ be a $\aleph_0$-tame AEC stable in $\lambda \geq \LS(\K)$ with AP, JEP, and NMM in $\K_\lambda$. Let $\kappa < \lambda^+$ be regular. Let $\dnf$ be an independence relation on $\K_\lambda$ that satisfies uniqueness, extension, non-forking amalgamation, universal continuity, and $(\geq \kappa)$-local character.

    Suppose $\delta_1, \delta_2 < \lambda^+$ with $\cof(\delta_1) < \cof(\delta_2)$. Then for any $N_1, N_2, M \in \K_\lambda$ where $N_l$ is a $(\lambda, \delta_l)$-limit model over $M$ for $l = 1, 2$, 
    \[N_1 \text{ is isomorphic to } N_2 \text{ over } M  \iff \cof(\delta_1) \geq \kappadnfu\]
    
    Moreover, for any $N_1, N_2, \in \K_\lambda$ where $N_l$ is a $(\lambda, \delta_l)$-limit model for $l = 1, 2$, 
    \[N_1 \text{ is isomorphic to } N_2 \iff \cof(\delta_1) \geq \kappadnfu\]
  
\end{restatable}

\begin{proof}
    We can restrict ourselves to regular $\delta_1, \delta_2$ by Fact \ref{cofinalityiso*}. Our assumptions imply that $\K$ and $\dnf$ satisfy Hypothesis \ref{long-limit-dnf-hypotheses} with $\kappa = \kappadnfu$ and Hypothesis \ref{short-limit-dnf-hypotheses}. The conclusion of Theorem \ref{non-iso} gives both of the forward implications. The conclusion to Theorem \ref{largelimitsareisothm*} gives both of the reverse implications.
\end{proof}

\begin{remark}
The value of Theorem \ref{full-picture} is that we can fully understand the limit models locally (that is, in $\K_\lambda$) just by computing $\kappadnfu$. 
\end{remark}
\begin{remark}
Observe that in Theorem \ref{full-picture} one can substitute every occurrence of  ${\kappa(\dnf, \K_\lambda, \lea^u)}$ with $\kappa(\dnf_{\splt}, \K_\lambda, \lea^u)$ by Corollary \ref{equal-kappa}. 
\end{remark}

Conjecture 1.1 of \cite{bovan} follows directly from the previous result under the assumptions of Theorem \ref{full-picture}.
 
 \begin{cor}
Let $\K$ be a $\aleph_0$-tame abstract elementary class stable in regular $\lambda \geq \LS(\K)$, with AP, JEP, and NMM in $\K_\lambda$. Let $\kappa < \lambda^+$ be a regular cardinal. Suppose $\dnf$ is an independence relation on $\K_\lambda$ that satisfies uniqueness, extension, non-forking amalgamation, universal continuity, and $(\geq \kappa)$-local character. Let

\[ \Gamma =  \{ \alpha < \lambda^+ : cf(\alpha) = \alpha \text{ and  the } (\lambda, \alpha)\text{-limit model is isomorphic to the } (\lambda, \lambda)\text{-limit model } \}. \]

Then  $\Gamma =[\kappadnfu, \lambda^+) \cap \operatorname{Reg}$. 
 \end{cor}
 
We now focus on obtaining a global result. For the rest of this section we assume the following hypothesis.

\begin{hypothesis}\label{global-short-limit-dnf-hypotheses}
    Let $\K$ be a stable AEC with AP, JEP and NMM. Let $\kappa < \LS(\K)^{++}$ be a regular cardinal. Let $\dnf$ be an independence relation on $\K$ that satisfies invariance, monotonicity, base monotonicity, uniqueness, extension, non-forking amalgamation, universal continuity, and $(\geq \kappa)$-local character.
\end{hypothesis}

Note we do not assume tameness. This will be assumed explicitly when necessary for generality.

The following is essentially \cite[8.8]{lrv1}, but the argument there has a minor error: if $A \nsupseteq M_0$ and $p \in \gS(A)$, then whether $p$ $\dnf$-forks over $M_0$ or not is not well defined. This is easy to correct - instead of using uniqueness on the types $p \upharpoonright A$ and $q \upharpoonright A$, apply it to $p \upharpoonright M_0 \cup A$ and $q \upharpoonright M_0 \cup A$, which are both $\dnf$-non-forking. This still gives $p \upharpoonright A = q \upharpoonright A$. We provide the details in our context for convenience (note $M_0 \cup A$ is replaced by some model $N_0$ containing $|M_0| \cup A$ since we have not defined $\dnf$-forking over types with domains that are not models).

\begin{fact}\label{tame-gives-wp}
Suppose $\K$ is stable in $\lambda \geq \LS(\K)$, and $\theta \leq \lambda$.  If $\K$ is $(<\theta)$-tame, then $\dnf$ has the $(<\theta)$-witness property. In particular if $\K$ is $(<\aleph_0)$-tame, then  $\dnf$ has the $(<\aleph_0)$-witness property.\footnote{Fact \ref{tame-gives-wp} holds assuming only that $\K$ is an AEC with an independence relation $\dnf$ satisfying extension and uniqueness, rather than all of Hypothesis \ref{global-short-limit-dnf-hypotheses}.}
\end{fact}

\begin{proof}
    Suppose $M \lea N$ in $\K$ and $p \in \gS(N)$ satisfies that for all $A \subseteq |N|$ where $|A| < \theta$, there exists $N_0 \lea N$ with $A \cup |M| \subseteq |N_0|$ and $p \upharpoonright N_0$ $\dnf$-does not fork over $M$. We must show $p$ $\dnf$-does not fork over $M$.

    Using $A = \varnothing$ and monotonicity, $p \upharpoonright M$ $\dnf$-does not fork over $M$. By extension, there is $q \in \gS(N)$ extending $p \upharpoonright M$ such that $q$ $\dnf$-does not fork over $M$. It suffices to show $p = q$. By $(<\theta)$-tameness, it is enough to show $p \upharpoonright A = q \upharpoonright A$ for all $A \subseteq |N|$ where $|A| < \theta$. 
    
    So fix $A \subseteq |N|$ with $|A| < \theta$. By our hypothesis, there exists $N_0 \lea N$ with $A \cup |M| \subseteq |N_0|$ and $p \upharpoonright N_0$ $\dnf$-does not fork over $M$. By monotonicity, $q \upharpoonright N_0$ $\dnf$-does not fork over $M$ either. As $p \upharpoonright M = q \upharpoonright M$, uniqueness gives $p \upharpoonright N_0 = q \upharpoonright N_0$. In particular, $p \upharpoonright A = q \upharpoonright A$, as desired.
\end{proof}

We show \cite[4.5]{vaseyt} holds under Hypothesis \ref{global-short-limit-dnf-hypotheses} when non-splitting is replaced by $\dnf$ for $\lambda \geq \LS(\K)$.

\begin{lemma}\label{kappa-increases}
    Let $\K$ be stable in $\lambda$ and $\mu$ with  $LS(\K) \leq \mu < \lambda$ and let $\theta \leq \mu$. If $\dnf$ has the $\theta$-witness property, then $\underline{\kappa}(\dnf, \K_\mu, \lea^{u}) \subseteq \underline{\kappa}(\dnf, \K_\lambda, \lea^{u})$. In particular, ${\kappa(\dnf, \K_\lambda, \lea^{u})} \leq \kappa(\dnf, \K_\mu, \lea^{u})$.\footnote{Lemma \ref{kappa-increases} holds assuming only that $\K$ is an AEC with an independence relation $\dnf$ satisfying extension, uniqueness, and universal continuity, rather than all of Hypothesis \ref{global-short-limit-dnf-hypotheses}.}
\end{lemma}

\begin{proof}
    Let $\delta \in \underline{\kappa}(\dnf, \K_\mu, \lea^{u})$. We may assume $\delta$ is regular, as whether $\delta \in \underline{\kappa}(\dnf, \K_\lambda, \lea^{u})$ is determined by $\cof(\delta)$. Suppose $\langle M_i : i \leq \delta \rangle$ is a $\lea^u$-increasing continuous chain in $\K_\lambda$ and $p \in \gS(M_\delta)$.

    Suppose for contradiction that $p$ $\dnf$- forks over $M_i$ for all $i < \delta$. First we show we can assume without loss of generality that for each $i< \delta$, $M_{i+1}$ is a $(\lambda, \mu^+)$-limit over $M_i$. Taking $M_i' = M_i$ for $i = 0$ or limit, and $M'_{i+1}$ to be a $(\lambda, \mu^+)$-limit over $M_i$ with $M_{i+1}' \lea M_{i+1}$ (using universality of $M_{i+1}$ over $M_i$), we see the sequence $\langle M'_i : i < \delta \rangle$ satisfies all the same conditions as $M_i$ - that is, it is continuous, $M'_\delta = M_\delta$, and $p$ $\dnf$-forks over $M_i'$ for $i < \delta$ by base monotonicity. Hence (replacing $M_i$ with $M_i'$) we may assume for $i < \delta$ that $M_{i+1}$ is $(\lambda, \mu^+)$-limit over $M_i$.

    Now we show that without loss of generality this sequence witnesses failure of universal weak local character (this is similar to \cite[11(1)]{bgvv16}), i.e., that  $p \upharpoonright M_{i+1}$ $\dnf$-forks over $M_i$ for all $i < \delta$.
    
      Construct a continuous $<$-increasing sequence of ordinals $\langle i_j : j \leq \delta \rangle$ such that $i_j < \delta$ for $j < \delta$, and $p \upharpoonright M_{i_{j+1}}$ $\dnf$-forks over $M_{i_j}$. To do this, set $i_0 = 0$, at limits take $i_j = \sup_{j' < j} i_{j'}$ (possible by regularity of $\delta$), and for successors, since $p$ $\dnf$-forks over $M_{i_j+1}$, by applying universal continuity to $p$ and the sequence $\langle M_i : i > i_j +1 \rangle$, there is $i_{j+1} > i_j + 1$ such that $p \upharpoonright M_{i_{j+1}}$ $\dnf$-forks over $M_{i_j+1}$. Setting $M_0'' = M_0$, $M_j'' = M_{i_j}$ for $j< \delta$ limit, and $M_j'' = M_{i_j + 1}$ for $j < \delta$ successor, we have that $\langle M_j'' : j < \delta \rangle$ is $\lea^u$-increasing continuous with $p \upharpoonright M_{j+1}''$ $\dnf$-forking over $M_j''$, and $M_{j+1}''$ is $(\lambda, \mu^+)$-limit over $M_j''$ for all $j < \delta$ (this is why we used $i_j + 1$ rather than $i_j$, which may not be a successor). Thus without loss of generality we now have (replacing $M_i$ by $M_i''$) that $p \upharpoonright M_{i+1}$ $\dnf$-forks over $M_i$ for all $i < \delta$.
    
    We now construct $\langle N_i : i \leq \delta \rangle$, a $\lea^u$-increasing continuous sequence in $\K_\mu$, such that for all $i < \delta$,

    \begin{enumerate}
        \item $N_i \lea M_i$
        \item $p \upharpoonright N_{i+1}$ $\dnf$-forks over $N_i$.
    \end{enumerate}
    
    \textbf{This is possible:} We may take $N_0$ to be any $N_0 \lea M_0$ in $\K_\mu$, which satisfies requirement (1). At limit $i$, take unions (which preserves requirements (1) and (2)). 
    
    For successors, given $N_i$, we know $p \upharpoonright M_{i+1}$ $\dnf$-forks over $N_i$ by base monotonicity. By the $\theta$-witness property, there is $A \subseteq |M_{i+1}|$ where $|A|\leq \theta$ such that for all $N \lea M_{i+1}$ with $A \cup |N_i| \subseteq N$, $p \upharpoonright N$ $\dnf$-forks over $N_i$. Let $\langle M_i^j : j < \mu^+ \rangle$ witness that $M_{i+1}$ is $(\lambda, \mu^+)$-limit over $M_i$. As $|A| \leq \theta \leq \mu$ and $\| N_i \| \leq \mu$, there is some $j < \mu^+$ such that $A \cup |N_i| \subseteq |M_i^j|$. As $M_i^{j+1}$ is universal over $M_i^j$, there is $N_{i+1}  \in \K_\mu$ where $N_i \lea^u N_{i+1}  \lea M_{i+1}$ and $A \subseteq N_{i+1}$. By our choice of $A$, we know $p \upharpoonright N_{i+1}$ $\dnf$-forks over $N_i$, so this $N_{i+1}$ is as required. This completes the construction.

    \textbf{This is enough:} This sequence $\langle N_i : i \leq \delta \rangle$ along with $p \upharpoonright N_\delta$ shows that $\delta \notin {\underline{\kappa}^{\operatorname{wk}}(\dnf, \K_\mu, \lea^{u})}$. But this contradicts Remark \ref{weak-loc-char-iff-strong}.
\end{proof}

\begin{remark}\label{t-cont}
A couple of results in this section will use  results of \cite{vaseyt} which assume that  splitting has universal continuity. When we use those results we will assume that the AEC is $(<\aleph_0)$-tame. As universal continuity of splitting follows from  $(<\aleph_0)$-tameness (by \cite[2.5(1)]{maya24} and \cite[3.2]{mvy}) we can use those results.

In those results, instead of assuming   $(<\aleph_0)$-tameness, one could assume $\aleph_0$-tameness and work in  $\K_{\geq \LS(\K)^+}$ instead of in $\K$ (see Remark \ref{aleph0-tameness-is-enough} for more details).
\end{remark}

The notation of the following two definitions is similar to that of \cite[4.6]{vaseyt}.

\begin{defin}
Assume $\dnf$ is an independence relation.
    \begin{enumerate}
        \item $\stab(\K) = \{\lambda \geq \LS(\K) : \K \text{ is stable in } \lambda \}$
        \item $\underline{\chi}(\dnf, \K \lea^{u}) = \bigcup_{\lambda \in \operatorname{Stab}(\K)} \underline{\kappa}(\dnf, \K_\lambda, \lea^{u}) $ 
        \item  $\chi(\dnf, \K, \lea^{u}) = \operatorname{min} (\underline{\chi}(\dnf, \K \lea^{u}) \cap \text{Reg})$ if it exists or $\infty$ otherwise.
    \end{enumerate}
\end{defin}

\begin{remark}\label{chi-dif-def}
    Our definitions differ slightly in character from \cite[4.6]{vaseyt}. In particular $\chi(\dnf_{\splt}, \K, \lea^{u})$ is defined differently to $\chi(\K)$ from \cite[4.6]{vaseyt}, but they are the same assuming Hypothesis \ref{global-short-limit-dnf-hypotheses} and $(<\aleph_0)$-tameness  by \cite[2.9, 4.5]{vaseyt}.
\end{remark}

\begin{cor}\label{equal-chi}
    If $\K$ is stable in $\lambda \geq \LS(\K)^+$, then $\underline{\kappa}(\dnf, \K_\lambda, \lea^u) = \underline{\kappa}(\dnf_{\splt}, \K_\lambda, \lea^u)$ and  $\kappa(\dnf, \K_\lambda, \lea^u) = \kappa(\dnf_{\splt}, \K_\lambda, \lea^u)$ . In particular, if $\K$ is $(<\aleph_0)$-tame then $\chi(\dnf, \K, \lea^{u})= \chi(\dnf_{\splt}, \K, \lea^{u})$.
\end{cor}

\begin{proof}
The first part follows from Corollary \ref{equal-kappa}. The \emph{in particular} follows from the result and \cite[4.5]{vaseyt}, which says that $\kappa(\dnf_{\splt}, \K_\lambda, \lea^u)$ is weakly decreasing with $\lambda$, and Lemma \ref{kappa-increases} (these prevent the values when $\lambda = \LS(\K)$ from being smallest).
\end{proof}

\begin{defin}
    Let $\theta(\K, \dnf)$ be the least stability cardinal $\theta \geq \LS(\K)$ such that ${\kappa(\dnf, \K_\mu, \lea^{u})}= \kappa(\dnf, \K_{\theta}, \lea^{u})$ for every stability cardinal $\mu \geq \theta$, or $\infty$ if no such cardinal exists.
\end{defin}

The following follows from Fact \ref{tame-gives-wp} and Lemma \ref{kappa-increases}.

\begin{cor}\label{theta-chi}
    If $\K$ is $\aleph_0$-tame, then $\theta(\K, \dnf) < \infty$ and ${\chi(\dnf, \K, \lea^{u})} = \kappa(\dnf, \K_{\theta(\K, \dnf)}, \lea^{u})$.
\end{cor}

The following result is a global version of Theorem \ref{full-picture}. As this is one of our main black box results, we mention the hypotheses.

\begin{theorem}\label{cor-2}
    Assume Hypothesis \ref{global-short-limit-dnf-hypotheses} holds for an AEC $\K$ and independence relation $\dnf$. Suppose $\K$ is  $\aleph_0$-tame and stable in $\lambda \geq  \theta(\dnf, \K)$. 

    Suppose $\delta_1, \delta_2 < \lambda^+$ with $\cof(\delta_1) < \cof(\delta_2)$. Then for any $N_1, N_2, M \in \K_\lambda$ where $N_l$ is a $(\lambda, \delta_l)$-limit model over $M$ for $l = 1, 2$, 
    \[N_1 \text{ is isomorphic to } N_2 \text{ over } M  \iff \cof(\delta_1) \geq \chidnfu.\]
    
    Moreover, for any $N_1, N_2, \in \K_\lambda$ where $N_l$ is a $(\lambda, \delta_l)$-limit model for $l = 1, 2$, 
    \[N_1 \text{ is isomorphic to } N_2 \iff \cof(\delta_1) \geq \chidnfu.\]
\end{theorem} 
\begin{proof}
As $\lambda \geq \theta(\dnf, \K)$, it follows that $\chi(\dnf, \K, \lea^{u}) = \kappa(\dnf, \K_{\lambda}, \lea^{u})$. The result then follows from Theorem \ref{full-picture}.
\end{proof}

\begin{remark}
The value of Theorem \ref{cor-2} is that  we only need to compute $\chidnfu$ and $\theta(\K, \dnf)$  to understand the $\lambda$-limit models for large enough $\lambda$.
\end{remark}

\begin{remark}
Observe that in Theorem \ref{cor-2}, assuming $\K$ is $(<\aleph_0)$-tame, one can substitute every occurrence of  ${\chi(\dnf, \K, \lea^{u})}$ with $\chi(\dnf_{\splt}, \K, \lea^{u})$ by Corollary \ref{equal-chi}.
\end{remark}

\subsection{Toward applications} If one wants to simply apply Theorem \ref{cor-2} as a black box to understand limit models in a natural AEC, it still has the shortfall that one needs to calculate both $\theta(\dnf, \K)$ and $\chidnfu$. In this subsection we focus on simplifying these calculations.

We focus first on $\theta(\dnf, \K)$. 

\begin{defin}[{\cite[4.9]{vaseyt}}]
    $\lambda'(\K)$ is  the least stability cardinal $\lambda' \geq \LS(\K)$ such that $\kappa(\dnf_{\splt}, \K_\mu, \lea^u) = {\kappa(\dnf_{\splt}, \K_{\lambda'}, \lea^u)}$  for every stability cardinal $\mu \geq \lambda'$, or $\infty$ if no such cardinal exists.
\end{defin}

\begin{remark}
    This is a little different from how $\lambda'(\K)$ is defined in \cite[4.9]{vasey18}, to accommodate for us not including the regular cardinals $\geq \lambda^+$ in our definition of $\Kkappalims$, but it is equivalent under Hypothesis \ref{global-short-limit-dnf-hypotheses} and $\aleph_0$-tameness by \cite[4.5]{vasey18}.
\end{remark}

\begin{lemma}\label{bound-theta} 
If $\K$ is $(<\aleph_0)$-tame, then $\theta(\K, \dnf) < \beth_{\beth_{(2^{LS(\K)})^+}}$.
\end{lemma}
\begin{proof} We show that $\theta(\K, \dnf) \leq \lambda'(\K)^+$. This is enough as $\lambda'(\K) < \beth_{\beth_{(2^{LS(\K)})^+}}$ by \cite[11.3]{vaseyt}. Observe that we can use \cite[\S 11]{vaseyt} by Remark \ref{t-cont}.

Suppose $\lambda'(\K) \geq \LS(\K)^+$, then for every $\mu \geq \lambda'(\K)$, $\kappa(\dnf_{\splt}, \K_\mu, \lea^u) = {\kappa(\dnf_{\splt}, \K_{\lambda'}, \lea^u)}$ by  definition of $\lambda'(\K)$. Hence  $\kappa(\dnf, \K_\mu, \lea^u) = {\kappa(\dnf, \K_{\lambda'}, \lea^u)}$ by Corollary \ref{equal-chi}  for every $\mu \geq \lambda'(\K)$. Hence $\theta(\dnf, \K) \leq \lambda'(\K)$.

If $\lambda'(\K) = \LS(\K)$, then $\lambda'(\K)^+ = \LS(\K)^+$, so the argument of the previous paragraph shows that $\kappa(\dnf, \K_\mu, \lea^{u}) = \kappa(\dnf, \K_{\lambda'(\K)^+}, \lea^{u})$ for all $\mu \geq \lambda'(\K)^+$. Hence $\theta(\dnf, \K) \leq \lambda'(\K)^+$.\end{proof}

\begin{remark}
    In fact, the $\lambda'(\K) \geq \LS(\K)^+$ case in the previous lemma shows that in this case, $\lambda'(\K) = \theta(\dnf, \K)$, since $\kappa(\dnf_{\splt}, \K_\mu, \lea^u) = {\kappa(\dnf_{\splt}, \K_{\lambda'}, \lea^u)}$ for all $\lambda \geq \LS(\K)^+$ and by definition both $\lambda'(\K)$ and $\theta(\dnf, \K)$ are the $\lambda$ at which this stabilises.
\end{remark}

\begin{probl}
Find a better upper bound for $\theta(\K, \dnf)$ or show that the result of Lemma \ref{bound-theta} is sharp. 
\end{probl}

We now focus on $\chidnfu$. 

\begin{lemma}\label{calculate-chi} 
Suppose $\K$ is $(< \aleph_0)$-tame.
Assume $\rho$ is an infinite cardinal and consider the following two statements:
\begin{enumerate}[(a)]
\item For every $\lambda \geq 2^{LS(\K)}$, if $\K$ is stable in $\lambda$, then $\lambda^{<\rho} = \lambda$
\item For every $\lambda \geq 2^{LS(\K)}$, if $\lambda^{<\rho} = \lambda$, then $\K$ is stable in $\lambda$. 
\end{enumerate}

The following hold:

\begin{enumerate}
\item If Statement $(a)$ holds, then $\rho \leq \chidnfu$
\item If Statement $(b)$ holds and $\rho$ is regular , then $\chidnfu \leq \rho$
\item If Statements  $(a)$ and $(b)$ hold and $\rho$ is regular, then $\chidnfu = \rho$
\item If Statement $(b)$ holds and $\rho$ is singular, then  $\chidnfu \leq \rho^+$
\item  If Statements  $(a)$ and $(b)$ hold and $\rho$ is singular, then $\chidnfu = \rho^+$.
\end{enumerate}
\end{lemma}
\begin{proof}\
\begin{enumerate}
\item This is essentially \cite[11.2]{vaseyt} with the observation that $\chidnfu = {\chi(\dnf_{\splt}, \K, \lea^{u})} = \chi(\K)$ by  Corollary \ref{equal-chi} and Remark \ref{chi-dif-def}.
\item Let $\lambda = \beth_{\rho + \rho} (\theta)$ where  $\rho + \rho$ is the ordinal given by ordinal arithmetic, $\theta = \theta(\K, \dnf) + \chi_0$, and $\chi_0$ is the cardinal given in \cite[4.10]{vaseyt}. The precise value of $\chi_0$ is not important, we only need to notice that $\lambda > \chi_0$.

 Using that $\rho$ is regular and the definition of $\lambda$, it follows that $\lambda^\mu = \lambda$ for all cardinals $\mu < \rho$, so $\lambda^{<\rho} = \lambda$. Hence $\K$ is stable in $\lambda$ by Statement (b). 

We show that $\K$ is stable in unboundably many cardinals below $\lambda$. Let  $\mu < \lambda$, then there is an ordinal $\alpha \geq \rho$ such that $\mu \leq \beth_{\alpha + 1}(\theta)$. Observe that

 $$\beth_{\alpha +1}(\theta)^{<\rho} \leq \beth_{\alpha +1}(\theta)^\rho = 2^{\beth_\alpha(\theta) \rho} = 2^{\beth_\alpha(\theta)} = \beth_{\alpha +1}(\theta).$$ Hence $ \beth_{\alpha +1}(\theta)^{<\rho}=\beth_{\alpha +1}(\theta)$ and $\K$ is stable in $\beth_{\alpha +1}(\theta)< \lambda$ by Statement (b). 

Therefore, $\cof(\lambda) = \rho \in \underline{\kappa}(\dnf_{\splt}, \K_\lambda, \lea^u)$ by \cite[4.11]{vaseyt}. Hence $\rho \in  \underline{\kappa}(\dnf, \K_\lambda, \lea^u)$ by Corollary \ref{equal-chi}. Then $\kappa(\dnf, \K_\lambda, \lea^u) \leq \rho$ by minimality of $\kappa$. Hence $\chi(\dnf, \K, \lea^u) \leq \rho$ as $\lambda \geq \theta(\K, \dnf)$ and Corollary \ref{theta-chi}. 

\item Follows directly from $(1)$ and $(2)$.

\item The argument given in $(2)$ works if one changes $\rho$ by $\rho^+$ as if $\lambda^{<\rho^+} = \lambda$ then $\lambda^{<\rho} = \lambda$. 

\item Follows directly from $(1)$ and $(4)$. 

\end{enumerate}
\end{proof}

\begin{remark}
The value of Theorem \ref{calculate-chi} is that  it allows us to compute $\chidnfu$ directly from the stability spectrum of the AEC.
\end{remark}

We present a revised version of Theorem \ref{cor-2} which serves better as a black box in applications (see next section). However in the case where $\theta(\K, \dnf)$ can be computed, Theorem \ref{cor-2} gives a better bound (see Lemma \ref{first-order}). As before, we emphasise the hypotheses.

\begin{theorem}\label{cor-2.5}
    Assume Hypothesis \ref{global-short-limit-dnf-hypotheses} holds for an AEC $\K$ and independence relation $\dnf$. Suppose $\K$ is  $(<\aleph_0)$-tame and stable in $\lambda \geq   \beth_{\beth_{(2^{LS(\K)})^+}}$.

    Suppose $\delta_1, \delta_2 < \lambda^+$ with $\cof(\delta_1) < \cof(\delta_2)$. Then for any $N_1, N_2, M \in \K_\lambda$ where $N_l$ is a $(\lambda, \delta_l)$-limit model over $M$ for $l = 1, 2$, 
    \[N_1 \text{ is isomorphic to } N_2 \text{ over } M  \iff \cof(\delta_1) \geq \chidnfu.\]
    
    Moreover, for any $N_1, N_2, \in \K_\lambda$ where $N_l$ is a $(\lambda, \delta_l)$-limit model for $l = 1, 2$, 
    \[N_1 \text{ is isomorphic to } N_2 \iff \cof(\delta_1) \geq \chidnfu.\]
\end{theorem} 
\begin{proof}
Since $ \theta(\dnf, \K) < \beth_{\beth_{(2^{LS(\K)})^+}} $ by Lemma \ref{bound-theta}, the result follows from Theorem \ref{cor-2}.
\end{proof}

\begin{remark}
    This removes the need to calculate $\theta(\dnf, \K)$, but restricts which $\lambda$ the result applies to. One can calculate $\chidnfu$ using Theorem \ref{calculate-chi}.
\end{remark}

\begin{remark}\label{aleph0-tameness-is-enough}
    In fact, we can replace $(<\aleph_0)$-tameness with $\aleph_0$-tameness if we increase the lower bound for $\lambda$. Note we only needed the stronger form of tameness to guarantee universal continuity of splitting in $\K_\lambda$. But for $\lambda \geq \LS(\K)^+$, universal continuity of $\lambda$-non-splitting in $\K_\lambda$ follows from Lemma \ref{split-fork}, universal continuity of $\dnf$, Fact \ref{dnf_implies_dns} and weak transitivity of $\lambda$-non-splitting \cite[3.7]{vasey16b}. Note we need $\lambda \geq \LS(\K)^+$ as Lemma \ref{split-fork} relies on $(\geq \kappa)$-local character of $\dnf$ for some regular $\kappa < \lambda^+$.

    Then by the same reasoning as before, but working in $\K_{\geq \LS(\K)^+}$ rather than in $\K$, we deduce the statement of Theorem \ref{cor-2.5} but with $(<\aleph_0)$-tameness, $\beth_{\beth_{(2^{LS(\K)})^+}}$, and $\chi(\dnf, \K, \lea^u)$ replaced by $\aleph_0$-tameness, $\beth_{\beth_{(2^{(LS(\K)^+)})^+}}$, and $\chi(\dnf, \K_{\geq \LS(\K)^+}, \lea^u)$ respectively. Note that under the hypotheses of Theorem \ref{cor-2.5}, $\chi(\dnf, \K, \lea^u) = \chi(\dnf, \K_{\geq \LS(\K)^+}, \lea^u)$.
\end{remark}

\section{Some applications}\label{applications}

The objective of this section is to present some applications of the results obtained in this paper. In particular, we will show how the results of the previous section can be used to understand the limit models of some natural AECs.

\subsection{Checking the assumptions} Although Hypothesis \ref{global-short-limit-dnf-hypotheses} is a natural \emph{theoretical} assumption, in applications independence relations in the sense of LRV \cite{lrv1} (see  Example \ref{dnf-examples-lrv}) arise more naturally.

The following Lemma, which relies on the notions introduced in Example \ref{dnf-examples-lrv} and follows almost immediately from Lemma \ref{lrv-dnf-property-lemma}, shows when Hypothesis \ref{global-short-limit-dnf-hypotheses} follows from the existence of a stable independence relation in the sense of LRV. We include the result because we expect this result to be useful when studying limit models in natural abstract elementary classes. We showcase how to use this result later in this section.

\begin{lemma}\label{check-st}

Suppose $\K$ is an AEC with JEP, NMM, and $\K$ has a weakly stable independence relation in the sense of LRV \cite{lrv1} with strong $ \LS(\K)$-local character. Then $\K$ satisfies Hypothesis \ref{global-short-limit-dnf-hypotheses} with the possible exception of universal continuity.

Moreover, if $\K$ is $(<\aleph_0)$-tame, then $\K$ satisfies Hypothesis \ref{global-short-limit-dnf-hypotheses}. 

Furthermore,  if $\K$ is $(<\aleph_0)$-tame, for each stability cardinal $\lambda \geq \LS(\K)^+$, $\K$ satisfies the hypothesis of Theorem \ref{full-picture}.
\end{lemma}

\begin{proof}
   Observe that if $\lambda^{LS(\K)}= \lambda$, then $\K$ is stable in $\lambda$ (see for example \cite[8.15]{lrv1} for an idea of the proof). The rest of the main statement all follows from Lemma \ref{lrv-dnf-property-lemma}.

   For the moreover part,  note that by $(<\aleph_0)$-tameness and Fact \ref{tame-gives-wp}, $\dnf$ has the $(<\aleph_0)$-witnessing property, and therefore universal continuity (the only property missing from the first part)  by Lemma \ref{wp-implies-universal_continuity}.
   
   For the furthermore part, observe that the hypothesis of Theorem \ref{full-picture} are a local analogue of Hypothesis \ref{global-short-limit-dnf-hypotheses} so the result follows from the previous two parts.
\end{proof}

\subsection{Modules with embeddings}  We showcase  how our results can be used to understand the limit models in the abstract elementary class of modules with embeddings. Most of the results we present in this subsection were originally obtained in \cite{maz24} using algebraic methods and were one of the key motivations to write this paper. 

Denote by $\K^{R\text{-Mod}}$ the abstract elementary class of modules with embeddings.  Observe that  the language of this AEC is $\{ +, - , 0\} \cup \{r \cdot : r \in R\}$ where $r \cdot$ is a unary function for every $r \in R$ and that $LS(\Km)= \operatorname{card}(R) + \aleph_0$. 

\begin{lemma}\label{check-hyp}
$\Km$ satisfies Hypothesis \ref{global-short-limit-dnf-hypotheses} and is $(<\aleph_0)$-tame.
\end{lemma}
\begin{proof}
$\Km$ has  joint embedding and no maximal models by for example \cite[3.1]{m4}. $\Km$ has a weakly stable independence relation  in the sense of LRV with strong $(\operatorname{card}(R) + \aleph_0)$-local character by \cite[3.9, 3.10, 3.11]{mj}. As $\Km$ is $(< \aleph_0)$-tame by \cite[3.3]{m4}, it follows that $\Km$ satisfies Hypothesis \ref{global-short-limit-dnf-hypotheses} by Lemma \ref{check-st}. \end{proof}

\begin{defin}[{\cite{ekl}}]
    Let $R$ be a ring. Define $\gamma(R)$ as the minimum infinite cardinal such that every left ideal of $R$ is $(< \gamma(R))$-generated. Let $\gamma_r(R) = \gamma(R)$ if $\gamma(R)$ is regular and $\gamma(R)^+$ if $\gamma(R)$ is singular.
\end{defin}

Observe that for every ring, $\gamma(R) \leq \operatorname{card}(R) + \aleph_0$. 

\begin{fact}[{\cite[3.8]{m2}}]\label{st-no} Assume $\lambda \geq (\operatorname{card}(R) + \aleph_0)^+$.
  $\K^{R\text{-Mod}}$ is stable in $\lambda$ if and only if $\lambda^{<\gamma(R)}=\lambda$. 
\end{fact}

The following result is new.

\begin{lemma}\label{chi-mod}
$\chidnfum = \gamma_r(R)$.
\end{lemma}
\begin{proof}
$\Km$ satisfies the assumptions of Lemma \ref{calculate-chi} with $\rho = \gamma(R)$  by Lemma \ref{check-hyp} and Fact \ref{st-no}. The result then follows from  Lemma \ref{calculate-chi}.
\end{proof}

\begin{lemma}\label{lim-mod}
Let $\lambda \geq \beth_{\beth_{(2^{\operatorname{card}(R) + \aleph_0})^+}}$ such that $\Km$ is stable in $\lambda$. 

Suppose $\delta_1, \delta_2 < \lambda^+$ with $\cof(\delta_1) < \cof(\delta_2)$. Then for any $N_1, N_2\in \K_\lambda$ (and $M \in \K_\lambda$)  where $N_l$ is a $(\lambda, \delta_l)$-limit model (over $M$) for $l = 1, 2$, 
    \[N_1 \text{ is isomorphic to } N_2 \text{ (over } M ) \iff \cof(\delta_1) \geq \gamma_r(R)\]

\end{lemma}
\begin{proof}
The result follows directly from Lemma \ref{chi-mod},  Lemma \ref{check-hyp} and Theorem \ref{cor-2.5}.
\end{proof}

\begin{remark}

Lemma \ref{lim-mod} is slightly weaker than the original result \cite[3.22]{maz24} as the original result has $\lambda \geq \gamma(R)^+$ instead of  $\lambda \geq \beth_{\beth_{(2^{\operatorname{card}(R) + \aleph_0})^+}}$. It is possible that Lemma \ref{lim-mod} could be improved to obtain the bound  $\gamma(R)^+$ using Theorem \ref{cor-2}, but in order to do that one would need to show that $\theta(\Km, \dnf) = (\operatorname{card}(R) + \aleph_0)^+$. 

In fact, by combining Lemma \ref{lim-mod} with the original result \cite[3.22]{maz24}, we can actually calculate $\theta(\Km, \dnf)$ (see Remark \ref{theta}).
\end{remark}

Recall that for an integer $n \geq 0$, we say a ring $R$ is $(<\aleph_{n } )$-noetherian if $\gamma(R) \leq \aleph_n$.  

\begin{theorem}\label{char-noe}
Let $n \geq 0$ be an integer. The following are equivalent.
\begin{enumerate}
\item $R$ is left $(<\aleph_{n } )$-noetherian but not left $(< \aleph_{n -1 })$-noetherian\footnote{If $n=0$ this should be understood as  $R$ is left $(<\aleph_{0 } )$-noetherian, i.e., $R$ is left noetherian. }
\item $\Km$ has exactly $n +1 $ non-isomorphic $\lambda$-limit models for every $\lambda \geq \beth_{\beth_{(2^{\operatorname{card}(R) + \aleph_0})^+}}$ such that $\Km$ is stable in $\lambda$. 
\end{enumerate}
\end{theorem}
\begin{proof}
$(1) \Rightarrow (2)$: It follows from the assumption on the ring that  $\gamma(R) = \aleph_n$. Then  the $(\lambda, \aleph_{s})$-limit models for $0 \leq s \leq n$ are all the $\lambda$-limit models up to isomorphisms by Lemma \ref{lim-mod}. Hence there are exactly $n +1$ non-isomorphic $\lambda$-limit models.

$(2) \Rightarrow (1)$: Assume for the sake of contraction that (1) fails. Then either $\gamma(R) > \aleph_n$ or $\gamma(R) \leq  \aleph_{n-1}$. Doing a similar argument to that of (1) $\Rightarrow$ (2) one can show that there are either $n+2$ non-isomorphic $\lambda$-limit models or less than $n$  non-isomorphic $\lambda$-limit models for $\lambda = \beth_{\beth_{(2^{\operatorname{card}(R) + \aleph_0})^+}}$. This is clearly a contradiction. 
\end{proof}

\begin{remark}
Theorem \ref{char-noe} is also slightly weaker than the original result \cite[3.17]{maz24} as the original result has $\lambda \geq (\operatorname{card}(R) + \aleph_0)^+$ instead of  $\lambda \geq \beth_{\beth_{(2^{\operatorname{card}(R) + \aleph_0})^+}}$.
\end{remark}

\begin{remark}\label{theta}
Our results can also be used to find $\theta(\Km, \dnf)$ in a roundabout way. Let $\theta_0$ be the least stability cardinal with $\theta_0 \geq (\operatorname{card}(R) + \aleph_0)^+$. Observe that $\theta_0 \leq 2^ {\operatorname{card}(R) + \aleph_0}$ by Fact \ref{st-no}.

Theorem \ref{full-picture} and \cite[3.22]{maz24} give the `cut offs' of where limit models go from non-isomorphic to isomorphic as $\kappa(\dnf, \Km_{\theta_0}, \lea^u)$ and $\gamma_r(R)$ respectively. That is, for regular $\mu_1 < \mu_2 < \theta_0^+$, the $(\theta_0, \mu_1)$-limit model and $(\theta_0, \mu_2)$-limit model are isomorphic if and only if $\mu_1 \geq {\kappa(\dnf, \Km_{\theta_0}, \lea^u)}$, again if and only if $\mu_1 \geq \gamma_r(R)$. Therefore $\kappa(\dnf, \Km_{\theta_0}, \lea^u) = \gamma_r(R)$. As ${\chi(\dnf, \Km, \lea^u)}= \gamma_r(R)$ by Fact \ref{chi-mod}, we must have $\theta(\Km, \dnf) \leq \theta_0$.

So either $\theta(\Km, \dnf) = \theta_0$, or $\theta(\Km, \dnf) = \LS(\Km) = \operatorname{card}(R) + \aleph_0$. 
    
\end{remark}

\subsection{Beyond modules with embeddings}\label{beyond-modules}

We briefly present two additional applications of our results.

Given a cpmplete first order theory $T$, let $\kappa(T)$ be Shelah's local character cardinal from \cite[Definition III.3.1]{shbook} - that is, $\kappa(T)$ is the least $\kappa$ such that for all $\langle A_i : i < \kappa\rangle$ $\subseteq$-increasing sequences of sets, and all $p \in \gS(A_\kappa)$, there exists $i < \kappa$ such that $p$ does not fork over $A_i$. Let $\kappa_r(T)$ be the least regular $\kappa \geq \kappa(T)$. Let $\dnf_{f}$ denote first-order non-forking. 

\begin{lemma}\label{first-order} Let $T$ be a complete first-order theory and $\K^T$ the abstract elementary class of models of $T$ with elementary embeddings. 
\begin{enumerate}
\item $\kappa_r(T) = \chi(\dnf_{f}, \K^T, \preccurlyeq^u)$
\item If $\lambda \geq |T|$ is stable, then $\kappa_r(T) = \kappa(\dnf_{f}, \K^T_\lambda, \preccurlyeq^u)$
\item $\theta(\dnf_f, \K^T) = |T|$
\item For every stability cardinal $\lambda \geq |T|$, suppose $\delta_1, \delta_2 < \lambda^+$ with $\cof(\delta_1) < \cof(\delta_2)$. Then for any $N_1, N_2\in \K^T_\lambda$ (and $M \in \K^T_\lambda$)  where $N_l$ is a $(\lambda, \delta_l)$-limit model (over $M$) for $l = 1, 2$, 
    \[N_1 \text{ is isomorphic to } N_2 \text{ ( over } M ) \iff \cof(\delta_1) \geq \kappa_r(T)\]
    
\end{enumerate}
\end{lemma}
\begin{proof}
 Let $\mon$ be the monster model of $T$. First we show that $\K^T$ is $(<\aleph_0)$-tame and that $\K^T$ (with $\dnf_{f}$) satisfies Hypothesis \ref{global-short-limit-dnf-hypotheses}. $(<\aleph_0)$-tameness follows from the fact that first order types are determined by formulas over only finitely many elements of the model. $\dnf_{f}$ satisfies invariance, monotonicity, base monotonicity, uniqueness, extension, existence, and $(\geq|T|)$-local character by \cite[Chapter III]{shbook} (in particular, strong $(<|T|)$-local character follows from \cite[Corollary III.3.3]{shbook}). Universal continuity holds because of $(<\aleph_0)$-tameness, Fact \ref{tame-gives-wp}, and Lemma \ref{wp-implies-universal_continuity}. For non-forking amalgamation, suppose $M_0, M_1, M_2 \models T$ where $M_0 \preccurlyeq M_l$ and $a_l \in M_l$ for $l = 1, 2$. Let $p = \operatorname{tp}(M_1/M_0)$. By existence, $p$ does not fork over $M_0$, so by extension, there is $q \in \gS(M_2)$ such that $p \subseteq q$ and $q$ does not fork over $M_0$. Let $f_1:\mon \rightarrow \mon$ be an automorphism fixing $M_0$ where $f(M_1) \models q$, and $f_2: \mon \rightarrow \mon$ be the identity. Then $\operatorname{tp}(f_1[M_1]/f_2[M_2]) = q$ does not fork over $M_0$, and by symmetry of non-forking \cite[2.5]{kim98} $\operatorname{tp}(f_2[M_2]/f_1[M_1])$ does not fork over $M_0$. Therefore by monotonicity, $\operatorname{tp}(f_l(a_l)/f_{3-l}[M_{3-l}])$ does not fork over $M_0$ for $l = 1, 2$, and non-forking amalgamation holds.

\begin{enumerate}
\item Since $\K^T$ is stable in $\lambda \geq 2^{|T|}$ if and only if $\lambda^{< \kappa_r(T)} = \lambda$ by \cite[Corollary III.3.8]{shbook}. The result follows from Lemma \ref{calculate-chi} with $\rho =  \kappa_r(T)$. 
\item By their definitions, $\kappa(\dnf_{f}, \operatorname{Mod}(T)_\lambda, \preccurlyeq^u)$ and $\kappa_r(T)$ differ only in that the $A_i$ in the definition of $\kappa_r(T)$ are restricted to be $\lea^u$-increasing models for $\kappa(\dnf_{f}, \operatorname{Mod}(T)_\lambda, \preccurlyeq^u)$, so $\kappa_r(T) \geq \kappa(\dnf_{f}, \operatorname{Mod}(T)_\lambda, \preccurlyeq^u)$. By Lemma \ref{kappa-increases}, the definition of $\chi(\dnf_{f}, \operatorname{Mod}(T), \preccurlyeq^u)$, and (1), $\kappa_r(T) \leq {\kappa(\dnf_{f}, \operatorname{Mod}(T)_\lambda, \preccurlyeq^u)}$. So $\kappa_r(T) = {\kappa(\dnf_{f}, \operatorname{Mod}(T)_\lambda, \preccurlyeq^u)}$.
\item Follows immediately from (2).
\item Follows directly from Theorem \ref{cor-2}.
\end{enumerate}
\end{proof}

\begin{remark}
Lemma \ref{first-order}(1) was originally obtained in \cite[4.18]{vaseyt}.
\end{remark}

Now we move on to our final application. Let $G, H$ be abelian groups. Recall that $G$ is a pure subgroup of $H$ if divisibility is preserved between $G$ and $H$. 

\begin{lemma}\label{tor} Let $\K^{\text{Tor}}$ be the AEC of torsion abelian groups with pure embeddings.
\begin{enumerate}
\item $\chidnfut = \aleph_1$
\item  For every stability cardinal $\lambda \geq \beth_{\beth_{(2^{\aleph_0})^+}}$, there are exactly 2 non-isomorphic limit models. 
\end{enumerate}
\end{lemma}
\begin{proof}
First observe that $\Kto$ satisfies  Hypothesis \ref{global-short-limit-dnf-hypotheses} and is $(<\aleph_0)$-tame by \cite[4.1]{m-tor}, \cite[4.12, 4.14]{maz2} , \cite[4.3]{m-tor}, and Lemma \ref{check-st}.

\begin{enumerate}
\item Since $\Kto$ is stable in $\lambda$ if and only if $\lambda^{\aleph_0} = \lambda$ by \cite[5.5]{m4}. The result follows from Lemma \ref{calculate-chi} with $\rho = \aleph_1$. 
\item Follows directly from Theorem \ref{cor-2.5}.
\end{enumerate}
\end{proof}

\begin{remark}
Lemma \ref{tor} is a weakening of \cite[5.7]{m-tor}. Observe that for $\lambda < \beth_{\beth_{(2^{\aleph_0})^+}}$ we know that there at least $2$ $\lambda$-limit models by Lemma \ref{tor}(1), Lemma \ref{kappa-increases} and Theorem \ref{full-picture}.
\end{remark}


\end{document}